\documentclass[12pt]{article}
\usepackage{fullpage, amsmath, amssymb, amsfonts, amscd, amsxtra, mathrsfs, amsthm, stmaryrd, url, hyperref}
\usepackage[all]{xy}

\newtheorem{theorem}{Theorem}[subsection]
\newtheorem{lemma}[theorem]{Lemma}
\newtheorem{corollary}[theorem]{Corollary}

\newtheorem{question}[theorem]{Question}
\newtheorem{proposition}[theorem]{Proposition}
\theoremstyle{definition}
\newtheorem{definition}[theorem]{Definition}

\newtheorem{example}[theorem]{Example}
\newtheorem{remark}[theorem]{Remark}
\newtheorem{convention}[theorem]{Convention}

\numberwithin{equation}{theorem}

\def\Q{{\mathbb {Q}}_p}

\def\e{\mathrm{e}}

 \def\ra{\rightarrow}
\def\ZZ{{\mathbb Z}}

\def\OO{\mathcal{O}}

\def\int{\mathrm{int}}
\def\bd{\mathrm{bd}}

\def\at{\widetilde{\mathbf{A}}}
\def\bt{\widetilde{\mathbf{B}}}

\def\calE{\mathcal{E}}

\def\rig{\mathrm{rig}}

\DeclareMathOperator{\rank}{rank}

\def\m{(\varphi,\Gamma)}

\def\r{\mathcal{R}}

\begin{document}
\title{Slope filtrations in families}
\author{Ruochuan Liu\\ Department of Mathematics\\ University of Michigan\\ Ann Arbor, MI 48109\\ ruochuan@umich.edu}
\maketitle

\abstract
{This paper concerns arithmetic families of $\varphi$-modules over reduced affinoid spaces. For such a family, we first prove that the slope polygons is lower semicontinuous around any rigid point. If the slope polygons are locally constant around a rigid point, we further prove that around this point, the family has a global slope filtration after base change to some extended Robba ring.}

\tableofcontents
\section*{Introduction}
The slope filtrations for Frobenius modules over the Robba ring were originally introduced in the context of rigid cohomology by Kedlaya as the key ingredient of his proof of Crew's conjecture \cite{Ke04}. Roughly speaking, the slope filtrations give a partial analogue, for Frobenius-semilinear actions on finite free modules over the
Robba ring, of the eigenspace decompositions of linear transformations. It was discovered by Berger, through his construction of $\m$-modules associated to $p$-adic Galois representations, that the slope filtration theorem is also a fundamental ingredient for $p$-adic Hodge theory. For instance, it allowed Berger to prove Fontaine's conjecture that de Rham implies potentially
semistable and to give a new proof of the Colmez-Fontaine theorem that weakly admissible implies admissible. Recently, the work of Fontaine and Fargues \cite{FF10} has revealed more $p$-adic Hodge theoretic aspects of the slope filtration theorem. Namely, they reformulate it in terms of the Harder-Narasimhan filtrations for vector bundles over the \emph{fundamental curve of $p$-adic Hodge theory}.

This paper grows out of an attempt to generalize the slope filtration theorem to families of Frobenius modules with an eye towards applications to families of $p$-adic representations (i.e. relative $p$-adic Hodge theory). There are actually two distinct forms of ``families" of $p$-adic representations. One is continuous representations of absolute Galois groups of finite extensions of $\Q$ on finite locally free modules over affinoid algebras over $\Q$ such as the families of $p$-adic representations associated to $p$-adic families of automorphic forms; these are called \emph{arithmetic families}. Another one is continuous representations of \'etale fundamental groups of affinoid spaces over finite extensions of $\Q$ on finite dimensional $\Q$-vector spaces; these are called \emph{geometric families}. In \cite{BC07}, Berger and Colmez constructed a functor from arithmetic families of $p$-adic representations to families of $\m$-modules. For geometric families, the $\m$-module functor is constructed in \cite{KL11B}. It turns out that the $\m$-modules associated to these two types of families of $p$-adic representations have quite different features. Loosely speaking, the ``coefficients" for arithmetic families of $\m$-modules are of characteristic $0$, and $\varphi$ acts trivially on them whereas the ``coefficients" for geometric families of $\m$-modules are of characteristic $p$, and $\varphi$ acts on them as the $p$-th power Frobenius.

In this paper, inspired by Berger-Colmez's construction, we consider slope filtrations for arithmetic families of $\varphi$-modules. First of all, it is straightforward to see that for such families, a necessary condition to have global slope filtrations, at least locally around rigid points, is the local constancy of slope polygons. However, it is not difficult to see that this is not true in general (see \S2.4 for an example). Due to this fact, our first main result then concerns variations of slope polygons.
To state the result, we first introduce a few notations (see the body of the paper for more details). We fix a complete discretely valued field $K$ of mixed characteristic $(0,p)$ to be the base field of the Robba ring, and fix a relative Frobenius lift $\varphi$ on $\r_K$. Fix a reduced affinoid space $M(A)$ over $\Q$ to be the base for the families. Let $\r_{A_K}$ be the Robba ring over $A_k$, and set the $\varphi$-action on $\r_{A_K}$ as the continuous extension of  $\mathrm{id}\otimes\varphi$ on $A\otimes_{\Q}\r_K$. By a \emph{ family of $\varphi$-modules} over $\r_{A_K}$ we mean a vector bundle $M_A$ over $\r_{A_K}$ equipped with a semilinear $\varphi$-action such that the natural map $\varphi^*M_A\ra M_A$ is an isomorphism. For any $x\in M(A)$, we set $M_x$, the fiber of $M_A$ at $x$, as the base change of $M_A$ to $k(x)\otimes_{\Q}\r_K$.

\begin{theorem}\label{thm:semicontinuity-intro}(Theorem \ref{thm:semicontinuity})
Let $M_A$ be a family of $\varphi$-modules over $\r_{A_K}$. Then for any $x\in
M(A)$, there is a Weierstrass subdomain $M(B)$ containing $x$ such that the HN-polygon of $M_y$ lies above the HN-polygon of
$M_x$ with the same endpoint for
any $y\in M(B)$.
\end{theorem}

If $M_x$ is pure, the above theorem then implies that the fibers of $M_A$ are also pure of the same slope around $x$. In fact, a stronger result holds if $k(x)\subset A$. Namely, $M_A$ is \emph{globally pure} around $x$.
\begin{theorem} \label{thm:pure-family-intro}(Theorem \ref{thm:pure-family})
Let $M_A$ be a family of $\varphi$-modules over $\r_{A_K}$.
Suppose that $M_x$ is pure of slope $s$ for some $x\in M(A)$ with $k(x)\subset A$, then there
exists a Weierstrass subdomain $M(B)$ containing $x$ such that the base change of $M_A$ to $M(B)$ admits a finite free $(c,d)$-pure model $N_B$ where $d>0,(c,d)=1$ and $c/d=s$. In particular, $M_B$ is globally pure of slope $s$.
\end{theorem}

Although the slope polygons are not locally constant in general, we prove that one can shrink the Weierstrass subdomain $M(B)$ in Theorem \ref{thm:semicontinuity-intro} so that the set of $y\in M(B)$ where the slope polygon of $M_y$ coincides with the slope polygon of $M_x$ is a Zariski closed subset of $M(B)$. Furthermore, we have a global slope filtration on this Zariski closed subset after base change to some \emph{extended Robba ring}. This forms our second main result.
To state the result, fix an \emph{admissible extension} $L$ of $K$ so that its residue field is \emph{strongly difference-closed}, and let $\widetilde{\r}_L$ be the extended Robba ring over $L$.

\begin{theorem}\label{thm:global-filtration-intro-RF}(Theorem \ref{thm:global-filtration} for the $\mathrm{RF}$ case)
Let $M_A$ be a family of $\varphi$-modules over $\r_{A_K}$, and let $x\in M(A)$. Then there exists a Weierstrass subdomain $M(B)$ containing $x$ such that
the set of $y\in M(B)$ where the HN-polygon of $M_y$ coincides with the HN-polygon of $M_x$
forms a Zariski closed subset $M(C)$ of $M(B)$, and
 $$\widetilde{M}_{C}=M_A\otimes_{\r_{A_K}}(C\widehat{\otimes}_{\Q}\widetilde{\r}_L)$$
admits a unique slope filtration which lifts the HN-filtration of the $\varphi$-module
\[
\widetilde{M}_x=M_x\otimes_{k(x)\otimes_{\Q}\r_K}(k(x)\otimes_{\Q}\widetilde{\r}_L).
\]
\end{theorem}

In the case when $K$ is a finite unramified extension of $\Q$, and the $\varphi$-action on $\r_K$ is an absolute Frobenius lift (this is the case for the $\varphi$-modules arising from $p$-adic Hodge theory), we can use a canonical and smaller period ring $\bt_{\rig}^\dagger$ instead of $\widetilde{\r}_L$ in the statement of Theorem \ref{thm:global-filtration-intro-RF}. More precisely, we have the following theorem which we expect to be useful for $p$-adic Hodge theory.

\begin{theorem}\label{thm:global-filtration-intro-AF}(Theorem \ref{thm:global-filtration} for the $\mathrm{AF}$ case)
Suppose that $K$ is a finite unramified extension of $\Q$, and the $\varphi$-action on $\r_K$ is an absolute Frobenius lift. Let $M_A$ be a family of $\varphi$-modules over $\r_{A_K}$, and let $x\in M(A)$. Then there exists a Weierstrass subdomain $M(B)$ containing $x$ such that
the set of $y\in M(B)$ where the HN-polygon of $M_y$ coincides with the HN-polygon of $M_x$
forms a Zariski closed subset $M(C)$ of $M(B)$, and
$$\widetilde{M}_{C}=M_A\otimes_{\r_{A_K}}(C\widehat{\otimes}_{\Q}\bt_{\rig}^\dagger)$$
admits a unique slope filtration which lifts the HN-filtration of the $\varphi$-module
\[
\widetilde{M}_x=M_x\otimes_{k(x)\otimes_{\Q}\r_K}(k(x)\otimes_{\Q}\bt_{\rig}^\dagger).
\]
\end{theorem}

One can ask similar questions for Berkovich points rather than rigid points. However, since the residue field of a general Berkovich point is not necessarily discretely valued, this requires a slope theory for Frobenius modules over the Robba ring $\r_K$ for non-discretely valued $K$. By passing to the spherical completion of $K$, we may reduce to the case that $K$ is spherically complete. In this case, it is not difficult to show the existence of Harder-Narasimhan filtrations for Frobenius modules over $\r_K$ (Theorem \ref{thm:existence-HN-filtraion}). However, we can not prove the equivalence of semistability and purity which is the key of Kedlaya's original slope theory. Another issue is that the pure locus is not necessarily open (see \cite[Remark 7.5]{KL10} for more details) which prevents the semicontinuity of variation of slope polygons in the topology of Berkovich spaces. A possible solution for this issue is to use Huber's adic spaces instead of Berkovich spaces as shown in the work of Hellmann \cite{H11}.

We now sketch the structure of the paper. In \cite{Ke05}, the slope theory for absolute Frobenius was fully developed. A large part of this theory, especially the slope filtration theorem, was then generalized to relative Frobenius in \cite{Ke06}. In \S1, we further generalize some of the results of \cite{Ke05}, especially comparisons of various slope polygons, to the relative Frobenius case. In \S1.1, we give the definitions of various base rings. In \S1.2, we prove the existence of HN-filtrations for $\varphi$-modules over $\r_K$ for spherically complete  $K$, and review the slope filtration theorem for discretely valued $K$. In \S1.3, we generalize the classical Dieudonn\'e-Manin decomposition theorem to spherically complete difference fields that have strongly difference-closed residue fields. This result is irrelevant to the main results of this paper, but may be of independent interest. We define various extended base rings in \S1.4. In \S1.5, we review the slope theory for $\varphi$-modules over the extended Robba ring $\widetilde{\r}_K$, and we prove that the HN filtrations are split when $K$ has strongly difference-closed residue field. In \S1.6, when $K$ has strongly difference-closed residue field, using the slope decomposition for $\varphi$-modules over $\widetilde{\calE}_K$, we prove the existence of reverse filtrations for $\varphi$-modules over the extended bounded Robba ring $\widetilde{\r}_K^{\bd}$. In \S1.7, we prove that the generic HN-polygon lies above the special HN-polygon with the same endpoints.

We prove our main results in \S2. In \S2.1, we first establish some basic results for various base rings with coefficients in  certain Banach algebras. Then we introduce the definition of families of $\varphi$-modules. We prove Theorem \ref{thm:pure-family-intro} in \S2.2. In \S2.3, we first prove Theorem \ref{thm:semicontinuity-intro}. We then prove Theorem \ref{thm:global-filtration-intro-RF} and Theorem \ref{thm:global-filtration-intro-AF} in a uniform way. To do this, we introduce the notation $\widetilde{\r}$ which represents $\bt_{\rig}^\dagger$ in the case that $K$ is a finite unramified extension of $\Q$ and $\varphi$ is an absolute Frobenius lift (the $\mathrm{AF}$ case), and represents $\widetilde{\r}_L$ for general $K,\varphi$ (the $\mathrm{RF}$ case). In \S2.4, we construct a family of $\varphi$-modules where the HN-polygons are not locally constant over the base.

Finally, we mention that some techniques developed in this paper will be used in subsequent work on an improvement of Kisin's construction of finite slope subspaces (\cite{L11}).

\section*{Acknowledgements}
The author is grateful to Kiran S. Kedlaya for suggesting the topic of this paper, and also for his useful comments and suggestions. In this last regard, thanks are due as well to Liang Xiao. Thanks Eugen Hellmann for intriguing conversations. Thanks Jonathan Pottharst for the comments on early drafts of this paper. Thanks Institut de Math\'ematiques de Jussieu for their kind hospitality. The author would also like to thank the anonymous referee for a very careful reading of the paper and for many useful suggestions. When writing this paper, the author was partially funded by Kedlaya's NSF CAREER grant DMS-0545904.

\begin{convention}
Throughout this paper, let $K$ be a
complete nonarchimedean valued field of mixed characteristic $(0,p)$. Let $\OO_K$ be its valuation ring.
Let $\mathfrak{m}_K$ be the maximal ideal of $\OO_K$, and let $k=\OO_K/\mathfrak{m}_K$ be the residue field.  Let $v$ denote the valuation on $\widehat{\overline{K}}$ extending the one on $K$. Let $\pi\in\mathfrak{m}_K$ satisfying $v(\pi)=1$. From $\S1.4$ on, we further assume that $K$ is discretely valued and $\pi$ is a uniformizer. In $\S1$, we set the norm on $\widehat{\overline{K}}$ as $|\cdot|=p^{-v(\cdot)}$. In $\S2$, we further assume that $K$ is a $p$-adic field, and we renormalize the norm on $K$ so that $|p|=p^{-1}$ to fit the standard normalization on $\Q$.
\end{convention}
\section{Slope theory of $\varphi$-modules}

In this section we develop the slope theory for relative Frobenius lift. Beware that in $\S1.1$ and $\S1.2$, we do not assume that $K$ is discretely valued which makes things a bit subtler.

\subsection{The base rings}
\begin{definition}\label{def:Robba-ring}
For any interval
$I\subseteq(0,\infty]$, let $\r_K^{I}$ be the ring of Laurent series $f=\sum_{i\in\mathbb{Z}}a_i T^i$ for which $a_i\in K$
and $v(a_i)+si\ra\infty$ as $i\ra\pm\infty$ for all $s\in I$. Geometrically, $\r^I_K$ is the ring of $K$-holomorphic functions on the annulus $\{T\in \overline{K}|v(T)\in I\}$.
For any $s\in I$,
the valuation $w_s$ on $\r_K^I$ is defined as
$$w_s(f)=\min_{i\in\mathbb{Z}}\{v(a_i)+si\}.$$
The corresponding multiplicative norm is $|f|_s=\max_{i\in\mathbb{Z}}\{|a_i|p^{-is}\}$.
For $I=(0,r]$, we denote $\r_K^{(0,r]}$ by $\r_K^r$ for simplicity. We call the union $\r_K=\cup_{r>0}\r_K^r$ the \emph{Robba
ring} over $K$.
\end{definition}

\begin{definition}\label{def:bounded-Robba-ring}
For any $r>0$, let $\r_K^{\bd,r}$ be the subring of $\r_K^r$
consisting of Laurent series $f=\sum_{i\in\mathbb{Z}}a_i T^i$ with $\{v(a_i)\}_{i\in\mathbb{Z}}$ bounded below. Set
\[
w(f)=\inf_{i\in\mathbb{Z}}\{v(a_i)\},
\]
and set $|f|=\sup_{i\in\ZZ}\{|a_i|\}$. Let $\r_K^{\int,r}$ be the subring of $\r_K^r$
consisting of all $f$ with $w(f)\geq0$.  Let $\r_K^\bd=\cup_{r>0}\r_K^{\bd,r}$ and $\r_K^\int=\cup_{r>0}\r_K^{\int,r}$. We call $\r_K^\bd$ the \emph{bounded Robba ring} over $K$.
\end{definition}

\begin{proposition}\label{prop:w-bounded-additive}
For any $f\in\r_K^\bd$, we have $\lim_{r\ra0^+}w_r(f)=w(f)$. As a consequence, $w$ is additive and $|\cdot|$ is multiplicative on $\r_K^\bd$.
\end{proposition}
\begin{proof}
For any $\epsilon>0$, pick $i_0$ such that $v(a_{i_0})<w(f)+\frac{\epsilon}{2}$. Let $r_0=\frac{\epsilon}{|2i_0|+1}$, and we may suppose that $f\in\r_K^{\bd,r_0}$ by shrinking $\epsilon$. It thus follows that for any $r\in(0,r_0]$, $w_r(f)\leq ri_0+v(a_{i_0})<w(f)+\epsilon$. On the other hand, choose $N\in\mathbb{N}$ sufficiently large such that $r_0i+v(a_i)\geq w(f)$ for any $i\leq-N$. Therefore for any $r\leq r_1=\min\{r_0, \frac{\epsilon}{N}\}$, if $i\leq-N$, then $ri+v(a_i)\geq r_0i+v(a_i)\geq w(f)$; if $i\geq-N$, then $ri+v(a_i)\geq w(f)-rN\geq w(f)-\epsilon$. We thus deduce that $|w_r(f)-w(f)|\leq\epsilon$ for any $r\in(0,r_1]$, yielding the desired result.
\end{proof}

\begin{definition}\label{def:calE}
Let $\calE_K$ be ring of Laurent series $f=\sum_{i\in\mathbb{Z}}a_i T^i$ for which $\{v(a_i)\}_{i\in\mathbb{Z}}$ is bounded below and $v(a_i)\ra\infty$ as $i\ra-\infty$. Set
$w(f)=\inf_{i\in\mathbb{Z}}\{v(a_i)\}$,
and set $|f|=\sup_{i\in\ZZ}\{|a_i|\}$.  Let $\OO_{\calE_K}=\{f\in\calE_K|w(f)\geq0\}$.
\end{definition}
\begin{remark}
It is clear that $\r^{\bd,r}_K$ is a subring of $\calE_K$ consisting of the series such that $v(a_i)+ri\ra\infty$  as $i\ra -\infty$. In addition, the natural inclusion $\r_K^{\bd}\rightarrow\calE_K$ is an isometry with respect to $w$ and identifies $\calE_K$ with the $w$-completion of $\r_K^{\bd}$. In particular, $w$ is a valuation on $\calE_K$, and its corresponding multiplicative norm is $|\cdot|$.
\end{remark}

\begin{remark}\label{rem:fields}
If $K$ is discretely valued, both $\r_K^\bd$ and $\calE_K$ are discretely valued fields.
\end{remark}

\begin{definition}\label{def:topology}
For any interval $I\subseteq(0,\infty]$, we equip $\r_K^I$ with the Fr\'echet topology defined by $|\cdot|_s$ for all $s\in I$, and $\r_K^I$ is complete for this topology. If $I=[r_1,r_2]$ is a closed interval, $\r_K^I$ becomes a $K$-Banach algebra with the norm $\max\{|\cdot|_{r_1},|\cdot|_{r_2}\}$. We equip  $\r_K=\cup_{r>0}\r_K^r$ with the locally convex inductive limit topology  (in the sense of \cite[\S II.4]{bourbaki-evt}). In particular, a sequence converges in $\r_K$ if and only if it is a convergent sequence in
$\r_K^r$ for some $r>0$. For any $r>0$, we equip $\r_K^{\bd,r}$ with the norm $\max\{|\cdot|,|\cdot|_r\}$ under which it is a $K$-Banach algebra. The topology defined by this norm is the weakest topology so that the natural maps $\r_K^{\bd,r}\ra\r_K^r$ and $\r_K^{\bd,r}\ra\calE_K$ are continuous. We equip $\r_K^\bd=\cup_{r>0}\r_K^{\bd,r}$ with the locally convex inductive limit topology.
\end{definition}

\begin{proposition}\label{prop:units}
$\r_K^\times=(\r_K^{\bd})^\times$. In particular, if $K$ is discretely valued, the units of $\r_K$ are precisely the nonzero elements of $\r_K^\bd$.
\end{proposition}
\begin{proof}
Note that for any $f=\sum_{i\in\mathbb{Z}}a_iT^i\in\r_K^r$, $s\mapsto w_s(f)$ is a concave function on $(0,r]$. Suppose that $f$ is a unit in $\r_K^r$ with inverse $g$. It follows that the sum of two concave functions $w_s(f),w_s(g)$ is the constant function $0$. Thus both $w_s(f),w_s(g)$ are affine in $s$. Hence $v(a_i)=\lim_{s\ra0^+}(v(a_i)+si)\geq\lim_{s\ra0^+}w_s(f)$ for any $a_i$, yielding that $f\in\r_K^{\bd,r}$.
\end{proof}

\subsection{$\varphi$-modules over the Robba ring}

\begin{definition}\label{def:Frob-lift}
Fix an integer $q>1$. A \emph{relative (q-power) Frobenius lift} on
the Robba ring $\r_K$ is a homomorphism
$\varphi:\r_K\ra\r_K$ of
the form
\begin{center}
$\sum_{i\in\mathbb{Z}}a_i
T^i\mapsto\sum_{i\in\mathbb{Z}}\varphi_K(a_i) S^i,$
\end{center}
 where
$\varphi_K$ is an isometric endomorphism on $K$
and $S$ lies in $\r_K^{\int}$ satisfying $w(S-T^q)>0$.
If $q$ is a power of $p$, we define an \emph{absolute (q-power) Frobenius lift}
as a relative Frobenius lift for which $\varphi_K$ is a $q$-power
Frobenius lift.
\end{definition}

\begin{remark}\label{rem:phi-r-to-qr}
Note that $w(T^{-q}(S-T^q))>0$. Thus by Proposition \ref{prop:w-bounded-additive}, we have $w_r(T^{-q}(S-T^q))>0$ for $r$ sufficiently small. This yields $w_r(S)=w_r(T^q)=qr$; hence $\varphi$ maps $\r_K^r$ to $\r_K^{qr}$ for $r$ sufficiently small.
\end{remark}

%\begin{definition}\label{def:phi-action-twist}
%We can view a $\varphi$-module $M$ over a $\varphi$-ring $R$ as a left
%module over the twisted polynomial ring $R\{X\}$. For a positive
%integer $a$, define the \emph{a-pushforward} functor $[a]_*$ from
%$\varphi$-modules to $\varphi^a$-modules to be the restriction along the
%inclusion $R\{X^a\}\hookrightarrow R\{X\}$. Define the
%\emph{a-pullback} functor $[a]^*$ from $\varphi^a$-modules to
%$\varphi$-modules to be the extension of scalars functor $M\ra
%M\otimes_{R\{X^a\}}R\{X\}$.
%\end{definition}

Henceforth we fix a relative Frobenius lift $\varphi$ on $\r_K$ such that $\varphi_K$ is an automorphism. From Definition \ref{def:Frob-lift}, it is clear that $\varphi$ restricts to an isometry on $\r_K^\bd$ with respect to $|\cdot|$. Hence $\varphi$ extends to an automorphism on $\calE_K$ by continuity which we again denote by $\varphi$.

\begin{definition}\label{def:difference-ring}
A \emph{difference algebra/field} is an algebra/field $R$ equipped with an endomorphism $\varphi$. We say that $R$ is \emph{inversive} if $\varphi$ is an automorphism. A \emph{difference module} over $R$ is a finite free $R$-module $M$ equipped with an $R$-linear map $\varphi^{*}M\ra M$, which we also think of as a semilinear
action $\varphi$ on $M$; the semilinearity means that for $r\in R$ and
$m\in M$, $\varphi(rm)=\varphi(r)\varphi(m)$. We say that $M$ is \emph{dualizable} if $\varphi^*M\ra M$ is an isomorphism. By a \emph{$\varphi$-module} over $R$ we mean a dualizable difference module over $R$.
\end{definition}

\begin{definition}
For any $\varphi$-module $M$ over a difference algebra $R$, we define $\mathrm{H}^0(M)$ and $\mathrm{H}^1(M)$ by setting $\mathrm{H}^{0}(M)=M^{\varphi=1}$ and $\mathrm{H}^1(M)=\frac{M}{(\varphi-1)M}$ respectively. It is clear that $\mathrm{H}^{1}(M)$ classifies the extensions of the trivial $\varphi$-module $R$ by $M$ in the category of $\varphi$-modules over $R$.
\end{definition}

\begin{definition}\label{def:phi-module-slope}
For any $R\in\{\calE_K,\r_K^\bd,\r_K\}$, if $M$ is a $\varphi$-module over $R$ of rank
$n>0$, let $v$ be a generator of $\wedge^n
M$, and suppose $\varphi(v)=\lambda v$ for some $\lambda\in R^\times$. It follows from Proposition \ref{prop:units} that $R^\times\subseteq\calE_K^\times$.
We then define the \emph{degree} of $M$ by setting $\deg(M)=-w(\lambda)$ which is independent of the choice of $v$ because $\varphi$ is an isometry on $\calE_K$, and we define the
\emph{slope} of $M$ by setting $\mu(M)=\deg(M)/\rank(M)$.
\end{definition}

\begin{remark}\label{rem:sign-convention}
The sign convention used here for degrees of $\varphi$-modules is opposite to that used in the previous work of Kedlaya \cite{Ke04, Ke05, Ke06}. We change it here to match the sign convention used in the coming work \cite{KL11} which matches the sign convention used in geometric invariant theory, in which the ample line bundle $\mathcal{O}(1)$ on any projective space has degree $1$.
\end{remark}
\begin{definition}\label{def:r(c)}
For any difference algebra $R$ over $K$ and any $n\in\mathbb{Z}$, define the rank $1$ $\varphi$-module \emph{$R(n)$} by setting the $\varphi$-action as
\begin{equation}\label{eq:twist-def}
\varphi(rv)=\pi^{-n}\varphi(r)v, \quad r\in R
\end{equation}
for some generator $v$.
For any $\varphi$-module $M$ over $R$, set the $\varphi$-module $M(n)=M\otimes_RR(n)$.
\end{definition}
%\begin{enumerate}
%\item[(1)] If $0\rightarrow M_1\rightarrow M\rightarrow
%M_2\rightarrow 0$ is exact, then $\deg(M)=\deg(M_1)+\deg(M_2)$.
%\item[(2)] We have $\mu(M_1\otimes M_2)=\mu(M_1)+\mu(M_2)$.
%\item[(3)] We have $\deg(M^\vee)=-\deg(M)$ and
%$\mu(M^\vee)=-\mu(M)$.
%\item[(4)] We have $\mu(M(c))=\mu(M)+c$.
%\item[(5)] If $M$ is a $\varphi$-module, then $\mu([a]_{*}M)=a\mu(M)$.
%\item[(6)] If $M$ is a $\varphi^a$-module, then $\mu([a]^*M)=a^{-1}\mu(M)$.
%\end{enumerate}

\begin{lemma}\label{lemma:degree-q}
There exists an $r_\varphi>0$ such that for any $a\in \widehat{\overline{K}}$ with $0<v(a)<r_{\varphi}$, the equation $\varphi(T)=a$ has $q$ roots (with multiplicity) in $\widehat{\overline{K}}$. Furthermore, each of the roots has valuation $v(a)/q$.
\end{lemma}
\begin{proof}
Note that the conditions $\varphi(T)\in\r_K^\int$ and $w(\varphi(T)-T^q)>0$ imply that the Newton polygon for $\varphi(T)$ has a minimal positive slope $r_0$. It therefore follows that if $0<v(a)<qr_0$, the Newton polygon for $\varphi(T)-a$ has $v(a)/q$ as the minimal positive slope with multiplicity $q$. Hence by the theory of Newton-polygons, the equation $\varphi(T)=a$ has $q$ roots (with multiplicity), and each of the roots has valuation $v(a)/q$. Therefore we can choose $r_\varphi$ to be $qr_0$.
\end{proof}

\begin{lemma}\label{lemma:Robba-to-bounded}
Let $f$ be a nonzero element of $\r_K$. If $\varphi(f)=\lambda f$ for some $\lambda\in\r^{\times}_K$ with $w(\lambda)\leq0$, then
$f\in\r^{\times}_K$.
\end{lemma}

\begin{proof}
We first get that $f\in\r_K^\bd$ by \cite[Proposition 1.2.6]{Ke06} (although it is proved under the
hypothesis that $K$ is discretely valued, the proof works in our situation).
For any $g\in\r_K^{\bd,r}$, it follows from \cite[Lemma 8.2.6(c)]{Ke07} that $g\in(\r_K^{\bd,r})^\times$ if and only if its Newton polygon has no slopes in $[0,r]$. Now suppose that the contrary of the lemma is true. We may choose some $0<r_0<r_\varphi$ so that $f$ has a root of valuation $r_0$ and $\lambda\in(\r_K^{\bd,r_0})^\times$. We then deduce from the equality $\varphi(f)=\lambda f$ and Proposition \ref{lemma:degree-q} that $f$ has at least $q$ roots with valuation $r_0/q$. Iterating this argument, for any $n\in\mathbb{N}$, we get that $f$ has at least $q^n$ roots (with multiplicity) with valuation $r_0/q^n$. Since $\sum_{n\in\mathbb{N}}q^n\times(r_0/q^n)=\infty$, we get that the sum of the slopes of $f$ in $[0,r_0]$ is not finite. However, since $f$ is a nonzero element of $\r_K^\bd$, the sum of its slopes in $[0,r_0]$ is finite.
This yields a
contradiction.
\end{proof}

\begin{definition}\label{def:semistable}
Let $R\in\{\calE_K,\r_K\}$, and let $M$ be a nonzero $\varphi$-module over $R$. We say that $M$ is \emph{semistable} if $\mu(N)\leq\mu(M)$ for any nonzero
$\varphi$-submodule $N$. We say that $M$ is
\emph{stable} if $\mu(N)<\mu(M)$ for any proper nonzero $\varphi$-submodule
$N$.
\end{definition}

\begin{proposition}\label{prop:rank1-stable}
Any rank $1$ $\varphi$-module $M$ over $\r_K$ is
stable.
\end{proposition}
\begin{proof}
By tensoring $M$ with $M^\vee=\mathrm{Hom}_{\r_K}(M,\r_K)$, it reduces to prove the proposition for $M\cong\r_K$.
Now suppose that $N\subseteq\r_K$ is a nonzero $\varphi$-submodule so that $\mu(N)\geq\mu(\r_K)=0$. Choose a generator $f$ of $N$, and write $\varphi(f)=\lambda f$ for some $\lambda\in\r_K^{\times}$; then $w(\lambda)\leq0$ since $\mu(N)\geq0$. It thus follows from Lemma \ref{lemma:Robba-to-bounded} that
$f$ is invertible in $\r_K$, yielding $N=\r_K$. In other words, $\mu(N)<\mu(\r_K)$ unless $N=\r_K$, as
desired.
\end{proof}
\begin{corollary}\label{cor:stable}
Suppose that $N\subseteq M$ are two $\varphi$-modules over $\r_K$ of the
same rank; then $\mu(N)\leq\mu(M)$, with equality if and only if
$N=M$.
\end{corollary}
\begin{proof}
Suppose $\rank M=n$. Then apply the above proposition to
$\wedge^n N\subseteq\wedge^n M$.
\end{proof}

\begin{definition}\label{def:semistable-filtration}
Let $R\in\{\calE_K,\r_K\}$. For any nonzero $\varphi$-module $M$ over $R$, a
\emph{semistable filtration} of $M$ is a filtration $0=M_0\subset
M_1\cdots\subset M_l=M$ of $M$ by saturated $\varphi$-submodules, such
that each successive quotient $M_i/M_{i-1}$ is a semistable
$\varphi$-module of some slope $s_i$.  The \emph{slope
multiset} of a semistable filtration of $M$ is the multiset in which
each slope of a successive quotient occurs with multiplicity equal
to the rank of that quotient, and we call the associated Newton polygon of the slope multiset (see \cite[Definition 3.5.1]{Ke05}) the \emph{slope polygon} of
this filtration.
%is the polygonal line starting at $(0,0)$ and consisting of, for $i=l,\ldots,1$ in order, a segment of horizontal width $\rank(M_i/M_{i-1})$ and slope $\mu(M_i/M_{i-1})$; the right endpoint of the slope polygon is $(\rank(M),\deg(M))$. A \emph{Harder-Narasimhan
%(HN) filtration} of $M$ is a semistable filtration $0=M_0\subset
%M_1\subset M_2\cdots\subset M_l=M$ so that $s_1>
%\cdots>s_l$.
\end{definition}

\begin{proposition}\label{prop:maximal-sub}
If $K$ is spherically complete, then every
nonzero $\varphi$-module $M$ over $\r_K$ has a unique maximal $\varphi$-submodule which has the maximal slope. Furthermore, it is semistable and saturated.
\end{proposition}
\begin{proof}
The spherical completeness of $K$ implies that $\r_K$ is a B\'ezout domain by \cite[Th\'eor\`{e}me 2]{L62}. Hence if $N$ is a finite $\r_K$-submodule of $M$, both $N$ and its saturation are finite free $\r_K$-modules. It follows that the saturation of any $\varphi$-submodule of
$M$ and the sum of any two $\varphi$-submodules of $M$ are still $\varphi$-submodules of $M$.

We proceed by induction on the rank of $M$. The initial case follows from Proposition \ref{prop:rank1-stable}. Now suppose that $\rank M=d$ for some $d\geq2$  and the proposition is
true for $\varphi$-modules having rank $\leq d-1$. Let $\mu(M)=s$. If $M$
is semistable, then we are done. Otherwise, let $P$ be a
$\varphi$-submodule of slope $>s$ and of maximal rank. By Corollary \ref{cor:stable}, the saturation $\widetilde{P}$ of $P$ satisfies
$\mu(\widetilde{P})\geq\mu(P)$. Replacing $P$ with $\widetilde{P}$, we may suppose that $P$ is saturated. Hence $\rank P\leq d-1$, otherwise we must have $P=M$, yielding $\mu(M)=\mu(P)>s$ which is a
contradiction. By
inductive assumption, $P$ has a unique maximal $\varphi$-submodule
$P_1$ which has the maximal slope. We claim that $P_1$ is also the unique
maximal $\varphi$-submodule of $M$ which has the maximal slope. Suppose that the contrary of the claim is true. Let $Q$ be a $\varphi$-submodule of $M$ so that either $\mu(Q)>\mu(P_1)$ or $\mu(Q)=\mu(P_1)$ and $Q\nsubseteq P_1$; then $\mu(Q)\geq
\mu(P_1)$ and $Q\nsubseteq P$. Consider the following exact
sequence of $\varphi$-submodules
\[
0\longrightarrow P\cap Q\longrightarrow P\oplus Q\longrightarrow
P+Q\longrightarrow 0.
\]
Since $\mu(P\cap Q)\leq\mu(P_1)\leq\mu(Q)$, we get
\begin{eqnarray*}
\deg(P+Q)&=&\rank(P)\mu(P)+\rank(Q)\mu(Q)-
\rank(P\cap Q)\mu(P\cap Q)\\
&\geq&\rank(P)\mu(P)+(\rank(Q)-\rank(P\cap Q))\mu(Q)\\
&\geq&(\rank(P)+\rank(Q)-\rank(P\cap Q))\mu(P)\\
&=&\rank(P+Q)\mu(P);
\end{eqnarray*}
hence $\mu(P+Q)\geq\mu(P)>s$. However, since $P$ is saturated and $Q\nsubseteq P$, it follows that $\rank(Q+P)>\rank P$ which contradicts the maximality of $\rank P$. This yields the claim which finishes the inductive step.

Now let $P\subseteq M$ be the unique maximal $\varphi$-submodule which has the maximal slope. The saturation $\widetilde{P}$ of $P$ satisfies $\mu(\widetilde{P})\geq P$. This forces $\widetilde{P}=P$ by the maximality of $P$. Hence $P$ is saturated. The semistability of $P$ follows directly.
\end{proof}

\begin{theorem}\label{thm:existence-HN-filtraion}
If $K$ is spherically complete, then every $\varphi$-module $M$ over $\r_K$
admits a unique HN filtration.
\end{theorem}
\begin{proof}
The uniqueness follows from formal properties of slopes and Corollary \ref{cor:stable}. In fact, by
the definition of HN filtration, for any $i\geq1$, $M_i$ can be
characterized as the preimage of the unique maximal $\varphi$-submodule of $M/M_{i-1}$
which has the maximal slope. This also suggests the way of showing the existence. We take $M_1$ to be the maximal $\varphi$-submodule of $M$ which has the maximal slope. Since $M_1$ is saturated by Proposition \ref{prop:maximal-sub}, $M/M_1$ is a $\varphi$-module over $\r_K$. Then we take $M_2$ to be the preimage of the maximal $\varphi$-submodule of $M/M_1$ which has the maximal slope. Iterating this process, we get the HN filtration of $M$.
\end{proof}

\begin{definition}
Suppose that $K$ is spherically complete. For any $\varphi$-module $M$ over $\r_K$, the slopes $s_i$'s of the HN filtration are called the \emph{slopes} of $M$ and the slope polygon of the HN filtration is called the \emph{HN-polygon} of $M$.
\end{definition}

\begin{proposition}\label{prop:HN-below}
Suppose that $K$ is spherically complete. The for any $\varphi$-module $M$ over $\r_K$, the HN-polygon of $M$ lies above the slope polygon of any semistable filtration of $M$, with the same endpoint.
\end{proposition}
\begin{proof}
This is a formal consequence of the definition of HN filtration. We
refer to \cite[Proposition 3.5.4]{Ke05} for a proof. Beware that both our sign convention of slopes and definition of slope polygons are ``opposite" to that used in \cite{Ke05}.
\end{proof}

\begin{definition}\label{def:pure}
Let $M$ be a $\varphi$-module over $\calE_K$ (resp. $\r_K^{\bd}$). For $c,d\in\mathbb{Z}$ with $d>0$,  a \emph{$(c,d)$-pure model} of $M$ is a finite free $\OO_{\calE_K}$-submodule (resp. $\r_K^{\int}$-submodule) $M_0$ of $M$ with $M_0\otimes_{\OO_{\calE_K}}\calE_K=M$ (resp. $M_0\otimes_{\r^{\int}_K}\r_K^\bd=M$) so that the $\varphi$-action on $M$ induces an isomorphism $\pi^{c}(\varphi^d)^*M_0\cong M_0$. For a $\varphi$-module $M$ over $\r_K$, a \emph{$(c,d)$-pure model} of $M$ is a $\r_K^{\int}$-submodule $M_0$ with $M_0\otimes_{\r_K^\int}\r_K=M$ so that $M_0\otimes_{\r_K^\int}\r_K^\bd$ is stable under $\varphi$ and the $\varphi$-action induces an isomorphism $\pi^{c}(\varphi^d)^*M_0\cong M_0$.
For $s\in\mathbb{Q}$, we say that $M$ is \emph{pure of slope $s$} if $M$ admits a $(c,d)$-pure model for some (hence any) $c,d\in\mathbb{Z}$ with $d>0$ and $s=c/d$. If $s=0$, we also say that $M$ is \emph{\'etale}, and a $(0,1)$-pure model is also called an \emph{\'etale model}.
\end{definition}

\begin{proposition}\label{prop:pure-semistable}
If $M$ is a pure $\varphi$-module over $\r_K$, then $M$ is
semistable.
\end{proposition}
\begin{proof}
We follows the proof of \cite[Theorem 1.6.10(a)]{Ke06}. Suppose that $M$ admits a $\varphi$-submodule $N$ such that $\mu(N)>\mu(M)$. By replacing $M$ with $\wedge^{\rank N}M$, we may assume that $\rank N=1$. By twisting, we may further assume that $N$ is trivial. Hence $\mathrm{H}^0(N)\neq0$. Choose a nonzero $\varphi$-invariant vector $v\in N$. By replacing $\varphi$ with $\varphi^a$ for a suitable positive integer $a$, we may assume that $\mu(M)=n\in\mathbb{Z}_{<0}$. We choose a $(n,1)$-pure model $M_0$ of $M$. Let $\mathrm{e}=\{e_1,\dots,e_m\}$ be a basis of $M_0$, and write $\varphi(e_i)=\sum_{j=1}^me_jF_{ji}$ for $F_{ji}\in\r^\bd_K$. By the definition of pure models, we see that $w(F_{ji})\geq -n$ for all $j,i$.  By \cite[Proposition 1.5.4]{Ke06} (this proposition ultimately relies on \cite[Proposition 1.2.6]{Ke06} whose proof works for general $K$), we get that $v\in M_0\otimes_{\r^\int_K}\r^\bd_K$. Write $v=\sum_{i=1}^mc_{i}e_i$; then $\varphi(v)=v$ implies $c_i=\sum_{i=1}^mF_{ij}\varphi(c_j)$. This yields $\min_i\{w(c_i)\}\geq -n+\min_j\{w(c_j)\}$ which is a contradiction.
\end{proof}

The converse of Proposition \ref{prop:pure-semistable} is more difficult. It is only known for discretely valued $K$ thanks to the following slope filtration theorem of Kedlaya \cite[Theorem 1.7.1]{Ke06}.

\begin{theorem}\label{thm:slope-filtration}
If $K$ is discretely valued, then every semistable $\varphi$-module
over $\r_K$ is pure. In particular, every $\varphi$-module $M$ over
$\r_K$ admits a unique filtration $0=M_0\subset M_1\subset\cdots\subset
M_l=M$ by saturated $\varphi$-submodules whose successive quotients are
pure with $\mu(M_1/M_0)>\cdots>\mu(M_{l}/M_{l-1})$.
\end{theorem}

The following propositions will be used later.
\begin{proposition}\label{prop:model-unique}
Suppose that $K$ is discretely valued, and that $M$ is a pure $\varphi$-module over $\r_K$.
If $M_1$ and $M_2$ are two pure models of $M$, then $M_1\otimes_{\r^\int_K}\r_K^\bd=M_2\otimes_{\r^\int_K}\r_K^\bd$.
\end{proposition}
\begin{proof}
This follows from \cite[Proposition 1.5.5]{Ke06}.
\end{proof}

\begin{proposition}\label{prop:wedge}
Suppose that $K$ is discretely valued, and let $M$ be a $\varphi$-module over $\r_K$. The following are true.
\begin{enumerate}
\item[(1)]Let $a$ be a positive integer. Then $M$ is semistable of slope $s$ if and only if it is semistable of slope $as$ as a $\varphi^a$-module.
\item[(2)]Suppose that $M$ has slopes $s_1\geq\cdots\geq s_n$ counted with multiplicity. Then for any $1\leq d\leq n$, the slope multiset of $\wedge^d M$ is $\{s_{i_1}+\cdots+s_{i_d}|1\leq i_1<\cdots<i_d\leq n\}$.
\end{enumerate}
\end{proposition}
\begin{proof}
For $(1)$, it suffices to show that $M$ is pure of slope $s$ if and only if it is pure of slope $as$ as a $\varphi^a$-module; this is \cite[Lemma 1.6.3]{Ke06}. For $(2)$, see \cite[Remark 1.7.2]{Ke06}.
\end{proof}

It is clear that the purity of $\varphi$-modules is preserved by tensor products. Hence for discretely valued $K$, it follows from Theorem \ref{thm:slope-filtration} that the semistability of $\varphi$-modules over $\r_K$ is also preserved by tensor products.

\begin{question}
 Suppose that $K$ is merely spherically complete. Do we still have the equivalence of purity and semistability for $\varphi$-modules over $\r_K$? If this fails to be true, is the semistability of $\varphi$-modules over $\r_K$ still preserved by tensor products?
\end{question}

\subsection{Dieudonn\'e-Manin decomposition}

\begin{definition}
Let $R$ be a difference algebra. A difference module over $R$ is \emph{trivial} if it admits a $\varphi$-invariant basis. We say that $R$ is \emph{weakly difference-closed} if every dualizable difference module over $R$ is trivial. We say that $R$ is \emph{strongly difference-closed} if $R$ is inversive and weakly difference-closed.
\end{definition}

We fix a difference field $F$ which is complete for a $\varphi$-invariant nonarchimedean absolute value $|\cdot|_F$. Then $\varphi$ induces an endomorphism on the residue field $k_F$ of $F$; we view $k_F$ as a difference field with this endomorphism.

\begin{lemma}\label{lem:difference-closed}
Suppose that $F$ is spherically complete. Then the following are true.
\begin{enumerate}
\item[(1)]If $k_F$ is weakly difference-closed, then for any $a\in F$, there exists $x\in F$ with $|x|_F=|a|_F$ such that $\varphi(x)-x=a$;
\item[(2)]If $k_F$ is inversive, so is $F$.
\end{enumerate}
\end{lemma}
\begin{proof}
We first prove $(1)$. We equip $F$ with a partial order: for any $x,y\in F$, we say $x>y$ if
\[
|\varphi(x) - x - a|_F < |x - y|_F\leq|\varphi(y) - y - a|_F.
\]
We first show that the set $\{x\mid|x|_F\leq|a|_F\}$ has a maximal element. Suppose that $x_1<x_2<\dots$ is an infinite chain in $\{x\mid|x|_F\leq|a|_F\}$. Let $r_i=|\varphi(x_i)-x_i-a|_F$. It follows that $\mathrm{B}(x_1,r_1)\supset \mathrm{B}(x_2,r_2)\supset\cdots$. Since $F$ is spherically complete, $\bigcap_{i=1}^{\infty}\mathrm{B}(x_i,r_i)$ is nonempty.
Pick an $x_0\in\bigcap_{i=1}^{\infty}\mathrm{B}(x_i,r_i)$. Then $|x_0-x_i|_F\leq r_i$ for any $i\geq 1$. Hence
\[
|\varphi(x_0)-x_0-a|_F=|(\varphi(x_i)-x_i-a)+\varphi(x_0-x_i)-(x_0-x_i)|_F\leq r_i
\]
for any $i\geq1$. On the other hand, since $|x_i-x_{i+1}|_F>r_{i+1}$, we get $|x_0-x_{i}|_F=|(x_0-x_{i+1})-(x_i-x_{i+1})|_F=|x_i-x_{i+1}|_F$. Hence
\[
|\varphi(x_0)-x_0-a|_F\leq r_{i+1}<|x_i-x_{i+1}|_F=|x_0-x_i|_F,
\]
yielding $x_0>x_i$ for any $i\geq1$. We therefore prove the claim by Zorn's lemma. Let $x'$ be a maximal element of the set $\{x\mid|x|_F\leq|a|_F\}$. We claim that $\varphi(x')-x'=a$. If this is not the case, let $b=\varphi(x')-x'-a$. Since $k_F$ is
weakly difference-closed, by \cite[Lemma 14.3.3(b),(c)]{Ke07}, we may choose some $y$ with $|y|_F=1$ so that $|\frac{\varphi(b)}{b}\varphi(y)-y+1|_F<1$ . Let $y'=x'+by$. Then
\[
|\varphi(y')-y'-a|_F=|b(\frac{\varphi(b)}{b}\varphi(y)-y+1)|_F<|b|_F=|x'-y'|_F=|\varphi(x')-x'-a|_F.
\]
This implies that $y'>x'$ which contradicts the maximality of $x$. Hence $\varphi(x')-x'=a$; it is clear that $|x'|_F=|a|_F$.

The proof of $(2)$ is similar. Let $a\in F$. We equip $F$ with a partial order: for any $x,y\in F$, we say $x>y$ if
\[
|\varphi(x)-a|_F < |x - y|_F\leq|\varphi(y)-a|_F.
\]
Suppose that $x_1<x_2<\cdots$ is an infinite chain in $\{x\mid|x|_F\leq|a|_F\}$. Let $r_i=|\varphi(x_i)-a|_F$. It follows that $\mathrm{B}(x_1,r_1)\supset \mathrm{B}(x_2,r_2)\supset\cdots$.
Pick an $x_0\in\bigcap_{i=1}^{\infty}\mathrm{B}(x_i,r_i)$. Then $|x_0-x_i|_F\leq r_i$ for any $i\geq 1$.
A similar argument shows that
$x_0>x_i$ for any $i\geq1$. By Zorn's lemma we choose a maximal element $x'\in\{x\mid|x|_F\leq|a|_F\}$. Now suppose that $b'=\varphi(x')-a$ is nonzero. Since $k_F$ is inversive, we may choose some $y\in F$ with $|y|_F=1$ so that $|\frac{\varphi(b)\varphi(y)}{b}-1|_F<1$. It follows that $x'-by>x'$ which is a contradiction. Hence $\varphi(x')=a$.
\end{proof}

\begin{convention}\label{con:norm}
For any valuation $v$ (resp. norm $|\cdot|$) and a matrix $A=(A_{ij})$, we use $v(A)$ (resp. $|A|$) to denote the minimal valuation (resp. maximal norm) among the entries.
\end{convention}

\begin{lemma}\label{lem:pure-trivial}
If $F$ is spherically complete, and if $k_F$ is weakly difference-closed, then for any $A\in\mathrm{GL}_d(\OO_F)$,
there exists $U\in\mathrm{GL}_d(\OO_F)$ so that $U^{-1}A\varphi(U)=I_d$.
\end{lemma}
\begin{proof}
The reduction of $A$ in $\mathrm{GL}_d(k_F)$ defines a dualizable difference module over $k_F$. Since $k_F$ is weakly difference-closed, this module is trivial. This implies that we may choose some $U_1\in\mathrm{GL}_d(\OO_F)$ so that $|U_1^{-1}A\varphi(U_1)-I_d|_F=c<1$.  We will inductively construct a convergent sequence $U_1, U_2,\dots\in\mathrm{GL}_d(\OO_F)$ so that
\[
|U_i-U_{i+1}|_F\leq c^i,\quad |U_i^{-1}A\varphi(U_i)-I_d|_F\leq c^i
\]
for every $i\geq1$. Choose $u\in F$ so that
$|u|_F=c$. Given $U_i$, by Lemma \ref{lem:difference-closed}(1), there exists some $X_i\in \mathrm{M}_d(u^i\OO_F)$ so that
\[
\varphi(X_i)-X_i+(U_i^{-1}A\varphi(U_i)-I_d)=0.
\]
Put $U_{i+1}=U_i(I_d+X_i)$. It follows that $U_{i+1}\in\mathrm{GL}_d(\OO_F)$ and $|U_{i+1}-U_i|_F\leq c^{i}, |U_{i+1}^{-1}A\varphi(U_{i+1})- I_d|_F\leq c^{i+1}$.
Let $U=\lim_{i\ra\infty}U_i$. Then $U^{-1}A\varphi(U)=I_d$.
\end{proof}

\begin{definition}
Let $R$ be a difference algebra. For $\lambda\in R$ and a positive integer $d$, define
\emph{$V_{\lambda,d}$} to be the difference module over $R$ with a basis $e_1$, $\dots$, $e_d$ so that
\[
\varphi(e_1)=e_2, \dots, \varphi(e_{d-1})=e_d, \varphi(e_d)=\lambda e_1;
\]
and any such a basis is
called a \emph{standard basis} of $V_{\lambda,d}$. For a difference module $V$ over $R$, a \emph{Dieudonn\'e-Manin decomposition} of $V$
is a direct sum decomposition $V\cong\oplus_{i=1}^{n}V_{\lambda_i,d_i}$ for some $\lambda_i,d_i, 1\leq i\leq n$.
\end{definition}

The following theorem generalizes the usual Dieudonn\'e-Manin classification theorem for difference modules over complete discretely valued difference fields (e.g. \cite[Theorem 14.6.3]{Ke07}) to difference modules over spherically complete difference fields.

\begin{theorem}
If $F$ is spherically complete, and if $k_F$ is strongly difference-closed, then every dualizable difference module $V$ over $F$ admits a Dieudonn\'e-Manin decomposition.
\end{theorem}

\begin{proof}
First note that if $V$ is pure of spectral norm 1 (see \cite[Definition 14.4.6]{Ke07} for the definition), then $V$ admits a basis on which $\varphi$ acts via an element of $\mathrm{GL}(\OO_F)$ (\cite[Proposition 14.4.16]{Ke07}); hence $V$ is trivial by Lemma \ref{lem:pure-trivial}. Furthermore, Lemma \ref{lem:difference-closed}(1) implies that $\mathrm{H}^1(V)=V/(\varphi-1)$ is trivial in this case.

Now we follow the line of the proof of \cite[Theorem 14.6.3]{Ke07}.  By Lemma \ref{lem:difference-closed}(2), $F$ is inversive. Hence by \cite[Theorem 14.4.13]{Ke07}, it reduces to show the theorem for those $V$ which are pure of spectral norm $s>0$. Let $m$ be the smallest positive integer so that $s^m\in |F|_F$, and choose $u\in F$ so that $|u|_F=s^m$.  Then the first paragraph implies that $u^{-1}\varphi^m$ fixes some nonzero vector $v$ of $V$. This induces a nonzero map from $V_{u,m}$ to $V$. It follows from \cite[Lemma 14.6.2]{Ke07} that $V_{u,m}$ is irreducible. Hence this map is injective. Repeating this argument we get that $V$ is a successive extension of copies of $V_{u,m}$. Note that $\mathrm{Ext}^1(V_{u,m}, V_{u,m})=\mathrm{H}^1(V_{u,m}^{\vee}\otimes V_{u,m})=0$ since $V_{u,m}^{\vee}\otimes V_{u,m}$ is pure of spectral norm 1. We thus deduce that $V$ is a direct sum of copies of $V_{u,m}$.
\end{proof}

\subsection{Extended base rings}

Henceforth we assume that $K$ is discretely valued and $\pi$ is a uniformizer of $K$.
\begin{definition}\label{def:extended-Robba}
For any interval $I\subset(0,\infty]$, let $\widetilde{\r}_K^I$ be the set of formal sums
\[
f=\sum_{i\in\mathbb{Q}}a_iu^i
\]
with $a_i\in K$ satisfying the following conditions.
\begin{enumerate}
\item[(1)]For any $c>0$, the set of $i\in\mathbb{Q}$ so that $|a_i|\geq c$ is well-ordered (has no infinite decreasing subsequence).
\item[(2)]For any $s\in I$, $v(a_i)+si\ra\infty$ as $i\ra\pm\infty$, and $\inf_{i\in\mathbb{Q}}\{v(a_i)+si\}>-\infty$.
\end{enumerate}
Then $\inf_{i\in\mathbb{Q}}\{v(a_i)+si\}$ is attained at some $i$ because $K$ is discretely valued. These series form a ring under formal series addition and multiplication. For any $s\in I$, set the valuation $w_s(f)=\min_{i\in\mathbb{Q}}\{v(a_i)+si\}$, and the corresponding multiplicative norm $|f|_s=\max_{i\in\mathbb{Q}}\{|a_i|p^{-si}\}$. We denote $\widetilde{\r}_K^{(0,r]}$ by $\widetilde{\r}_K^r$ for simplicity.  We call the union $\widetilde{\r}_K=\cup_{r>0}\widetilde{\r}^r_K$ the \emph{extended Robba ring} over $K$. We view $\widetilde{\r}_K$ as a difference algebra over $K$ with the endomorphism $\varphi(f)=\sum_{i\in\mathbb{Q}}\varphi_K(a_i)u^{qi}$.
\end{definition}

\begin{remark}\label{rem:kedlaya-definition}
The definition of the extended Robba ring in \cite{Ke06} misses the second part of condition (2) of Definition \ref{def:extended-Robba}.
\end{remark}

\begin{definition}\label{def:extended-bounded}
For any $r>0$, let $\widetilde{\r}_K^{\bd,r}$ be the subring of $\widetilde{\r}_K^r$
consisting of series with $\{v(a_i)\}_{i\in\mathbb{Q}}$ bounded below. Let $\widetilde{\r}_K^\bd=\cup_{r>0}\widetilde{\r}_K^{\bd,r}$. We equip $\widetilde{\r}^\bd_K$ with the valuation $w(f)=\min_{i\in\mathbb{Q}}\{v(a_i)\}$ and the corresponding multiplicative norm $|f|=\max_{i\in\mathbb{Q}}\{|a_i|\}$. Let $\widetilde{\r}_K^{\int}$ be the valuation ring of $\widetilde{\r}_K^{\bd}$, and let
$\widetilde{\r}_K^{\int,r}=\widetilde{\r}_K^{\int}\cap\widetilde{\r}_K^{\bd,r}$. We call $\widetilde{\r}^\bd_K$ the \emph{extended bounded Robba ring} over $K$.
\end{definition}

\begin{definition}\label{def:extended-calE}
Let $\widetilde{\calE}_K$ be the ring of
formal sums
\[
f=\sum_{i\in\mathbb{Q}}a_iu^i
\]
with coefficients in $K$ satisfying the following conditions.
\begin{enumerate}
\item[(1)] For each $c>0$, the set
of $i\in\mathbb{Q}$ such that $|a_i|\geq c$ is well-ordered.
\item[(2)]The set $\{v(a_i)\}_{i\in\mathbb{Q}}$ is bounded below and
$v(a_i)\ra\infty$ as $i\ra-\infty$.
\end{enumerate}
We equip $\widetilde{\calE}_K$ with the valuation $w(f)=\min_{i\in\mathbb{Q}}\{v(a_i)\}$ and the corresponding multiplicative norm $|f|=\max_{i\in\mathbb{Q}}\{|a_i|\}$. Let $\OO_{\widetilde{\calE}_K}$ be the valuation ring of $\widetilde{\calE}_K$.
\end{definition}

\begin{remark}
It is clear that $\widetilde{\r}_K^{\bd,r}$ is the subring of $\widetilde{\calE}_K$ consisting of series such that $v(a_i)+ri\ra \infty$ as $i\ra-\infty$. The natural inclusion $\widetilde{\r}_K^\bd\rightarrow\widetilde{\calE}_K$ is an isometry with respect to $w$, and identifies $\widetilde{\calE}_K$ with the $w$-completion of
$\widetilde{\r}_K^{\bd}$. The restriction of $\varphi$ on $\widetilde{\r}^\bd_K$ is an isometry with respect to $w$, and
we still denote by $\varphi$ its continuous extension to $\widetilde{\calE}_K$. We view $\widetilde{\r}_K^\bd$ and $\widetilde{\calE}_K$ as difference fields with the endomorphism $\varphi$.
\end{remark}

\begin{definition}\label{def:extended-topology}
For any interval $I\subseteq(0,\infty]$, we equip $\widetilde{\r}_K^I$ with the Fr\'echet topology defined by $|\cdot|_s$ for all $s\in I$; $\widetilde{\r}_K^I$ is complete for this topology. If $I=[r_1,r_2]$ is a closed interval, then $\widetilde{\r}_K^I$ becomes a $K$-Banach algebra with norm $\max\{|\cdot|_{r_1},|\cdot|_{r_2}\}$. We equip  $\widetilde{\r}_K=\cup_{r>0}\widetilde{\r}_K^r$ with the locally convex inductive limit topology. For any $r>0$, we equip $\widetilde{\r}_K^{\bd,r}$ with the norm $\max\{|\cdot|,|\cdot|_r\}$; it is a $K$-Banach algebra under this norm. We equip $\widetilde{\r}_K^\bd=\cup_{r>0}\widetilde{\r}_K^{\bd,r}$ with the locally convex inductive limit topology.
\end{definition}

\begin{definition}
For $S$ a commutative ring, let $S((u^{\mathbb{Q}}))$ denote the Hahn-Malcev-Neumann algebra of generalized power series $\sum_{i\in\mathbb{Q}}c_iu^i$, where each $c_i\in S$ and the set of $i$ with $c_i\neq0$ is well-ordered; these series form a ring under formal series addition and multiplication.
\end{definition}

\begin{remark}
It is clear that the residue fields of $\widetilde{\r}_K^\bd$ and $\widetilde{\calE}_K$ are isomorphic to $k((u^\mathbb{Q}))$.
\end{remark}

\begin{proposition}\label{prop:units-extended}
The extended Robba ring $\widetilde{\r}_K$ is a B\'ezout domain, and the units of $\widetilde{\r}_K$ are precisely the nonzero elements of $\widetilde{\r}_K^\bd$. As a consequence, $\widetilde{\r}_K^\bd$ (and hence $\widetilde{\calE}_K$) is a discretely valued field.
\end{proposition}
\begin{proof}
As explained in \cite[Remark 2.2.5]{Ke06}, $\widetilde{\r}_K$ is the analytic ring with residue field $k((u^{\mathbb{Q}}))$ in the sense of \cite[\S2.4]{Ke05}, by taking $\varphi_K$ to be an absolute Frobenius lift on $K$. The proposition then follows from \cite[Theorem 2.9.6]{Ke05} and \cite[Lemma 2.4.7]{Ke05}.
\end{proof}

\begin{remark}\label{rem:embedding}
It follows from \cite[Proposition 2.2.6]{Ke06} that there is a
$\varphi$-equivariant embedding $\tau_K:\r_K\rightarrow\widetilde{\r}_K$ so that for $r$ sufficiently small, $\r_K^r$ maps to $\widetilde{\r}_K^r$ preserving $w_r$. It thus follows that $\tau_K$ maps $\r_K^{\bd}$ to
$\widetilde{\r}_K^{\bd}$ preserving $w$. Hence $\tau_K$ induces an $\varphi$-equivariant embedding from $\calE_K$ to $\widetilde{\calE}_K$ by taking the completion; we still denote by $\tau_K$ this embedding. In this way, we view $\r^\bd_K$, $\r_K$, $\calE_K$ as difference subalgebras of $\widetilde{\r}^\bd_K$, $\widetilde{\r}_K$, $\widetilde{\calE}_K$ respectively.
\end{remark}

\begin{proposition}\label{prop:kuQ-difference-closed}
If $k$ is strongly difference-closed, so is $k((u^{\mathbb{Q}}))$.
\end{proposition}
\begin{proof}
This is \cite[Proposition 2.5.5]{Ke07}.
\end{proof}

\subsection{$\varphi$-modules over the extended Robba ring}
\begin{definition}
For any $R\in\{\widetilde{\r}_K, \widetilde{\r}^\bd_K, \widetilde{\calE}_K\}$, let $M$ be a $\varphi$-module over $R$ of rank
$n>0$. Let $v$ be a generator of $\wedge^n
M$, and suppose $\varphi(v)=\lambda v$ for some $\lambda\in R^\times\subseteq\widetilde{\calE}_K^\times$.
We define the \emph{degree} of $M$ by setting $\deg(M)=-w(\lambda)$ which is independent of the choice of $v$ because $\varphi$ is an isometry on $\widetilde{\calE}_K$, and we define the
\emph{slope} of $M$ by setting $\mu(M)=\deg(M)/\rank(M)$. Define \emph{stable}, \emph{semistable}, \emph{semistable filtration}, \emph{slope multiset}, \emph{slope polygon}, \emph{HN filtration}, \emph{$(c,d)$-pure model}, \emph{pure of slope $s$}, \emph{\'etale}, \emph{\'etale model} for $\varphi$-modules over $R$ by changing $\calE_K$, $\r^\bd_K$, $\r_K$ to $\widetilde{\calE}_K$, $\widetilde{\r}_K^\bd$, $\widetilde{\r}_K$ respectively in Definitions \ref{def:semistable}, \ref{def:semistable-filtration} and \ref{def:pure}.
\end{definition}

\begin{definition}
By an \emph{extension} of $K$, we mean a field extension $L$ of $K$ which is complete for a discrete valuation extending the one on $K$, and is equipped with an isometric field automorphism $\varphi_L$ extending $\varphi_K$. The extension $L$ is called \emph{admissible} if it has the same value group as $K$.
\end{definition}

\begin{lemma}
The field $K$ admits an admissible extension $L$ so that its residue field $k_L$ is strongly difference-closed with respect to the reduction of $\varphi_L$.
\end{lemma}
\begin{proof}
By \cite[Proposition 3.2.4]{Ke06}, $K$ admits an admissible extension $L$ so that $k_L$ is weakly difference-closed. (The condition that any \'etale $\varphi$-module over $L$ is trivial is equivalent to the condition that $k_L$ is weakly difference-closed). Since $L$ is inversive, a fortiori $k_L$ is inversive.
\end{proof}

\begin{proposition}\label{prop:pure-semistable-extended}
Every pure $\varphi$-module over $\widetilde{\r}_K$ is semistable.
\end{proposition}
\begin{proof}
The proposition follows from \cite[Theorem 1.6.10(a)]{Ke06}.
\end{proof}

\begin{proposition}
Every $\varphi$-module over $\widetilde{\r}_K$ admits a unique HN filtration.
\end{proposition}
\begin{proof}
The proposition follows from \cite[Proposition 1.4.15]{Ke06}.
\end{proof}

\begin{proposition}\label{prop:robba-extension}
If $M$ is a $\varphi$-module over $\r_K$, and if $L$ is an extension of $K$, then the HN filtration of $M$, tensored up with $\r_L$ (resp. $\widetilde{\r}_L$), gives the HN filtration of $M\otimes_{\r_K}\r_L$ (resp. $M\otimes_{\r_K}\widetilde{\r}_L$).
\end{proposition}
\begin{proof}
It reduces to show that if $M$ is semistable, then its base changes are also semistable. Note that $M$ is pure by Theorems \ref{thm:slope-filtration}. Thus its base changes are also pure; hence they are semistable by Propositions \ref{prop:pure-semistable} and \ref{prop:pure-semistable-extended}.
\end{proof}

\begin{definition}
For any $\varphi$-module $M$ over $\widetilde{\r}_K$, the slopes of the successive quotients and the slope polygon of the HN filtration of $M$ are called the \emph{slopes} and the \emph{HN-polygon} of $M$ respectively.
\end{definition}

The following theorem is the combination of \cite[Proposition 2.1.6, Theorem 2.1.8]{Ke06}.
\begin{theorem}\label{thm:Robba-decomposition}
If $k$ is strongly difference-closed, then every semistable $\varphi$-module $M$ over $\widetilde{\r}_K$ is pure. Furthermore, it admits a Dieudonn\'e-Manin decomposition $M=\bigoplus V_{\lambda,d}$ so that each $\lambda$ is a power of $\pi$.
\end{theorem}

\begin{lemma}\label{lem:calE-bounded}
Let $\alpha=\sum_{i\in\mathbb{Q}}a_iu^i\in\widetilde{\r}^{\bd,r}_K$ and $n\in\mathbb{N}$. Suppose that $\beta=\sum_{i\in\mathbb{Q}}b_iu^i\in\widetilde{\calE}_K$ satisfies
\begin{equation}\label{eq:Frobenius}
\varphi(\beta)-\pi^n\beta=\alpha.
\end{equation}
Then for any $i<0$, we have $v(b_i)\geq\min_{j\leq i}\{v(a_j)\}$. As a consequence, we have $\beta\in\widetilde{\r}^{\bd,qr}_K$. Furthermore, if $w(\beta)\geq w(\alpha)$, then $w_r(\beta)\geq\min\{w(\alpha),w_r(\alpha)\}$, and if moreover $n>0$, then $w(\beta)=w(\alpha)$.
\end{lemma}
\begin{proof}
Suppose that there exists some $i_0<0$ such that $v(b_{i_0})<\min_{j\leq i_0}\{v(a_j)\}$. By (\ref{eq:Frobenius}), we have
\[
\varphi(b_{i_0})-\pi^n b_{qi_0}=a_{qi_0}
\]
by comparing the coefficients of $u^{qi_0}$. Since $v(b_{i_0})<v(a_{qi_0})$, we get
\[
v(b_{qi_0})=v(b_{i_0})-n\leq v(b_{i_0})<\min_{j\leq i_0}\{v(a_j)\}\leq\min_{j\leq qi_0}\{v(a_j)\}.
\]
Iterating this argument, we get $v(b_{q^mi_0})=v(b_{i_0})-mn$ for any $m\in\mathbb{N}$. Thus we get an infinite descending sequence $q^mi_0$ with $v(b_{q^mi_0})$ decreasing which contradicts the condition that $v(b_i)\ra\infty$ as $i\ra-\infty$. This proves the first statement of the lemma. Hence $\varphi(\beta)=\pi^{n}\beta+\alpha$ belongs to $\widetilde{\r}_K^{\bd,r}$, yielding $\beta\in\widetilde{\r}_K^{\bd,qr}$.

Note that
\begin{eqnarray*}
w_r(\varphi(\beta))&=&\min_{i\in\mathbb{Q}}\{v(\varphi_K(b_i))+rqi\}=\min_{i\in\mathbb{Q}}\{q(v(b_i)+ri)-(q-1)v(b_i)\}\\
&\leq&\min_{i\in\mathbb{Q}}\{q(v(b_i)+ri)\}-(q-1)w(\beta)= qw_r(\beta)-(q-1)w(\beta).
\end{eqnarray*}
Thus if $w(\beta)\geq w(\alpha)$ and $w_r(\beta)<\min\{w(\alpha),w_r(\alpha)\}$, then
\[
w_r(\varphi(\beta))\leq qw_r(\beta)-(q-1)w(\beta)<w_r(\beta)<w_r(\alpha).
\]
This contradicts the condition $w_r(\varphi(\beta))\geq\min\{w_r(\pi^n\beta),w_r(\alpha)\}\geq\min\{w_r(\beta),w_r(\alpha)\} $. If $n>0$, then $w(\pi^n\beta)>w(\varphi(\beta))=w(\beta)$; hence $w(\alpha)=\min\{w(\varphi(\beta)), w(\pi^n\beta)\}=w(\beta)$.
\end{proof}

For any difference algebra $R$ with an automorphism $\varphi$, we set the twisted powers $a^{\{m\}}$ for any $m\in\mathbb{Z}$ and $a\in R$ by the two-way recurrence
\[
a^{\{0\}}=1,\quad a^{\{m+1\}}=\varphi(a^{\{m\}})a.
\]
\begin{lemma}\label{lem:Frobenius-equation}
Suppose that $k$ is strongly difference-closed. Then the following are true.
\begin{enumerate}
\item[(1)]
Let $\alpha\in\widetilde{\calE}_K$. If $n\neq0$, then (\ref{eq:Frobenius}) admits a unique solution $\beta\in\widetilde{\calE}_K$ which is
\begin{equation}\label{eq:solution-n-nega}
\beta=-\sum_{m=0}^\infty(\pi^{-n})^{\{m+1\}}\varphi^m(\alpha)
\end{equation}
if $n<0$, or
\begin{equation}\label{eq:solution-n-posi}
\beta=\sum_{m=0}^\infty(\pi^{-n})^{\{-m\}}\varphi^{-m-1}(\alpha)
\end{equation}
if $n>0$. Furthermore, if $n>0$, then $w(\beta)=w(\alpha)$, and if $n<0$, then $w(\beta)=w(\alpha)-n$. If $n=0$, then (\ref{eq:Frobenius}) admits a solution $\beta\in\widetilde{\calE}_K$ with $w(\beta)=w(\alpha)$.
\item[(2)]
Let $\alpha\in\widetilde{\r}_K^{\bd,r}$.  If $n>0$, then (\ref{eq:solution-n-posi}) provides the unique solution $\beta\in\widetilde{\r}_K^{\bd}$ of (\ref{eq:Frobenius}). Furthermore, we have $\beta\in\widetilde{\r}_K^{\bd,qr}$, $w(\beta)=w(\alpha)$ and $w_r(\beta)\geq\min\{w(\alpha), w_r(\alpha)\}$. If $n=0$, then (\ref{eq:Frobenius}) admits a solution $\beta\in\widetilde{\r}_K^{\bd,qr}$ with $w(\beta)=w(\alpha), w_s(\beta)\geq w_s(\alpha)$ for any $0<s\leq r$.
\item[(3)]
If $\alpha\in\widetilde{\r}_K^{r}$ and $n\geq0$, then (\ref{eq:Frobenius}) admits a solution $\beta\in\widetilde{\r}_K^{qr}$ with  $w_r(\beta)\geq w_r(\alpha)-n$.
\end{enumerate}
\end{lemma}
\begin{proof}
We first prove $(1)$. Suppose $n\neq0$. Then it is clear that (\ref{eq:solution-n-nega}) and (\ref{eq:solution-n-posi}) provide a solution of (\ref{eq:Frobenius}). The uniqueness is obvious since $\varphi$ preserves $w$.
For $n=0$, since $k$ is strongly difference-closed, $k((u^{\mathbb{Q}}))$ is strongly difference-closed by Proposition \ref{prop:kuQ-difference-closed}. Therefore there exists $\beta\in\widetilde{\calE}_K$ with $w(\beta)=w(\alpha)$ such that $\varphi(\beta)-\beta=\alpha$
by Lemma \ref{lem:difference-closed}.

For $(2)$, the case $n>0$ follows from $(1)$ and Lemma \ref{lem:calE-bounded}. Now suppose $n=0$ and $\alpha=\sum_{i\in\mathbb{Q}}a_iu^i\in\widetilde{\r}_K^{\bd,r}$. Put $\alpha^+=\sum_{i>0}a_iu^i\in\widetilde{\r}_K^{\bd,r}$
and $\alpha^{-}=\sum_{i<0}a_iu^i\in\widetilde{\r}_K^{\bd,r}$.
Note that the infinite sum $\beta^+=-\sum_{m=0}^\infty \varphi^m(\alpha^+)$ is convergent in $\widetilde{\r}^r_K$ and has bounded coefficients; hence $\beta^+\in\widetilde{\r}_K^{\bd,r}$. It is clear that $w(\beta^+)\geq w(\alpha^+)\geq w(\alpha)$. Furthermore, since $\alpha^+$ have only positive powers of $u$, we get that $w_s(\beta^+)\geq w_s(\alpha^+)\geq w_s(\alpha)$ for any $0<s\leq r$. Choose $b_0\in K$ with $v(b_0)=v(a_0)$ such that $\varphi_K(b_0)- b_0=a_0$. Choose $\beta^-\in\widetilde{\calE}_K$ such that $w(\beta^{-})=w(\alpha^{-})$ and $\varphi(\beta^-)-\beta^-=\alpha^-$. We may suppose that $\beta^-$ only has negative powers of $u$ by dropping the nonnegative powers.
Write $\beta^-=\sum_{i<0}b_iu^i$. Then $v(b_i)\geq\min_{j\leq i}\{v(a_j)\}$ by Lemma \ref{lem:calE-bounded}, yielding $w(\beta^-)\geq w(\alpha^-)\geq w(\alpha)$ and $w_s(\beta^-)\geq w_s(\alpha^-)\geq w_s(\alpha)$ for any $0<s\leq r$. Then $\beta=\beta^++\beta^-+b_0$ is a desired solution and satisfies $w_s(\beta)\geq w_s(\alpha)$ for any $0<s\leq r$.

For $(3)$, if $n=0$ and $\alpha\in\widetilde{\r}_K^{r}$, write $\alpha=\sum_{i=1}^\infty \alpha_i$ such that $\alpha_i\in\widetilde{\r}_K^{\bd,r}, w_s(\alpha_i)\geq w_s(\alpha)$ for each $i\geq1$ and $w_s(\alpha_i)\rightarrow\infty$ as $i\rightarrow\infty$ for any $0<s\leq r$. For each $\alpha_i$, by $(2)$, choose $\beta_i\in\widetilde{\r}_K^{\bd,r}$ so that $\beta_i-\varphi(\beta_i)=\alpha_i$ and $w_s(\beta_i)\geq w_s(\alpha_i)$ for any $0<s\leq r$. Then $\sum_{i=1}^\infty \beta_i$ converges to a desired solution $\beta\in\widetilde{\r}_K^{r}$ of (\ref{eq:Frobenius}). Furthermore, since $\varphi(\beta)=\beta+\alpha\in\widetilde{\r}_K^r$, we get that $\beta\in\widetilde{\r}_K^{qr}$.

Now suppose $n>0$. For $\alpha=\sum_{i\in\mathbb{Q}}a_iu^i\in\widetilde{\r}_K^r$, write $\alpha=\alpha_1+\alpha_2$ where
\[
\alpha_1=\sum_{i\geq n/r}a_iu^i,\quad \alpha_2=\sum_{i<n/r}a_iu^i.
\]
Then a short computation shows that for any $m\geq0$ and $0<s\leq r$,
\[
w_s((\pi^{-n})^{\{m+1\}}\varphi^m(\alpha_1))\geq -n(m+1)+w_s(\alpha)+(q^m-1)ns/r,
\]
and
\[
w_s((\pi^{-n})^{\{-m\}}\varphi^{-m-1}(\alpha_2))\geq mn+w_s(\alpha)-(1-\frac{1}{q^{m+1}})ns/r.
\]
This implies that $w_s((\pi^{-n})^{\{m+1\}}\varphi^m(\alpha_1))$ and $w_s((\pi^{-n})^{\{-m\}}\varphi^{-m-1}(\alpha_2))$ approach to infinity as $m\ra \infty$. Furthermore, we get
\[
w_r((\pi^{-n})^{\{m+1\}}\varphi^m(\alpha_1))\geq -n(m+1)+w_r(\alpha)+(q^m-1)n\geq w_r(\alpha)-n,
\]
\[
w_r((\pi^{-n})^{\{-m\}}\varphi^{-m-1}(\alpha_2))\geq mn+w_r(\alpha)-(1-\frac{1}{q^{m+1}})n\geq w_r(\alpha)-n.
\]
Hence the sums
\[
\beta_1=\sum_{m=0}^\infty (\pi^{-n})^{\{m+1\}}\varphi^m(\alpha_1),\quad \beta_2=\sum_{m=0}^\infty (\pi^{-n})^{\{-m\}}\varphi^{-m-1}(\alpha_2)
\]
converge in $\widetilde{\r}_K^{r}$, and satisfy $w_r(\beta_i)\geq w_r(x)-n$ for $i=1,2$. It is clear that $\varphi(\beta_1)-\pi^n\beta_1=-\alpha_1$ and $\varphi(\beta_2)-\pi^n\beta_2=\alpha_2$. Hence $\beta=(\beta_2-\beta_1)$ is a solution of (\ref{eq:Frobenius}); the condition $\varphi(\beta)=\pi^n\beta+\alpha\in\widetilde{\r}_K^{r}$ implies $\beta\in \widetilde{\r}_K^{qr}$. Hence $\beta$ is a desired solution.
\end{proof}

\begin{proposition}\label{prop:V-d-lambda-H1}
Suppose that $k$ is strongly difference-closed. Let $\lambda_1,\lambda_2\in \widetilde{\r}_K^\times$. Then $\mathrm{Ext}^1_{\varphi,\widetilde{\r}_K}(V_{\lambda_1,d_1}, V_{\lambda_2,d_2})=0$ if $\frac{w(\lambda_2)}{d_2}\leq\frac{w(\lambda_1)}{d_1}$. In particular, we have $\mathrm{H}^{1}({\widetilde{\r}}_K(n))=0$ if $n\geq 0$.
\end{proposition}
\begin{proof}
If we equip $\widetilde{\r}_K$ with the endomorphism $\varphi^{d_1d_2}$, then $V_{\lambda_1,d_1}$ and $V_{\lambda_2,d_2}$ become direct sums of rank $1$ $\varphi^{d_1d_2}$-modules with slopes $-d_2w(\lambda_1)$ and $-d_1w(\lambda_2)$ respectively.
Hence  $V=V_{\lambda_1,d_1}^\vee\otimes V_{\lambda_2,d_2}$ is a direct sum of rank $1$ $\varphi^{d_1d_2}$-modules with slopes $d_2w(\lambda_1)-d_1w(\lambda_2)\geq0$. By Theorem \ref{thm:Robba-decomposition}, every rank 1 $\varphi^{d_1d_2}$-module is of the form $\widetilde{\r}_K(n)$ for some integer $n$. It thus follows from Lemma \ref{lem:Frobenius-equation}(3) that $V/(\varphi^{d_1d_2}-1)V=0$, yielding $V/(\varphi-1)V=0$. Hence $\mathrm{Ext}^1_{\varphi,\widetilde{\r}_K}(V_{\lambda_1,d_1}, V_{\lambda_2,d_2})=\mathrm{H}^1(V)=0$.
\end{proof}

\begin{proposition}\label{prop:DM-decomposition-robba}
Suppose that $k$ is strongly difference-closed. Then every $\varphi$-module over $\widetilde{\r}_K$ admits a Dieudonn\'e-Manin decomposition.
\end{proposition}
\begin{proof}
By Theorem \ref{thm:Robba-decomposition}, every semistable $\varphi$-module over $\widetilde{\r}_K$ admits a Dieudonn\'e-Manin decomposition. We therefore deduce from Proposition \ref{prop:V-d-lambda-H1} that HN filtrations for $\varphi$-modules over $\widetilde{\r}_K$ are split. This yields the desired result.
\end{proof}

\begin{proposition}\label{prop:filtraton-split}
Suppose that $k$ is strongly difference-closed. Let $0=M_0\subset M_1\subset\cdots\subset M_l=M$ be a semistable filtration of a $\varphi$-module $M$ over $\widetilde{\r}_K$. If the slope polygon of this filtration coincides with the HN-polygon of $M$, then the filtration splits.
\end{proposition}
\begin{proof}
The analogue of the proposition for $\varphi$-modules over $\Gamma_{\mathrm{an,con}}^{\mathrm{alg}}$ is \cite[Corollary 4.7.4]{Ke05} which is proved by using the formal properties of HN filtrations and the other two facts about $\varphi$-modules over $\Gamma_{\mathrm{an,con}}^{\mathrm{alg}}$. Namely, every $\varphi$-module over $\Gamma_{\mathrm{an,con}}^{\mathrm{alg}}$ admits a Dieudonn\'e-Manin decomposition, and
\[
\mathrm{Ext}^1_{\varphi,\Gamma_{\mathrm{an,con}}^{\mathrm{alg}}}(V_{\pi^{c_1},d_1}, V_{\pi^{c_2},d_2})=0
\]
if $\frac{c_2}{d_2}\leq\frac{c_1}{d_1}$. In our case the analogues of these two facts are Propositions \ref{prop:V-d-lambda-H1} and \ref{prop:DM-decomposition-robba}.  Therefore we can establish the proposition the same way as \cite[Corollary 4.7.4]{Ke05}.
\end{proof}

\subsection{Slope decomposition and reverse slope filtration}

\begin{proposition}\label{prop:generic-slope-filtration}
A $\varphi$-module $M$ over $\calE_K$ (resp. $\widetilde{\calE}_K$) is semistable of slope $s$ if and only if it is pure of spectral norm $p^{s}$ in the sense of difference modules. Every $\varphi$-module over $\calE_K$ (resp. $\widetilde{\calE}_K$) admits a unique HN filtration. Furthermore, for any $\varphi$-module $M$ over $\calE_K$ (resp. $\widetilde{\calE}_K$) and an  extension $L$ of $K$, the HN filtration of $M$, tensored up to $\widetilde{\calE}_L$, gives the HN filtration of $M\otimes_{\calE_K}\widetilde{\calE}_L$ (resp. $M\otimes_{\widetilde{\calE}_K}\widetilde{\calE}_L$).
\end{proposition}
\begin{proof}
Granting the first assertion, the second one then follows from \cite[Theorem 14.4.15]{Ke07}. Note that if $M$ is irreducible, then $M$ is clearly semistable. The ``if" part of the first assertion thus follows from the fact that any extension of two semistable $\varphi$-modules which have the same slope is still semistable with the same slope. Conversely, if $M$ is semistable of slope $s$, by \cite[Theorem 14.4.15]{Ke07}, there exists a unique filtration $0=M_0\subset
M_1\cdots\subset M_l=M$ so that each successive quotient $M_i/M_{i-1}$ is pure of spectral norm $p^{s_i}$ with $s_1>\cdots>s_l$. Since $\mu(M)\geq s_i$ for every $i$ and $\mu(M)$ is the weighted average of these $s_i$, we must have $l=1$, yielding that $M$ is pure of spectral norm $p^s$. The last assertion follows from \cite[Proposition 14.4.8]{Ke07}.
\end{proof}

\begin{definition}
For any $\varphi$-module $M$ over $\calE_K$ or $\widetilde{\calE}_K$, the slopes of the successive quotients and the slope polygon of the HN filtration of $M$ are called the \emph{slopes} and the \emph{HN-polygon} of $M$ respectively.
\end{definition}

\begin{proposition}\label{prop:split}
If $k$ is strongly difference-closed, then any exact sequence of $\varphi$-modules over $\widetilde{\calE}_K$ splits.
\end{proposition}
\begin{proof}
We first have that the residue field $k((u^{\mathbb{Q}}))$ of $\widetilde{\calE}_K$ is strongly difference-closed by Proposition \ref{prop:kuQ-difference-closed}. The proposition then follows immediately from \cite[Corollary 14.6.6]{Ke07}.
\end{proof}

\begin{proposition}\label{prop:slope-decomposition}
Suppose that $k$ is strongly difference-closed. Then for any $\varphi$-module $M$ over $\widetilde{\calE}_K$, its HN filtration splits uniquely, i.e. there exists a unique direct sum decomposition $M=\oplus_{1\leq i\leq l}M_{s_i}$ of $\varphi$-modules, in which each $M_{s_i}$ is a semistable submodule of slope $s_i$. Moreover, each $M_{s_i}$ admits a Dieudonn\'e-Manin decomposition. Furthermore, for each $V_{\lambda,d}$ in the decomposition, we may force $\lambda$ to be a power of $\pi$.
\end{proposition}

\begin{proof}
Note that the residue field $k((u^{\mathbb{Q}}))$ of $\widetilde{\calE}_K$ is strongly difference-closed by Proposition \ref{prop:kuQ-difference-closed}. The first assertion then follows from \cite[Theorem 14.4.13]{Ke07} and Proposition \ref{prop:generic-slope-filtration}, and the second assertion follows from \cite[Theorem 14.6.3]{Ke07}.
\end{proof}

\begin{corollary}\label{cor:additive}
If $0\ra M_1\ra M\ra M_2$ is an exact sequence of $\varphi$-modules over $\calE_K$, then the slope multiset of the HN filtration of $M$ is the union of the slope multisets of the HN filtrations of $M_1$ and $M_2$.
\end{corollary}
\begin{proof}
Let $L$ be an admissible extension of $K$ with strongly difference-closed residue field. By Proposition \ref{prop:split}, the exact sequence
\[
0\ra M_1\otimes_{\calE_K}\widetilde{\calE}_L\ra M\otimes_{\calE_K}\widetilde{\calE}_L\ra M_2\otimes_{\calE_K}\widetilde{\calE}_L\ra 0
\]
splits. We the deduce from by Proposition \ref{prop:slope-decomposition} that the slope multiset of the HN filtration of $M\otimes_{\calE_K}\widetilde{\calE}_L$ is the union of the slope multisets of the HN filtrations of
$M_1\otimes_{\calE_K}\widetilde{\calE}_L$ and $M_2\otimes_{\calE_K}\widetilde{\calE}_L$. It follows from Proposition \ref{prop:generic-slope-filtration} that the slope multisets of the HN filtrations of $M,M_1,M_2$ are equal to the slope multisets of the HN filtrations of $M\otimes_{\calE_K}\widetilde{\calE}_L, M_1\otimes_{\calE_K}\widetilde{\calE}_L, M_2\otimes_{\calE_K}\widetilde{\calE}_L$ respectively. The proposition then follows.
\end{proof}

\begin{definition}\label{def:slope-decomposition}
Suppose that $k$ is strongly difference-closed. For any $\varphi$-module $M$ over $\widetilde{\calE}_K$, we call the decomposition $M=\oplus_{1\leq i\leq l}M_{s_i}$ given by Proposition \ref{prop:slope-decomposition} the \emph{slope decomposition} of $M$. Moreover, suppose that $s_1>\cdots>s_l$, put $M^{\mathrm{rev}}_i=\oplus^{l}_{j=l-i+1}M_{s_j}$ for $1\leq i\leq l$. We call
\[
0=M^{\mathrm{rev}}_0\subset M_1^{\mathrm{rev}}\subset\cdots\subset M_l^{\mathrm{rev}}=M
\]
the \emph{reverse filtration} of $M$.
\end{definition}

\begin{lemma}\label{lem:lattice}
 Suppose that $k$ is strongly difference-closed. If $N$ is a $\varphi$-module over $\widetilde{\r}^{\bd}_K$ so that $M=N\otimes_{\widetilde{\calE}^\bd_K}\widetilde{\calE}_K$ has nonpositive slopes, then $N$ admits a $\varphi$-stable $\widetilde{\r}_K^\int$-lattice.
\end{lemma}
\begin{proof}
By proposition \ref{prop:slope-decomposition}, we know that $M$ is a direct sum of some $V_{d_i,\lambda_i}$'s, where each $w(\lambda_i)$ is nonnegative. We fix a standard basis for each $V_{d_i,\lambda_i}$. Then the $\OO_{\widetilde{\calE}_K}$-lattice $L$ of $M$ generated by these standard bases is stable under $\varphi$. Choose an $\widetilde{\r}^{\int}_K$-lattice $Q$ of $N$; then there exist integers $m\geq n$ such that
\[
\pi^m Q\otimes_{\widetilde{\r}^\int_K}\OO_{\widetilde{\calE}_K}\subseteq L\subseteq \pi^n Q\otimes_{\widetilde{\r}^\int_K}\OO_{\widetilde{\calE}_K}
\]
since $\widetilde{\r}^\int_K$ is the valuation ring of $\widetilde{\r}^\bd_K$. Let $P=L\cap N$. Note that $(\pi^i Q\otimes_{\widetilde{\r}^\int_K}\OO_{\widetilde{\calE}_K})\cap N=\pi^i Q$ for any $i\in\mathbb{Z}$ because $\pi^i\OO_{\widetilde{\calE}_K}\cap\widetilde{\r}^\int_K=\pi^i\widetilde{\r}^\int_K$. Hence $\pi^m Q\subseteq P\subseteq \pi^nQ$.  This yields that $P$ is an $\widetilde{\r}^\int_K$-lattice of $N$, and is stable under $\varphi$.
\end{proof}
\begin{lemma}\label{lem:descent}
 Suppose that $k$ is strongly difference-closed. Let $N$ be a $\varphi$-module over $\widetilde{\r}^{\bd}_K$ so that $M=N\otimes_{\widetilde{\calE}^\bd_K}\widetilde{\calE}_K$ has nonnegative slopes. Let $v\in M$ satisfying $\varphi(v)=\lambda v$ for some $\lambda\in\widetilde{\r}^{\int}_K$. Then $v\in N$.
\end{lemma}
\begin{proof}
Applying Lemma \ref{lem:lattice} to the dual of $N$, we may choose an $\widetilde{\r}_K^\int$-lattice $P$ of $N$ which is stable under $\varphi^{-1}$. Choose an $\widetilde{\r}^\int_K$-basis $\e=\{e_1,\dots,e_n\}$ of $P$, and let $F$ be the matrix of $\varphi$ under $\e$; then $F^{-1}$ has entries in $\widetilde{\r}_K^{\int}$. Write $v=\e\mathbf{v}$ for some column vector $\mathbf{v}$ over $\widetilde{\calE}_K$. Then $\varphi(v)=\lambda v$ implies $F\varphi(\mathbf{v})=\lambda\mathbf{v}$; hence $F^{-1}\lambda\mathbf{v}=\varphi(\mathbf{v})$. By \cite[Proposition 2.5.8]{Ke06}, we get that $\mathbf{v}$ has entries in $\widetilde{\r}_K^{\bd}$. So $v\in N$.
\end{proof}

The following proposition establishes the existence of de Jong's ``reverse filtration" (\cite[Proposition 5.8]{de98}) for $\varphi$-modules over extended bounded Robba rings with a relative Frobenius lift.

\begin{proposition}\label{prop:reverse-filtration}
Suppose that $k$ is strongly difference-closed. Then for any $\varphi$-module $N$ over $\widetilde{\r}^\bd_K$, the reverse filtration of $M=N\otimes_{\widetilde{\r}^\bd_K}\widetilde{\calE}_K$ descends uniquely to a filtration of $N$. Furthermore, if $M$ is semistable, then its Dieudonn\'e-Manin decompositions descend to Dieudonn\'e-Manin decompositions of $N$.
\end{proposition}

\begin{proof}
Let $0=M^\mathrm{rev}_0\subset M^{\mathrm{rev}}_1\cdots \subset M^\mathrm{rev}_l=M$ be the reverse filtration of $M$. Replacing $\varphi$ with $\varphi^a$ for a suitable positive integer $a$, we may suppose that the slopes of $M$ are integral. It then suffices to show that $M^{\mathrm{rev}}_1$ and its Dieudonn\'e-Manin decompositions descend to $N$.  By twisting, we reduce to the case that $\mu(M^{\mathrm{rev}}_1)=0$. Then the slopes of $M$ are all nonnegative. We fix a Dieudonn\'e-Manin decomposition of $M^\mathrm{rev}_1$. If $e$ is part of a standard basis of some $V_{\lambda,d}$ in this decomposition, then $\varphi^d(e)=\varphi^{i}(\lambda)e$ for some $0\leq i\leq d-1$. We then deduce that $e\in N$ by Lemma \ref{lem:descent}. Hence $V_{\lambda,d}$ descends to $N$. This implies that $M_1^{\mathrm{rev}}$ together with this Dieudonn\'e-Manin decomposition descends to a $\varphi$-submodule $N_1^{\mathrm{rev}}$ of $N$. It is clear that $N_1^{\mathrm{rev}}=M^{\mathrm{rev}}_1\cap N$, yielding the uniqueness.
\end{proof}

We call this filtration the \emph{reverse filtration} of $N$.

\begin{lemma}\label{lem:V-n-1}
Suppose that $k$ is strongly difference-closed. Let $\lambda\in K$, and let $d$ be a positive integer. If $n=v_K(\lambda)/d\in\mathbb{Z}$, then $V_{\lambda,d}$ is isomorphic to the direct sum of $d$ copies of $V_{\pi^n,1}$.
\end{lemma}
\begin{proof}
By \cite[Corollary 14.4.9]{Ke07}, the $\varphi$-module $V_{\lambda,d}\otimes_{K}V_{\pi^{-n},1}$ is pure of norm 1. Hence it is trivial by \cite[Proposition 14.4.16]{Ke07} and Lemma \ref{lem:pure-trivial}. This yields the lemma.
\end{proof}

\begin{lemma}\label{lem:numerical-basis-generic}
Suppose that $k$ is strongly difference-closed. Let $D$ be an $n\times n$ diagonal matrix such that all the diagonal entries are powers of $\pi$. If $F$ is an $n\times n$ matrix over $\widetilde{\calE}_K$ satisfying $w(FD^{-1}-I_n)>0$, then there exists an invertible $n\times n$ matrix $U$ over $\widetilde{\calE}_K$ with $w(U-I_n)>0$ and $U^{-1}F\varphi(U)=D$.
\end{lemma}
\begin{proof}
We follow the proof of \cite[Proposition 5.9]{Ke04}. Suppose $w(FD^{-1}-I_n)=c_0$. We will inductively construct a sequence of invertible $n\times n$ matrices $\{U_i\}_{i\in\mathbb{N}}$ over $\OO_{\widetilde{\calE}_K}$ satisfying
\[
\min\{w(U_{i+1}-U_i), w(U_i^{-1}F\varphi(U_i)D^{-1}-I_n)\}\geq (i+1)c_0
\]
as follows. Put $U_0=I_n$. Given $U_i$, by Lemma \ref{lem:Frobenius-equation}(1), there exists an $n\times n$ matrix $X_i$ over $\widetilde{\calE}_K$ with
\[
X_i-D\varphi(X_i)D^{-1}=U^{-1}_iF\varphi(U_i)D^{-1}-I_n
\]
and
\[
\min\{w(X_i),w(D\varphi(X_i)D^{-1})\}=w(U^{-1}_iF\varphi(U_i)D^{-1}-I_n).
\]
Put $U_{i+1}=U_i(I_n+X_i)$, then $w(U_{i+1}-U_i)=w(X_i)\geq(i+1)c_0$ and
\begin{equation*}
\begin{split}
U_{i+1}^{-1}F\varphi&(U_{i+1})D^{-1}-I_n=\\
&(I_n-X_i+X_i^2-\cdots)(I_n+(U^{-1}_iF\varphi(U_i)D^{-1}-I_n))(I_n+D\varphi(X_i)D^{-1})-I_n.
\end{split}
\end{equation*}
It follows that $w(U_{i+1}^{-1}F\varphi(U_{i+1})D^{-1}-I_n)\geq2(i+1)c_0\geq(i+2)c_0$. Then $U=\lim_{i\rightarrow\infty}U_i$ satisfies the desired properties.
\end{proof}

\begin{corollary}\label{cor:numerical-generic}
Let $M$ be a $\varphi$-module over $\calE_K$, and let $F$ be the matrix of $\varphi$ under some basis of $M$. Then there exists $N=N(F)>0$ such that for any $\varphi$-module $M'$ over $\calE_K$ with the same rank as $M$, if $M'$ has a basis under which the matrix $F'$ of $\varphi$ satisfies $w(F-F')\geq N$, then the HN-polygons of $M'$ and $M$ coincide.
\end{corollary}
\begin{proof}
Replacing $\varphi$ with $\varphi^a$ for some suitable positive integer $a$, we may suppose that the slopes of $M$ are integral. Choose an admissible extension $L$ of $K$ with a strongly difference-closed residue field. By Proposition \ref{prop:slope-decomposition} and Lemma \ref{lem:V-n-1}, there exists an invertible matrix $U$ over $\widetilde{\calE}_L$ such that $D=U^{-1}F\varphi(U)$ is a diagonal matrix with all diagonal entries being powers of $\pi$, and their valuations are the slopes of $M\otimes_{\calE_K}\widetilde{\calE}_L$. Let $N=1-w(U^{-1})-w(U)-w(D^{-1})$. If $w(F-F')\geq N$, then
\[
w(U^{-1}F'\varphi(U)D^{-1}-I_n)=w((U^{-1}(F'-F)\varphi(U)D^{-1}))\geq1.
\]
By Lemma \ref{lem:numerical-basis-generic}, we get that there exists an invertible matrix $U'$ over $\widetilde{\calE}_L$ such that $U'^{-1}F'\varphi(U')=D$. Hence the slopes of $M'\otimes_{\calE_K}\widetilde{\calE}_L$ are the same as $M\otimes_{\calE_K}\widetilde{\calE}_L$'s. This implies that the slopes of $M'$ are the same as $M$'s by Proposition \ref{prop:generic-slope-filtration}.
\end{proof}

\subsection{Comparison of HN-polygons}

%\begin{proposition}
%If $L$ is an admissible extension of $K$, then for any $\varphi$-module $M$ over $\calE_K$ (resp. $\widetilde{\calE}_K$, $\r_K$, $\widetilde{\r}_K$), the HN filtration of $M$, tensored up with $\calE_L$ (resp. $\widetilde{\calE}_L$, $\r_L$, $\widetilde{\r}_L$), gives the HN filtration of
%$M\otimes_{\calE_K}\calE_L$ (resp. $M\otimes_{\widetilde{\calE}_K}\widetilde{\calE}_L$, $M\otimes_{\r_K}\r_L$, $M\otimes_{\widetilde{\r}_K}\widetilde{\r}_L$).
%\end{proposition}
%\begin{proof}
%Note that a $\varphi$-module over $\calE_K$ or $\widetilde{\calE}_K$ is semistable if and only it is pure in the sense of differential modules
%by Proposition \ref{prop:generic-slope-filtration}. The desired result for $M$ over $\calE_K$ or $\widetilde{\calE}_K$ thus follows from the fact that the purity of difference modules is preserved by base change \cite[Proposition 14.4.8]{Ke07}. By Theorem \ref{thm:slope-filtration}, a $\varphi$-module over $\r_K$ is semistable if and only if it is pure in the sense of Definition \ref{def:pure}. It is clear that the purity of $\varphi$-modules over $\r_K$ is preserved by base change. This yields the proposition for $M$ over $\r_K$. then $M$
%\end{proof}

\begin{definition}
For a $\varphi$-module
$N$ over $\r_K^{\bd}$ (resp. $\widetilde{\r}_K^\bd$), the \emph{generic slope filtration} of $N$
is the HN filtration of $N\otimes_{\r_K^{\bd}}\calE_K$ (resp. $N\otimes_{\widetilde{\r}_K^\bd}\widetilde{\calE}_K$); the slope
polygon of the generic slope filtration is called the \emph{generic
HN-polygon} of $N$. The \emph{special slope filtration}
of $N$ is the HN filtration of $N\otimes_{\r_K^{\bd}}\r_K$ (resp. $N\otimes_{\widetilde{\r}_K^{\bd}}\widetilde{\r}_K$); the slope polygon of the special slope filtration is called the
\emph{special HN-polygon} of $N$.
\end{definition}

\begin{proposition}\label{prop:polygon-comparison}
If $N$ is a $\varphi$-module over $\r^{\bd}_K$ or $\widetilde{\r}^{\bd}_K$, the special HN-polygon
of $N$ lies above the generic HN-polygon of $N$ with the same endpoint.
\end{proposition}
\begin{proof}
By base change, it suffices to treat the case where $N$ is over $\widetilde{\r}^\bd_K$ and $k$ is strongly difference-closed.  Let $M=N\otimes_{\r_K^{\bd}}\widetilde{\r}_K$. Suppose that $0=N_0\subset N_1\cdots \subset N_l=N$ is the reverse filtration of $N$, and we denote by
\begin{equation}\label{eq:reverse-base-change}
0=M_0\subset M_1\cdots \subset M_l=M
 \end{equation}
 the base change of the reverse filtration. It follows from Proposition \ref{prop:reverse-filtration} that each quotient $N_i/N_{i-1}$ admits a Dieudonn\'e-Manin decomposition. This yields that each successive quotient $M_i/M_{i-1}$ is a pure $\varphi$-module over $\widetilde{\r}_K$; hence it is semistable by Proposition \ref{prop:pure-semistable-extended}. Hence (\ref{eq:reverse-base-change}) is a semistable filtration of $M$. We thus deduce the desired result by Proposition \ref{prop:HN-below}.
\end{proof}

\begin{lemma}\label{lem:H1-injective}
If $N$ is a $\varphi$-module over $\r_K^\bd$  (resp. $\widetilde{\r}_K^\bd$) whose generic slopes are all nonpositive, then the natural map $\mathrm{H}^1(N)\ra\mathrm{H}^1(N\otimes_{\r_K^\bd}\r_K)$ (resp. $\mathrm{H}^1(N)\ra\mathrm{H}^1(N\otimes_{\widetilde{\r}_K^\bd}\widetilde{\r}_K)$) is injective.
\end{lemma}

\begin{proof}
Let $M=N\otimes_{\r^\bd_K}\r_K$ (resp. $N\otimes_{\widetilde{\r}_K^\bd}\widetilde{\r}_K$). It suffices to show that for any $m\in M$, if $(\varphi-1)m\in N$, then $m\in N$. Note that $(N\otimes\widetilde{\r}_L^\bd)\cap M=N$ for any extension $L$ of $K$. Hence it suffices to show the lemma in the case when $N$ is over $\widetilde{\r}_K^\bd$ and $k$ is strongly difference-closed. Therefore by Lemma \ref{lem:lattice}, $N$ admits a $\varphi$-stable $\widetilde{\r}^\int_K$-lattice; let $\mathrm{e}=\{e_1,\dots,e_n\}$ be a basis of this lattice, and write $\varphi(\mathrm{e})=\mathrm{e} F$ for some $n\times n$ matrix $F$ over $\widetilde{\r}^\int_K$. Suppose that $m\in M$ satisfies $(\varphi-1)m\in N$. Write $m=\mathrm{e}\textbf{m}$ for some column vector $\textbf{m}$ over $\widetilde{\r}_K$. Then $F\varphi(\textbf{m})-\textbf{m}$ is over $\widetilde{\r}_K^\bd$. By \cite[Proposition 2.2.8]{Ke06}, we have that $\textbf{m}$ is over $\widetilde{\r}^\bd_K$. Hence $m\in N$.
\end{proof}

The following proposition generalizes \cite[Theorem 5.5.2]{Ke06} to the relative Frobenius lift case.
\begin{proposition}
Suppose that $k$ is strongly difference-closed. Let $N$ be a $\varphi$-module over $\r_K^\bd$ whose generic and special HN-polygons coincide. Then the HN filtrations of $N\otimes_{\r_K^\bd}\widetilde{\calE}_K$ and $N\otimes_{\r_K^\bd}\widetilde{\r}_K$, respectively, are obtained by base change from a filtration of $N$.
\end{proposition}
\begin{proof}
We follow the proof of \cite[Theorem 5.5.2]{Ke05}. It suffices to show that the first steps of the generic and special HN filtrations of $\widetilde{N}$ descend to $N$ and coincide. Let $0\subset\widetilde{N}_1\subset\cdots\subset\widetilde{N}_{l-1}\subset\widetilde{N}_l=\widetilde{N}$ be the reverse filtration of $\widetilde{N}=N\otimes_{\r_K^\bd}\widetilde{\r}_K^\bd$. As showed in the proof of Proposition \ref{prop:polygon-comparison}, the filtration
\[
0\subset\widetilde{N}_1\otimes_{\widetilde{\r}^\bd_K}\widetilde{\r}_K\subset\cdots\subset\widetilde{N}_{l-1}
\otimes_{\widetilde{\r}_K^\bd}\widetilde{\r}_K\subset
\widetilde{N}\otimes_{\widetilde{\r}_K^\bd}\widetilde{\r}_K
\]
is semistable. Since the slope polygon of this filtration is the same as the HN filtration of $N\otimes_{\r_K^\bd}\widetilde{\r}_K$, it is split by Proposition \ref{prop:filtraton-split}, yielding that the exact sequence
\[
0\rightarrow \widetilde{N}_{l-1}\rightarrow \widetilde{N}\rightarrow \widetilde{N}/\widetilde{N}_{l-1}\rightarrow0
\]
is split by Lemma \ref{lem:H1-injective}. Let $\widetilde{N}'$ be a $\varphi$-submodule of $\widetilde{N}$ lifting $\widetilde{N}/\widetilde{N}_{l-1}$. It follows that $\widetilde{N}'\otimes_{\widetilde{\r}_K^\bd}\widetilde{\calE}_K$ is isomorphic to the first step of the generic HN filtration. Thus they coincide by the uniqueness of HN filtration. Similarly, we also have that $\widetilde{N}'\otimes_{\widetilde{\r}_K^\bd}\widetilde{\r}_K$ coincides with the first step of the HN filtration of $N\otimes_{\r_K^\bd}\widetilde{\r}_K$. Hence both the first steps of the HN filtrations of $N\otimes_{\r_K^\bd}\widetilde{\calE}_K$ and $N\otimes_{\r_K^\bd}\widetilde{\r}_K$ descend to a $\varphi$-submodule $\widetilde{N}'$ of $\widetilde{N}$. To show that $\widetilde{N}'$ can be further descended to a $\varphi$-submodule of $N$, by \cite[Lemma 3.6.2]{Ke05}, it suffices to treat the case where $\rank \widetilde{N}'=1$. Choose a basis $\mathrm{e}=\{e_1,\dots,e_n\}$ of $N$. Let $v=\sum_{i=1}^n{a_ie_i}$ be a generator of $\widetilde{N}'$, and suppose that $a_1\neq0$. By Proposition \ref{prop:generic-slope-filtration}, the first step of the HN filtration of $N\otimes_{\r_K^\bd}\widetilde{\calE}_K$ descends to $N\otimes_{\r_K^\bd}\calE_K$. Hence $a_i/a_1\in\calE_K$ for $1\leq i\leq n$. Thus $a_i/a_1\in\calE_K\cap\widetilde{\r}_K^\bd=\r_K^\bd$ for each $i$, yielding $v/a_1\in N$.
\end{proof}

\section{Variation of slopes}
 In this section, we consider families of $\varphi$-modules (over $\r_K$) over affinoid spaces. All affinoid algebras are equipped with the spectral norm, and we fix a reduced affinoid space $M(A)$ over $\Q$ as the base. For any Banach algebra $B$ and $p$-adic field $L$, we assume that $|B|$ and $|L|$ are discrete, and set $B_L=B\widehat{\otimes}_{\Q}L$. We assume that $K$ is a $p$-adic field, and that $\varphi_K$ acts trivially on $\Q$. We adapt the normalization on the norm on $K$ to $|p|=p^{-1}$ to fit the standard norm on $\Q$. We also set $v(b)=\log_{|\pi|}|b|$ for any $b\in B$. Beware that $K$, which is the ``base field" of the fibers of the families, is irrelevant to $A$.

\subsection{Families of $\varphi$-modules}
\begin{definition}\label{def:Robba-ring-family}
 For any $\Q$-Banach algebra $B$, interval $I\subset(0,\infty]$, $s\in I$ and $r>0$, define the rings
\[
\calE_B, \widetilde{\calE}_B, \r_B^{\int,r}, \widetilde{\r}_B^{\int,r},\r_B^\int, \widetilde{\r}_B^\int,\r_B^{\bd,r}, \widetilde{\r}_B^{\bd,r}, \r_B^\bd,\widetilde{\r}_B^\bd, \r_B^I, \widetilde{\r}_B^I, \r_B^r,\widetilde{\r}_B^r, \r_B, \widetilde{\r}_B
\]
and $w_s, w$, and equip these rings with certain topologies by changing $K$
to $B$ in Definitions \ref{def:Robba-ring}, \ref{def:bounded-Robba-ring}, \ref{def:calE}, \ref{def:topology}, \ref{def:extended-Robba}, \ref{def:extended-bounded}, \ref{def:extended-calE}, \ref{def:extended-topology}.
We set $|\cdot|_s=|\pi|^{w(\cdot)}$ and  $|\cdot|=|\pi|^{w(\cdot)}$. We call $\r_B$ (resp. $\r_B^\bd$) the \emph{Robba ring over $B$} (resp. \emph{bounded Robba ring over $B$}) and $\widetilde{\r}_B$ (resp. $\widetilde{\r}_B^\bd$) the \emph{extended Robba ring over $B$} (resp. \emph{extended bounded Robba ring over $B$}).  Note that for general $B$ we only have
$w_s(fg)\geq w_s(f)+w_s(g), w(fg)\geq w(f)+w(g)$ and $|fg|_s\leq |f|_s|g|_s, |fg|\leq|f||g|.$
\end{definition}

\begin{proposition}\label{prop:limit-generalize}
For any $\Q$-Banach algebra $B$, we have
$\lim_{r\ra 0^+}w_r(f)=w(f)$
for any $f=\sum_{i\in\mathbb{Q}}a_iu^i\in\widetilde{\r}_B^\bd$.
\end{proposition}
\begin{proof}
The proof is similar to the proof of Proposition \ref{prop:w-bounded-additive}. Suppose $w(f)=v(a_{i_0})$ for some $i_0\in\mathbb{Q}$. For any $\epsilon>0$, set $r_0=\frac{\epsilon}{|2i_0|+1}$. We may suppose that $f\in\widetilde{\r}_B^{\bd,r_0}$ by shrinking $\epsilon$. It thus follows that for any $r\in(0,r_0]$, $w_r(f)\leq ri_0+v(a_{i_0})<w(f)+\epsilon/2$. On the other hand, choose some positive integer $N$ such that $r_0i+v(a_i)\geq w(f)$ for any $i\leq-N$. Let $r_1=\min\{r_0, \frac{\epsilon}{N}\}$. It follows that for $0<r\leq r_1$, if $i\leq-N$, then $ri+v(a_i)\geq r_0i+v(a_i)\geq w(f)$; if $i>-N$, then $ri+v(a_i)\geq w(f)-rN\geq w(f)-\epsilon$. We thus deduce that $|w_r(f)-w(f)|\leq\epsilon$ for any $r\in(0,r_1]$, proving the proposition.
\end{proof}

\begin{definition}
Let $L$ be a $p$-adic field, and let $V$ be an $L$-Banach space. A \emph{Schauder basis} of $V$ is a sequence $\{v_i\}_{i\in I}$ of elements of $V$ for a countable index set $I$ such that for every element $v\in V$ there exists a unique sequence $\{\lambda_i\}_{i\in I}$ of elements of $L$ so that
\[
v = \sum_{i\in I} \lambda_i v_i.
\]
It is further called an \emph{orthogonal basis} if
\[
|v|=\max_{i\in I}\{|\lambda_i||v_i|\}
\]
for any $v\in V$.
\end{definition}

\begin{lemma}\label{lem:orthonomal-basis}
Let $L$ be a $p$-adic field. If $V$ is an $L$-Banach space of countable type with $|V|$ discrete, then $V$ admits an orthogonal basis.
\end{lemma}
\begin{proof}
Since $|V|$ is discrete, there is a finite sequence $1\leq c_1<c_2<\cdots<c_m<p$ such that
\[
|V|=\{p^nc_j|n\in\mathbb{Z}, 1\leq j\leq m\}.
\]
Put $h=\min\{c_2/c_1,\dots,c_m/c_{m-1},p/c_m\}$. Choose some $h'\in(1,h)$. By \cite[2.7.2/3]{BGR84}, $V$ admits a
Schauder basis $\{v_i\}_{i\in I}$ such that
\[
h'|\sum_{i\in I} \lambda_i v_i|\geq \max_{i\in I}\{|\lambda_i||v_i|\}
\]
for any convergent sum $\sum_{i\in I} \lambda_i v_i$. However, since $|\sum_{i\in I} \lambda_i v_i|\leq\max_{i\in I}\{|\lambda_i||v_i|\}$, if they are not equal, we must have
\[
h|\sum_{i\in I} \lambda_i v_i|\leq \max_{i\in I}\{|\lambda_i||v_i|\};
\]
this yields a contradiction. Hence $|\sum_{i\in I} \lambda_i v_i|=\max_{i\in I}\{|\lambda_i||v_i|\}$, yielding that $\{v_i\}_{i\in I}$ is an orthogonal basis.
\end{proof}

\begin{remark}
Note that $A_L$ is an affinoid algebra over $L$. Hence it is of countable type as an $L$-Banach space, and $|A_L|$ is discrete. Thus Lemma \ref{lem:orthonomal-basis} implies that $A_L$ admits an orthogonal basis over $L$.
\end{remark}

\begin{lemma}\label{lem:comparison}
Let $L$ be a $p$-adic field, and let $B$ be a $\Q$-Banach algebra. For $R\in\{\calE,\r^{\bd,r},\r^I\}$  and $\widetilde{R}\in \{\widetilde{\calE},\widetilde{\r}^{\bd,r},\widetilde{\r}^I\}$ where $I\subset(0,\infty]$ is a closed interval, the natural maps
\[
i: B\otimes_{\Q}R_{L}\rightarrow R_{B_L}, \quad \widetilde{i}: B\otimes_{\Q}\widetilde{R}_{L}\rightarrow\widetilde{R}_{B_L}
\]
are isometric embeddings of $L$-Banach algebras.
For $R=\r^r$ and $\widetilde{R}=\widetilde{\r}^{r}$, the natural maps
\[
i: B\otimes_{\Q}R_{L}\rightarrow R_{B_L}, \quad \widetilde{i}: B\otimes_{\Q}\widetilde{R}_{L}\rightarrow\widetilde{R}_{B_L}
\]
are isometric embeddings of $L$-Fr\'echet spaces. Furthermore, $i$ always has dense image. Hence $i$ induces an isomorphism $B\widehat{\otimes}_{\Q}R_{L}\cong R_{B_L}$ for any $R\in\{\calE,\r^{\bd,r},\r^I,\r^r\}$, and $\widetilde{i}$ induces an isometric embedding $B\widehat{\otimes}_{\Q}\widetilde{R}_{L}\hookrightarrow\widetilde{R}_{B_L}$ for any $\widetilde{R}\in\{\widetilde{\calE}, \widetilde{\r}^{\bd,r},\widetilde{\r}^I, \widetilde{\r}^r\}$.
\end{lemma}
\begin{proof}
 For $R=\calE,\r^{\bd,r},\r^I$ and $\widetilde{R}=\widetilde{\calE},\widetilde{\r}^{\bd,r},\widetilde{\r}^I$, we denote by $|\cdot|_1$ the tensor product norms on
Banach algebras $B\widehat{\otimes}_{\Q}R_{L}$ and $B\widehat{\otimes}_{\Q}\widetilde{R}_{L}$, and denote by $|\cdot|_{2}$ the norms of the Banach algebras $R_{B}$ and $\widetilde{R}_B$.
Fix some $s\in(0, r]$. For $R=\r^r$ and $\widetilde{R}=\widetilde{\r}^r$, we denote $|f|_s$ by $|f|$ for any $f\in R_L$ and $f\in\widetilde{R}_L$. We denote by $|\cdot|_1$ the tensor products of the norm on $B$ and $|\cdot|_s$ on $R_{L}$ and $\widetilde{R}_{L}$,
and denote by $|\cdot|_{2}$ the norms $|\cdot|_s$ on $R_{B_L}$ and $\widetilde{R}_{B_L}$.

For any $f\in B\otimes_{\Q}\widetilde{R}_{L}$, if we write $f=\sum_{j=1}^nb_j\otimes f_j$, then it is clear that
\[
|\widetilde{i}(f)|_{2}=|\sum_{j=1}^nb_j\otimes f_j|\leq\max\{|b_j||f_j|\};
\]
hence $|\widetilde{i}(f)|_2\leq |f|_1$ by the definition of tensor product norms. On the other hand, let $V$ be the $\Q$-subspace of $B$ generated by $b_1,\dots,b_n$. By Lemma \ref{lem:orthonomal-basis}, $V$ admits an orthogonal basis $\{v_1,\dots, v_m\}$. We may rewrite $f=\sum_{j=1}^m v_j\otimes f_j'$ for some $f_j'\in \widetilde{R}_{L}$. For any $i\in\mathbb{Q}$, let $c_i$ and $c_{ij}$ be the $i$-th coefficients of $f$ and $f_j'$; then $c_i=\sum_{j=1}^mv_jc_{ij}$. Hence $|c_i|=\max\{|v_j||c_{ij}|\}$. This implies that $|\widetilde{i}(f)|_2=\max\{|v_j||f_j'|\}$. This yields $|\widetilde{i}(f)|_2\geq|f|_1$. Hence $|\widetilde{i}(f)|_2=|f|_1$. The proof for $i$ is similar. The rest of the lemma is obvious.
\end{proof}

Henceforth for any $\Q$-Banach algebra $B$ and $\widetilde{R}\in\{\widetilde{\calE},\widetilde{\r}^{\bd,r},
\widetilde{\r}^{I},
\widetilde{\r}^r\}$, we view
$B\widehat{\otimes}_{\Q}\widetilde{R}_{L}$ as a subalgebra of $\widetilde{R}_{B_L}$ via $\widetilde{i}$

\begin{definition}
For any $\Q$-Banach algebra $B$, $p$-adic field $L$ and $\widetilde{R}\in\{\widetilde{\r}^{\bd}, \widetilde{\r}\}$, we set
\[
B\widehat{\otimes}_{\Q}\widetilde{R}_L=\cup_{r>0}B\widehat{\otimes}_{\Q}\widetilde{R}^r_L.
\]
which is a subalgebra of $\widetilde{R}_{B_L}$.
\end{definition}

\begin{lemma}\label{lem:intersection}
Let $B$ be a $\Q$-Banach algebra of countable type, and let $L$ be a $p$-adic field.  Suppose that $S$ is a closed subspace of $\widetilde{\calE}_{L}$ and put $S'=S\cap\widetilde{\r}_{L}^{\bd,r}$, which is a closed subspace of $\widetilde{\r}_{L}^{\bd,r}$. Then
$$(B\widehat{\otimes}_{\Q}S)\cap\widetilde{\r}^{\bd,r}_{B_L}=B\widehat{\otimes}_{\Q}S'.$$
\end{lemma}
\begin{proof}
We only need to show $(B\widehat{\otimes}_{\Q}S)\cap\widetilde{\r}^{\bd,r}_B\subseteq B\widehat{\otimes}_{\Q}S'.$ Since $B$ is of countable type, by Lemma \ref{lem:orthonomal-basis}, $B$ admits an orthogonal basis $\{v_j\}_{j\in J}$ over $\Q$; then $\{v_j\}_{j\in J}$ is also an orthogonal basis of $B\widehat{\otimes}_{\Q}L$ over $L$.

Now suppose $f=\sum_{i\in\mathbb{Q}}a_iu^i\in(B\widehat{\otimes}_{\Q}S)\cap\widetilde{\r}^{\bd,r}_{B_L}$. We may write $f$ as a convergent sum $f=\sum_{j\in \mathbb{N}} v_j\otimes f_j$ in $B\widehat{\otimes}_{\Q}S$ where each $f_j\in S$. Since $f\in \widetilde{\r}^{\bd,r}_{B_L}$, it follows that each $f_j\in \widetilde{\r}^{\bd,r}_L$ and satisfies $|v_j||f_j|\leq |f|, |v_j||f_j|_r\leq |f|_r$. Hence all $f_j$ belong to $S'=S\cap\widetilde{\r}_{L}^{\bd,r}$. It remains to show that the sum $\sum_{j\in \mathbb{N}} v_j\otimes f_j$ is convergent in $B\widehat{\otimes}_{\Q}S'$. For any $\epsilon>0$, choose $N<0$ so that $\max\{|a_i|,|a_iu^i|_r)\}<\epsilon$ if $i< N$. Choose $m\in\mathbb{N}$ so that $|v_j\otimes f_j|<\epsilon |\pi|^{-Nr}$ if $j\geq m$. We claim that
\[
\max\{|v_j\otimes f_j|, |v_j\otimes f_j|_r\}<\epsilon
\]
for each $j\geq m$. In fact, for any $j\geq m$, if we write $v_j\otimes f_j=\sum_{i\in\mathbb{Q}}b_{ji}u^i$ where $b_{ji}\in B$, then
\[
\max\{|b_{ji}|, |b_{ji}u^i|_r\}\leq \max\{|a_i|,|a_iu^i|_r\}<\epsilon
\]
if $i<N$, and
\[
\max\{|b_{ji}|, |b_{ji}u^i|_r\}< \max\{\epsilon |\pi|^{-Nr}, \epsilon |\pi|^{(i-N)r}\}\leq \epsilon
\]
if $i\geq N$. This yields the claim. Hence $f\in B\widehat{\otimes}_{\Q}S'$.
\end{proof}

\begin{definition}\label{def:phi-action-B}
Let $\varphi_A:A_K\rightarrow A_K$ be the continuous extension of $\mathrm{id}\otimes\varphi_K$ on $A\otimes_{\Q}K$. We set the $\varphi$-action on $\r_{A_K}$ as the continuous extension of $\mathrm{id}\otimes\varphi$ on $A\otimes_{\Q}\r_K$. We set the $\varphi$-action on $\widetilde{\r}_{A_K}$ as
\[
\varphi(\sum_{i\in\mathbb{Q}}a_iu^i)=\sum_{i\in\mathbb{Q}}\varphi_A(a_i)u^{qi}.
\]
\end{definition}

\begin{remark}
The embedding $\tau_K:\r_K\rightarrow\widetilde{\r}_K$ induces an embedding
\[
\tau_A:\r_{A_K}=\cup_{r>0}A\widehat{\otimes}_{\Q}\r^r_K\rightarrow A\widehat{\otimes}_{\Q}\widetilde{\r}_{K}=\cup_{r>0}A\widehat{\otimes}_{\Q}\widetilde{\r}^r_K
\]
by tensoring with the identity on $A$ and taking completion; then $\tau_A$ is $\varphi$-equivariant because $\tau_K$ is $\varphi$-equivariant. By Lemma \ref{lem:comparison}, $\tau_A$ further induces a $\varphi$-equivariant embedding $\r_{A_K}\rightarrow\widetilde{\r}_{A_K}$ which we again denote by $\tau_A$.
\end{remark}

\begin{definition}\label{def:vector-bundle}
By a \emph{vector bundle} over $\r^r_{A_K}\cong A_K\widehat{\otimes}_K\r^r_{K}$,
we mean a locally free coherent sheaf over the product
of the annulus $0 < v_p(T) \leq r$ over $K$ with $M(A_K)$
in the category of rigid analytic spaces over $K$.
In case $A_K$ is disconnected, we require that the rank be constant.
By a vector bundle over $\r_{A_K}$, we will mean an object in the direct limit
as $r \to 0$ of the categories of vector bundles over
$\r^r_{A_K}$. For any morphism of affinoid algebras $A\ra B$ and a vector bundle $M_A$ over $\r_{A_K}$, we denote by $M_A\otimes_{\r_{A_K}}\r_{B_K}$ the base change of $M_A$ to a vector bundle over $\r_{B_K}$.
\end{definition}

\begin{definition}\label{def:phi-module-family}
By \emph{a family of $\varphi$-modules} over $\r_{A_K}^\bd$ (resp. $\r_{A_K}$), we mean a finite locally free module $N_A$ over $\r_{A_K}^\bd$ (resp. a vector bundle $M_A$ over $\r_{A_K}$) equipped with an isomorphism $\varphi^*N_A\rightarrow N_A$ (resp. $\varphi^*M_A\rightarrow M_A$), viewed as a semilinear action $\varphi$ on $N_A$ (resp. $M_A$). In case $A_K$ is disconnected, we require that the rank be constant.
\end{definition}

\begin{remark}\label{rem:vector-bundle}
For $A=\Q$, every vector bundle over $\r_K$ is represented by a finite free $\r_K$-module by the B\'ezout property of $\r_K$ (\cite[Theorem 2.8.4]{Ke05}). Hence the category of families of $\varphi$-modules over $\r_K$ coincides with the category of $\varphi$-modules over $\r_K$. For general $A$, it is only known that any family of $\varphi$-modules over $\r_{A_K}$ is $A$-locally free (Corollary \ref{cor:base-lifting}).
\end{remark}

\begin{definition}
Let $N_A$ (resp. $M_A$) be a family of $\varphi$-modules over $\r_{A_K}^\bd$ (resp. $\r_{A_K}$). For any $x\in M(A)$,  $N_A$ (resp. $M_A$) specializes to a $\varphi$-module
$$N_x=N_A\otimes_{\r^\bd_{A_K}}(k(x)\otimes_{\Q}\r^\bd_{K})$$ over $k(x)\otimes_{\Q}\r^\bd_{K}$ (resp.
$M_x=M_A\otimes_{\r_{A_K}}(k(x)\otimes_{\Q}\r_{K})$ over $k(x)\otimes_{\Q}\r_{K}$). We denote by $p_x$ the natural projection map $N_A\rightarrow N_x$ (resp. $M_A\rightarrow M_x$).
\end{definition}

\begin{definition}
For a family of $\varphi$-modules $M_A$ over $\r_{A_K}$, a \emph{model} of $M_A$ is a sub-family of $\varphi$-modules $N_A$ over
$\r_{A_K}^{\bd}$ such that $N_A\otimes_{\r_{A_K}^{\bd}}\r_{A_K}=M_A$.
\end{definition}

\begin{definition}
Let $N_A$ be a family of $\varphi$-modules over $\r^\bd_{A_K}$. For $c,d\in\mathbb{Z}$ with $d>0$, a  \emph{$(c,d)$-pure model} of $N_A$ is a finite locally free sub-$\r^\int_{A_K}$-module $N_A'$ of $N_A$ with $N_A'\otimes_{\r^{\int}_{A_K}}\r^\bd_{A_K}=N_A$ so that the $\varphi$-action on $N'_A$ induces an isomorphism $\pi^{c}(\varphi^d)^*N'_A\cong N'_A$. For a family of $\varphi$-modules $M_A$ over $\r_{A_K}$, a \emph{$(c,d)$-pure model} of $M_A$ is a $(c,d)$-pure model of a model of $M_A$. For $s\in\mathbb{Q}$, we say that $N_A$ (resp. $M_A$) is \emph{globally pure of slope $s$} if $N_A$ (resp. $M_A$) admits a $(c,d)$-pure model for some (hence any) $c,d\in\mathbb{Z}$ with $d>0$ and $s=c/d$. If $s=0$, we also say that $N_A$ (resp. $M_A$) is \emph{globally \'etale}, and a $(0,1)$-pure model is also called an \emph{\'etale model}.
\end{definition}

\begin{proposition}\label{prop:HN-polygon-multifield}
Let $M_A$ (resp. $N_A$) be a family of $\varphi$-modules over $\r_{A_K}$ (resp. $\r^\bd_{A_K}$), and let $x\in M(A)$. Suppose that
\[
k(x)\otimes_{\Q}K\cong\oplus_{i=1}^nK_i
\]
where each $K_i$ is a finite field extension of $K$. Then the following are true.
\begin{enumerate}
\item[(1)]
The induced $\varphi$-action on each $K_i$ is an automorphism.
\item[(2)]
Let $M_{x,i}=M_x\otimes_{k(x)\otimes_{\Q}\r_K}\r_{K_i}$ (resp. $N_{x,i}=N_x\otimes_{k(x)\otimes_{\Q}\r^\bd_K}\r^\bd_{K_i}$) for $1\leq i\leq n$. Then the HN-polygons of all $M_{x,i}$ (generic HN-polygons of all $N_{x,i}$) coincide.
\end{enumerate}
\end{proposition}
\begin{proof}
Since the $\varphi$-action is an automorphism on $k(x)\otimes_{\Q}K$, it is an automorphism on each $K_i$. This yields $(1)$.
By Propositions \ref{prop:robba-extension} (resp. Proposition \ref{prop:generic-slope-filtration}), we see that HN-polygons of $\varphi$-modules over Robba rings (generic HN-polygons of $\varphi$-modules over bounded Robba rings) are stable under base change. By passing to normal closure of the field extension $k(x)/\Q$, we may suppose that $k(x)$ is Galois over $\Q$. In this case, $\mathrm{Gal}(k(x)/\Q)$ acts transitively on the set $\{K_i\}_{1\leq i\leq n}$; hence it acts transitively on $\{M_{x,i}\}_{1\leq i\leq n}$ (resp. $\{N_{x,i}\}_{1\leq i\leq n}$). Furthermore, this action commutes with $\varphi$. This implies that all $M_{x,i}$ (resp. $N_{x,i}$) have the same HN-polygon (generic HN-polygon), yielding (2).
\end{proof}

In the situation of Proposition \ref{prop:HN-polygon-multifield},  it is clear that $M_x$ (resp. $N_x$) is isomorphic to the direct sum of all $M_{x,i}$ (resp. $N_{x,i}$). We call each $M_{x,i}$ (resp. $N_{x,i}$) a \emph{component} of $M_x$ (resp. $N_x$). We set the \emph{slopes} and \emph{HN-polygon} of $M_x$ (resp. \emph{generic slopes} and \emph{generic HN-polygon} of $N_x$) as the slopes and HN-polygon of $M_{x,i}$ (resp. generic slopes and generic HN-polygon of $N_{x,i}$). We set the \emph{HN filtration} of $M_x$ (resp. \emph{generic HN filtration} of $N_x$) as the direct sum of the HN filtrations of all $M_{x,i}$ (generic HN filtrations of all $N_{x,i}$).

\begin{definition}
Let $M_A$ be a family of $\varphi$-modules over $\r_{A_K}$, and let $N_A$ be a model of it. We call $N_A$ a \emph{good model} if
for every $x\in M(A)$, the generic and special HN-polygons of $N_x$ coincide, i.e. the generic HN-polygon of $N_x$ coincides with the HN-polygon of $M_x$.
\end{definition}

\subsection{Semicontinuity of HN-polygons}
\begin{convention}
Let $r_\varphi$ be as in Lemma \ref{lemma:degree-q}. It follows that for $0<r< r_{\varphi}$ and $a\in\r_K$, if $\varphi(a)\in \r_K^{r/q}$, then $a\in \r_K^r$. Furthermore, by Remark \ref{rem:phi-r-to-qr}, we may shrink $r_{\varphi}$ so that $\varphi$ maps $\r_K^{r}$ to $\r^{r/q}_K$ for $0<r< r_{\varphi}$. Hence for $0<r< r_\varphi$, we have that $\varphi(a)\in\r_K^{r/q}$ if and only if $a\in\r_K^r$, and that $w_{r/q}(\varphi(a))=w_r(a)$ for any $a\in \r_K^r$.
\end{convention}

\begin{proposition}\label{prop:shrinking}
For any $\Q$-Banach algebra $S$ and $x\in M(A)$, the natural projection map
\[
\rho_x:A\widehat{\otimes}_{\Q}S\rightarrow k(x)\otimes_{\Q}S
\]
is surjective and $\ker(\rho_x)=\mathfrak{m}_x(A\widehat{\otimes}_{\Q}S)$ where $\mathfrak{m}_x$ is the maximal ideal of $A$ corresponding to $x$.
Furthermore, for any $\lambda>0$, there exists a Weierstrass subdomain $M(B)$ of $M(A)$ containing $x$ such that if $f\in\ker(\rho_x)$, then the norm of $f$ in $B\widehat{\otimes}_{\Q}S$ is no more than $\lambda$ times the norm of $f$ in $A\widehat{\otimes}_{\Q}S$.
\end{proposition}
\begin{proof}
By Hahn-Banach theorem for Banach spaces over discretely valued fields (\cite[Proposition 10.5]{S02}), the exact sequence
\[
0\rightarrow\mathfrak{m}_x\rightarrow A\rightarrow k(x)\rightarrow0
\]
splits as $\Q$-Banach spaces. This yields the exact sequence
\[
0\rightarrow\mathfrak{m}_x\widehat{\otimes}_{\Q}S\rightarrow A\widehat{\otimes}_{\Q}S
\rightarrow k(x)\otimes_{\Q}S\rightarrow0.
\]
This shows that $\rho_x$ is surjective.

Choose a finite set of generators $b_1,\dots,b_m$ of $\mathfrak{m}_x$ as an $A$-module. By the open mapping theorem for Banach spaces over discretely valued fields (\cite[Proposition~8.6]{S02}), the surjective map of $\Q$-Banach spaces $A^m\rightarrow \mathfrak{m}_x$ defined by $(a_1,\dots,a_m)\mapsto\sum_{i=1}^ma_ib_i$ is open. Hence there exists $c>0$ such that for any $a\in\mathfrak{m}_x$, there exist $a_1,\dots,a_m\in A$ with $|a_i|\leq c|a|$ such that $a=\sum_{i=1}^ma_ib_i$. Choose some nonzero $z\in\Q$ with $|z|=\lambda'\leq\lambda/c$. Set
\[
B=A\langle X_1,\dots,X_m\rangle/(zX_1-b_1,\dots,zX_m-b_m),
\]
then $M(B)=\{y\in M(A)||b_i(y)|\leq \lambda', 1\leq i\leq m\}$ is a Weierstrass subdomain containing $x$.
Let $\{v_i\}_{i\in I}$ be an orthogonal basis of $\mathfrak{m}_x$ over $\Q$.
Now if $f\in\ker(\rho_x)$, write $f=\sum_{i\in I}v_i\otimes g_i$ with $g_i\in S$; then $|f|=\max_{i\in I}\{|v_i||g_i|\}$. For each $i\in I$, choose $a_{1i},\dots,a_{mi}\in A$ so that $\sum_{j=1}^ma_{ji}b_j=v_i$ with $|a_{ji}|\leq c|v_i|$ for $1\leq j\leq m$. Put $f_j=\sum_{i\in I}a_{ji}g_i$ for $1\leq j\leq m$. It then follows that
\[
|f_j|\leq\max_{i\in I}\{|a_{ji}||g_i|\}\leq c\max_{i\in I}\{|v_i||g_i|\}=c|f|
\]
and $f=\sum_{j=1}^mb_jf_j$. This implies that $f\in\mathfrak{m}_x(A\widehat{\otimes}_{\Q}S)$. Furthermore, since the norms of $b_j$'s in $B$ are no more than $\lambda'$, the norm of $f$ in $B\widehat{\otimes}_{\Q}S$ is no more than $\lambda'c$, which is no more than $\lambda$ times the norm of $f$ in $A\widehat{\otimes}_{\Q}S$.
\end{proof}

\begin{corollary}\label{cor:lifting}
Let $S$ be a $\Q$-Banach algebra. Let $x\in M(A)$, and let $F_x$ be an invertible matrix over $k(x)\otimes_{\Q}S$. Let $F$ be a matrix over  $A\widehat{\otimes}_{\Q}S$ lifting $F_x$. Then there exists a Weierstrass subdomain $M(B)$ of $M(A)$ containing $x$ such that $F$ is invertible over $B\widehat{\otimes}_{\Q}S$.
\end{corollary}
\begin{proof}
Using the first part of Proposition \ref{prop:shrinking}, we lift $F_x^{-1}$ to a matrix $F'$ over $A\widehat{\otimes}_{\Q}S$. Note that $F'F-I$ vanishes at $x$. It therefore follows from the second part of Proposition \ref{prop:shrinking} that there exists a Weierstrass subdomain $M(B)$ containing $x$ such that the norm of $F'F-I$, viewed as a matrix over $B\widehat{\otimes}_{\Q}S$, is less than $1$. This implies that $F'F$ is invertible over $B\widehat{\otimes}_{\Q}S$; hence $F$ is invertible over $B\widehat{\otimes}_{\Q}S$.
\end{proof}

\begin{lemma}\label{lem:basis-extension}
Let $0<r<r_{\varphi}$, and let $M_A^r$ be a vector bundle over $\r^r_{A_K}$ equipped with an isomorphism $\varphi^*M_A^r\cong M_A^r\otimes_{\r_{A_K}^{r}}\r_{A_K}^{r/q}$ as vector bundles over
$\r_{A_K}^{r/q}$. Suppose that there exists a basis $e_1,\dots,e_n$ of $M_A^r\otimes_{\r^r_{A_K}}\r^{[r/q,r]}_{A_K}$ over
$\r^{[r/q,r]}_{A_K}$ on which $\varphi$ acts via an invertible matrix $F$ over $\r^{r/q}_{A_K}$, then $e_1,\dots,e_n$ extends to a basis of $M_A^r$.
\end{lemma}
\begin{proof}
We will proceed by induction on $l$ to show that one can extend $e_1,\dots,e_n$ to a basis of $M_A^r\otimes_{\r^r_{A_K}}\r^{[r/q^{l},r]}_{A_K}$
for each $l\geq1$. The initial case is already known by assumption. Suppose that the claim is true for some $l-1\geq1$. Write $\mathrm{e}=(e_1,\dots,e_n)$. Since $\varphi(\mathrm{e})$ is equal to $\mathrm{e}F$ in $M_A^r\otimes_{\r^r_{A_K}}\r^{[r/q,r/q]}_{A_K}$, they are equal in $M_A^r\otimes_{\r^r_{A_K}}\r^{[r/q^{l-1},r/q]}_{A_K}$ a priori.  Then using the relation $\mathrm{e}=\varphi(\mathrm{e})F^{-1}$, we extend $\mathrm{e}$ to $M_A^r\otimes_{\r^r_{A_K}}\r^{[r/q^{l},r]}_{A_K}$ by gluing $\mathrm{e}$ and $\varphi(\mathrm{e})F^{-1}$. It remains to prove that $\mathrm{e}$ generates $M_A^r\otimes_{\r^r_{A_K}}\r^{[r/q^{l},r]}_{A_K}$. Let $M'$ be the coherent subsheaf of $M_A^r\otimes_{\r^r_{A_K}}\r^{[r/q^{l},r]}_{A_K}$ generated by
$\mathrm{e}$. Note that $\varphi(\mathrm{e})$ is a basis of $M_A^r\otimes_{\r^r_{A_K}}\r^{[r/q^{l},r/q]}_{A_K}$ by the isomorphism $\varphi^*M_A^r\cong M_A^r\otimes_{\r_{A_K}^{r}}\r_{A_K}^{r/q}$. It therefore follows
$$M'|_{M(\r_{A_K}^{[r/q^{l-1},r]})}=M_A^r\otimes_{\r^r_{A_K}}\r^{[r/q^{l-1},r]}_{A_K},
\quad M'|_{M(\r_{A_K}^{[r/q^{l},r/q]})}=M_A^r\otimes_{\r^r_{A_K}}\r^{[r/q^l,r/q]}_{A_K}.$$
Hence $M'=M_A^r\otimes_{\r^r_{A_K}}\r^{[r/q^{l},r]}_{A_K}$.
\end{proof}

\begin{lemma}\label{lem:locally-free}
Let $M_A$ be a family of $\varphi$-modules over $\r_{A_K}$ such that it is represented by a vector bundle $M_A^r$ over $\r_{A_K}^r$ for some $0<r<r_{\varphi}$. Let $x\in M(A)$, and let $\mathrm{e}_x$ be a basis of
\[
M_x^{[r/q,r]}=M_A^r\otimes_{\r_{A_K}^r}(k(x)\otimes_{\Q}\r^{[r/q,r]}_K)
\]
over $k(x)\otimes_{\Q}\r^{[r/q,r]}_K$. Suppose that $\e$ is a lift of $\e_x$ in $M_A^{[r/q,r]}=M^r_A\otimes_{\r^r_{A_K}}\r_{A_K}^{[r/q,r]}$. Then there exists a Weierstrass subdomain $M(B)$ containing $x$ such that $\e$ is a basis of $M^{[r/q,r]}_B=M^r_A\otimes_{\r^r_{A_K}}\r^{[r/q,r]}_{B_K}$ over $\r^{[r/q,r]}_{B_K}$.
\end{lemma}
\begin{proof}
Since $M_A^{[r/q,r]}$ is a coherent sheaf over $M(\r_{A_K}^{[r/q,r]})$, we choose a finite set of generators $\mathrm{v}=(v_1,\dots,v_m)$ of it.
We lift the transformation matrix between the image of $\mathrm{v}$ in $M_x^{[r/q,r]}$ and $\e_x$ to a matrix $U$ over $\r_{A_K}^{[r/q,r]}$. It is clear that the image of $\e U-\mathrm{v}$ in $M^{[r/q,r]}_x$ vanishes. Since $M_A^{[r/q,r]}$ is a finite locally free $\r_{A_K}^{[r/q,r]}$-module, by Proposition \ref{prop:shrinking}, we deduce that $\e U-\mathrm{v}\in \mathfrak{m}_xM_A^{[r/q,r]}$; thus there is a square matrix $W$ over $\mathfrak{m}_x\r_{A_K}^{[r/q,r]}$ such that $\e U-\mathrm{v}=\mathrm{v}W$. By Proposition \ref{prop:shrinking}, we choose a Weierstrass subdomain $M(B)$ containing $x$ such that $\min\{w_{r/q}(W),w_r(W)\}>0$ over $\r^{[r/q,r]}_{B_K}$. This implies that $I+W$ is invertible over $\r^{[r/q,r]}_{B_K}$. Hence $\e U(I+W)^{-1}=\mathrm{v}$, yielding that $\e$ generates $M_B^{[r/q,r]}$. Since the number of entries of $\e$ is equal to the rank of  $M_B^{[r/q,r]}$, we get that $\e$ is a basis of  $M_B^{[r/q,r]}$ over $\r^{[r/q,r]}_{B_K}$.
\end{proof}

The following lemma is based on \cite[Lemma 6.1.1]{Ke05}.
\begin{lemma} \label{lem:modification}
For $r\in(0,r_\varphi/q)$, let $D$ be an
invertible $n\times n$ matrix over $\r_{A_K}^{[r, r]}$, and put
$h=-w_r(D)-w_r(D^{-1})$. Let $F$ be an $n\times n$ matrix over
$\r_{A_K}^{[r, r]}$ such that $w_r(FD^{-1}-I_n)\geq c+h/(q-1)$ for a
positive number $c$. Then for any positive integer $k$ satisfying
$2(q-1)k\leq c$, there exists an invertible $n\times n$ matrix
$U$ over $\r_{A_K}^{[r, qr]}$ such that $U^{-1}F\varphi(U)D^{-1}-I_n$ has
entries in $\pi^k\r_{A_K}^{\int, r}$ and
$w_r(U^{-1}F\varphi(U)D^{-1}-I_n)\geq c+h/(q-1)$.
\end{lemma}
\begin{proof}
For any $i\in\{v(a)|a\in A_K\}$, $r>0$, $f=\sum^{+\infty}_{j=-\infty}a_j
T^j\in\r_{A_K}$, we set $v_i(f)=\min\{j:v(a_j)\leq i\}$ and
$v_{i,r}(f)=rv_i(f)+i$. (In case $A=\Q$, they are $v_i^{\rm{naive}}$, $v_{i,r}^{\rm{naive}}$ introduced in \cite[p.
458]{Ke05}.) It is clear that
\[
v_{i,r}(f)=rv_i(f)+i\geq rv_i(f)+v(a_{v_i(f)})\geq w_r(f).
\]
Furthermore, we claim that $w_r(f)=\min_i\{v_{i,r}(f)\}$. In fact, suppose $w_r(f)=v(a_{j_0})+rj_0$ for some $j_0$. Let $i_0=v(a_{j_0})$. It follows that $v_{i_0}(f)\leq j_0$. This implies that $v_{i_0,r}(f)\leq w_r(f)$, yielding the claim.

We define a sequence of invertible matrices $U_0$, $U_1$, $\dots$ over
$\r_{A_K}^{[r, qr]}$ and a sequence of matrices $F_0$, $F_1$, $\dots$ over
$\r_{A_K}^{[r,r]}$ as follows. Set $U_0=I_n$. Given $U_l$, put
$F_l=U_l^{-1}F\varphi(U_l)$. Suppose
$F_lD^{-1}-I_n=\displaystyle{\sum_{m=-\infty}^{\infty}}V_mT^m$ where
the $V_m$'s are $n\times n$ matrices over $A_K$. Let
$X_l=\displaystyle{\sum_{v(V_m)\leq k}}V_mT^m$, and put
$U_{l+1}=U_l(I_n+X_l)$. Set
\[
c_l=\min_{i\leq k}\{{v_{i,r}(F_lD^{-1}-I_n)}-h/(q-1)\}.
\]
By the construction of $X_l$ we get
\[
w_r(X_l)=\min_{i}v_{i,r}(X_l)=\min_{i\leq k}v_{i,r}(X_l)=\min_{i\leq k}v_{i,r}(F_lD^{-1}-I_n)=c_l+h/(q-1).
\]
We now prove by induction that $c_l\geq\max\{c,\frac{l+1}{2}c\}$,
$w_{r}(F_lD^{-1}-I_n)\geq c+h/(q-1)$ and $U_l$ is invertible over
$\r_{A_K}^{[r,qr]}$ for any $l\geq0$. For $l=0$, by assumption, it is clear that
\[
c_0\geq w_{r}(FD^{-1}-I_n)-h/(q-1)\geq c.
\]
Suppose
that the claim is true for some $l\geq 0$.  Note that for any $s\in[r,qr]$ and $m\in\mathbb{Z}$,
\begin{eqnarray*}
(s/r)(v(V_m)+rm)&=&v(V_m)+sm+(s/r-1)v(V_m)\\
&\leq& v(V_m)+sm+(s/r-1)k.
\end{eqnarray*}
Hence $(s/r)w_r(X_l)\leq w_s(X_l)+(s/r-1)k$. Since $c_l\geq\frac{l+1}{2}c\geq (q-1)k$, we therefore deduce that
\begin{eqnarray*}
w_s(X_l)&\geq& (s/r)w_r(X_l)-(s/r-1)k\\
&=&(s/r)(c_l+h/(q-1))-(s/r-1)k\\
&>&0
\end{eqnarray*}
for any $s\in[r,qr]$. It follows that $U_{l+1}$ is invertible over $\r_{A_K}^{[r,qr]}$.
Furthermore, we have
\begin{eqnarray*}
w_r(D\varphi(X_l)D^{-1})&\geq& w_r(D)+w_r(\varphi(X_l))+w_r(D^{-1})\\
&=&w_{qr}(X_l)-h\\
&\geq&q(c_l+h/(q-1))-(q-1)k-h\\
&=&qc_l+h/(q-1)-(q-1)k\\
&\geq& c_l+\frac{1}{2}c+h/(q-1)+(\frac{1}{2}c-(q-1)k)\\
&\geq& \frac{(l+2)}{2}c+h/(q-1)
\end{eqnarray*}
since $c_l\geq c$. Note that
\begin{eqnarray*}
F_{l+1}D^{-1}-I_n&=&(I_n+X_l)^{-1}F_lD^{-1}(I_n+D\varphi(X_l) D^{-1})-I_n\\
&=&((I_n+X_l)^{-1}F_lD^{-1}-I_n)+(I_n+X_l)^{-1}(F_lD^{-1})D\varphi(X_l)D^{-1}.
\end{eqnarray*}
Since $w_r(F_lD^{-1})=w_r((I_n+X_l)^{-1})=0$, for $i\leq k$, we have
\begin{eqnarray*}
v_{i,r}((I_n+X_l)^{-1}(F_lD^{-1})D\varphi(X_l)D^{-1})&\geq& w_r(D\varphi(X_l)D^{-1})\\
&\geq &\frac{(l+2)}{2}c+h/(q-1).
\end{eqnarray*}
Write
\begin{eqnarray*}
(I_n+X_l)^{-1}F_lD^{-1}-I_n&=&(I_n+X_l)^{-1}(F_lD^{-1}-I_n-X_l)\\
&=&\displaystyle{\sum_{j=0}^{\infty}}(-X_l)^{j}(F_lD^{-1}-I_n-X_l).
\end{eqnarray*}
By definition of $X_l$, we have $v_i(F_lD^{-1}-I_n-X_l)=\infty$
for $i\leq k$ and
\[
w_{r}(F_lD^{-1}-I_n-X_l)\geq w_{r}(F_lD^{-1}-I_n)\geq c+h/(q-1).
\]
Thus $v_{i,r}(F_lD^{-1}-I_n-X_l)=\infty$ for $i\leq k$, and for $j\geq1$ and $i\leq k$, we have
\begin{eqnarray*}
v_{i,r}((-X_l)^{j}(F_lD^{-1}-I_n-X_l))&\geq& w_r((-X_l)^{j}(F_lD^{-1}-I_n-X_l))\\
&\geq& jw_r(X_l)+c+h/(q-1)\\
&=& j(c_l+h/(q-1))+c+h/(q-1)\\
&\geq& c+c_l+2h/(q-1)\\
&>&\frac{l+2}{2}c+h/(q-1).
\end{eqnarray*}
Putting all these together, we get
\[
v_{i,r}(F_{l+1}D^{-1}-I_n)\geq \frac{l+2}{2}c+h/(q-1)
\]
for any $i\leq k$ and
$w_{r}(F_{l+1}D^{-1}-I_n)\geq c+h/(q-1)$; this yields $c_{l+1}\geq \frac{l+2}{2}c$. The induction step is
finished.

Now since $w_s(X_l)\geq (s/r)(c_l+h/(q-1))-(s/r-1)k$ for $s\in[r,qr]$, and
$c_l\rightarrow\infty$ as $l\rightarrow\infty$, the sequence ${U_l}$
converges to a limit $U$, which is an invertible $n\times n$ matrix
over $\r_{A_K}^{[r,qr]}$ satisfying $w_r(U^{-1}F\varphi(U)D^{-1}-I_n)\geq
c+h/(q-1)$. Furthermore, we have
\[
v_{i,r}(U^{-1}F\varphi(U)D^{-1}-I_n)=\displaystyle{\lim_{l\rightarrow\infty}}
v_{i,r}(U_l^{-1}F\varphi(U_l)D^{-1}-I_n)
=\displaystyle{\lim_{l\rightarrow\infty}}
v_{i,r}(F_{l+1}D^{-1}-I_n)=\infty
\]
for any $i\leq k$. Therefore $U^{-1}F\varphi(U)D^{-1}-I_n$ has
entries in $\pi^k\r_{A_K}^{\int, r}$.
\end{proof}

\begin{lemma}\label{lem:model-extension}
For any free $\varphi$-modules $N_1,N_2$ over $\r^{\bd}_{A_K}$, the natural map
\[
\mathrm{Ext}^1_{\varphi,\r^\bd_{A_K}}(N_1,N_2)\ra\mathrm{Ext}^1_{\varphi,\r_{A_K}}
(N_1\otimes_{\r_{A_K}^{\bd}}\r_{A_K},N_2\otimes_{\r_{A_K}^{\bd}}\r_{A_K})
\]
is surjective. Here $\mathrm{Ext}^1_{\varphi,\r^\bd_{A_K}}$ and $\mathrm{Ext}^1_{\varphi,\r_{A_K}}$ denote the set of extensions
in the category of $\varphi$-modules over $\r^\bd_{A_K}$ and $\r_{A_K}$ respectively.
\end{lemma}
\begin{proof}
Let $M$ be any extension of $N_2\otimes_{\r_{A_K}^{\bd}}\r_{A_K}$ by $N_1\otimes_{\r_{A_K}^{\bd}}\r_{A_K}$ in the category of $\varphi$-modules over $\r_{A_K}$. We pick an $\r_{A_K}$-basis $\mathrm{e}=\{e_1, e_2, \dots, e_{n_1+n_2}\}$ of $M$ so that
$\{e_1, e_2, \dots, e_{n_1}\}$ is a basis of $N_1$ and $\{e_{n_1+1},
e_{n_1+2}, \dots, e_{n_1+n_2}\}$ is a lift of a basis of $N_2$.
The matrix of $\varphi$ under $\mathrm{e}$ is then of the form
\[
F=\left( \begin{array}{cc}
F_{11}&F_{12} \\
0 &F_{22} \\
\end{array} \right),\]
where $F_{11}$ is the matrix of $\varphi$ under $\{e_1, e_2, \dots,
e_{n_1}\}$ and $F_{22}$ is the matrix of $\varphi$ under the image of
$\{e_{n_1+1}, e_{n_2+2}, \dots, e_{n_1+n_2}\}$.  It is clear that for
 any $\lambda\in K^{\times}$,  the matrix of $\varphi$ under $\{e_1, e_2, \dots, e_{n_1},\lambda e_{n_1+1}, \dots, \lambda e_{n_1+n_2}\}$ is
\[ \left( \begin{array}{cc}
F_{11}&\varphi(\lambda)F_{12} \\
0 &(\varphi(\lambda)/\lambda)F_{22} \\
\end{array} \right).\]
Suppose that $F$ is invertible over $\r_{A_K}^{r}$ for some $0<r<r_{\varphi}/q$.
Put
\[ D=\left( \begin{array}{cc}
F_{11}& 0\\
0 &(\varphi(\lambda)/\lambda)F_{22}\\
\end{array} \right),
\]
and $h=-w_r(D)-w_r(D^{-1})$ which is independent of $\lambda$. We choose a positive integer $k$ and $\lambda$ such that
\[
w_r(FD^{-1}-I_{n_1+n_2})=w_r(\lambda F_{12}F_{22}^{-1})\geq2k(q-1)+h/(q-1).
\]
Then applying Lemma \ref{lem:modification}, we obtain an $(n_1+n_2)\times(n_1+n_2)$ invertible matrix $U$ over $\r_{A_K}^{[r,qr]}$ such that $U^{-1}F\varphi(U)D^{-1}-I_{n_1+n_2}$ lies in $\pi^k\r_{A_K}^{\int,r}$ and
\[
w_r(U^{-1}F\varphi(U)D^{-1}-I_{n_1+n_2})\geq2k(q-1)+h/(q-1)>0.
\]
This implies that $U^{-1}F\varphi(U)D^{-1}$, which is the matrix of $\varphi$ under $\mathrm{e} U$, is invertible over $\r_{A_K}^{\bd,r}$.  It follows from Lemma \ref{lem:basis-extension} that $\mathrm{e}U$ extends to a basis of $M$. Moreover, following the construction of $U$ given in Lemma \ref{lem:modification}, we see that each $U_l$ is of the form
\[
U_l=\left( \begin{array}{cc}
I_{n_1}& \ast\\
0 &I_{n_2}\\
\end{array} \right).
\]
Hence so is $U$.
Therefore the model $N$ of $M$ generated by $\mathrm{e} U$ is an extension of $N_2$ by $N_1$.
\end{proof}

\begin{proposition}\label{prop:good-model-point}
Every $\varphi$-module
over $\r_K$ admits a good model.
\end{proposition}
\begin{proof}
Each pure $\varphi$-module over $\r_K$ has a unique good model. The general case then follows from Theorem \ref{thm:slope-filtration}, Lemma \ref{lem:model-extension} and Corollary \ref{cor:additive}.
\end{proof}

\begin{proposition}\label{prop:global-model}
Let $M_A$ be a family of $\varphi$-modules over $\r_{A_K}$.
Then for any $x\in M(A)$ and a model $N_x$ of $M_x$, there exists a
Weierstrass subdomain $M(B)$ containing $x$ such that
$M_{B}=M_A\otimes_{\r_{A_K}}\r_{B_K}$ admits a finite free
model $N_B$ which lifts $N_x$. Furthermore, if $k(x)\subset A$, we can choose $M(B)$ so that $N_y$ has constant generic HN-polygons for any $y\in M(B)$.
\end{proposition}
\begin{proof}
Let $\mathrm{e}_x$ be a basis of $N_x$. By Lemma \ref{lem:locally-free}, after shrinking $M(A)$, we may lift $\mathrm{e}_x$ to a basis $\mathrm{e}_A$ of $M^{[r/q,r]}_A$ for some $0<r<r_{\varphi}$. Let $F$ be the matrix of $\varphi$ under $\mathrm{e}_A$. Then $F$ is invertible over $\r_{A_K}^{[r/q,r/q]}$. Since $F(x)$ is the matrix of $\varphi$ under $\mathrm{e}_x$, we get that it is invertible over
$k(x)\otimes_{\Q}\r_K^{\bd,r/q}$. By Corollary \ref{cor:lifting}, we may lift $F(x)^{-1}$ to an invertible matrix $F'$ over $\r^{\bd,r/q}_{A_K}$ by shrinking $M(A)$. Put
\[
h=-w_{r/q}(F')-w_{r/q}((F')^{-1}).
\]
Since $FF'-I$ vanishes at $x$, by Proposition \ref{prop:shrinking}, we choose a positive integer $k$ and
a Weierstrass subdomain $M(B)$ containing $x$ such that
\[
w_{r/q}(FF'-I)\geq2k(q-1)+h/(q-1)
\]
in $\r_{B_K}^{[r/q,r/q]}$. Put
$h'=-w_{r/q}(F')-w_{r/q}((F')^{-1})$
in $\r_{B_K}^{[r/q,r/q]}$; then $h'\leq h$.
Let $c=w_{r/q}(FF'-I)-h'/(q-1)$; then $c\geq 2k(q-1)$.
We therefore deduce from Lemma \ref{lem:modification} that there
exists an invertible matrix $U$ over $\r_{B_K}^{[r/q,r]}$ such that
$U^{-1}F\varphi(U)F'-I_n$ has entries in $\pi^k\r_{B_K}^{\int,
r/q}$ and satisfies
$w_{r/q}(U^{-1}F\varphi(U)F'-I_n)>0$.
This implies that $U^{-1}F\varphi(U)F'$ is invertible over $\r^{\int,r/q}_{B_K}$, yielding that $U^{-1}F\varphi(U)$,
which is the matrix of $\varphi$ under the basis $\mathrm{e}_AU$,
is invertible over  $\r^{\bd,r/q}_{B_K}$. It therefore follows from Lemma \ref{lem:basis-extension} that $\mathrm{e}_AU$ extends to a basis of $M_B$. Furthermore, following the construction of $U$ given in Lemma \ref{lem:modification}, we have $X_l(x)=0,U_l(x)=I$ for each $l$. Hence $U(x)=I$.  Therefore the $\r_{B_K}^\bd$-submodule $N_B$ generated by this basis is a finite free model of $M_B$ lifting $N_x$.

Finally, suppose $k(x)\subset A$. By Corollary \ref{cor:numerical-generic}, if we further shrink $M(B)$ so that $|U^{-1}F\varphi(U)-F(x)|$ is sufficiently small in $\calE_{B_K}$, then the generic HN-polygon of $N_y$ is the same as the generic HN-polygon of $N_x$ for any $y\in M(B)$.
\end{proof}

\begin{corollary}\label{cor:base-lifting}
Let $M_A$ be a family of $\varphi$-modules over $\r_{A_K}$, and
let $x\in M(A)$. Then there exists a
Weierstrass subdomain $M(B)$ containing $x$ such that the vector bundle
$M_{B}=M_A\otimes_{\r_{A_K}}\r_{B_K}$ is freely generated over $\r_{B_K}$.
\end{corollary}
\begin{proof}
This follows immediately from Proposition \ref{prop:good-model-point} and the first part of Proposition  \ref{prop:global-model}.
\end{proof}

\begin{proposition}\label{prop:model-unique-family}
Let $M_A$ be a global pure family of $\varphi$-modules over $\r_{A_K}$. If $N'_A$ and $N''_A$ are two finite free pure models of $M_A$, then they generate the same finite free good model of $M_A$.
\end{proposition}
\begin{proof}
Let $L_1,\dots, L_n$ be the residue fields of the generic points of $A_K$; then the natural map $A_K\ra L=L_1\times\cdots\times L_n$ is a closed embedding (see \cite{B90} for more details about generic points of affinoid algebras and the embedding). Let $U$ be the transformation matrix between some bases of $N_A'$ and $N_A''$. Note that the base changes $N_A'\otimes_{\r_{A_K}^\int}\r^\int_{L}$ and $N_A''\otimes_{\r_{A_K}^\int}\r^\int_{L}$ are pure models of $M_A\otimes_{\r_{A_K}}\r_L$. We therefore deduce from Proposition \ref{prop:model-unique} that $U$ is invertible over $\r_L^\bd$. Hence $U$ is invertible over $\r_L^\bd\cap\r_{A_K}=\r_{A_K}^\bd$, yielding the proposition.
\end{proof}

\begin{theorem}\label{thm:pure-family}
Let $M_A$ be a family of $\varphi$-modules over $\r_{A_K}$.
Suppose that $M_x$ is pure of slope $s$ for some $x\in M(A)$ with $k(x)\subset A$, then there
exists a Weierstrass subdomain $M(B)$ containing $x$ such that $M_B=M_A\otimes_{\r_{A_K}}\r_{B_K}$ admits a finite free $(c,d)$-pure model $N_B$ where $d>0,(c,d)=1$ and $c/d=s$. In particular, $M_B$ is globally pure of slope $s$.
\end{theorem}
\begin{proof}
We first prove the proposition for $M$ as a $\varphi^d$-module. By tensoring with $\r_{A_K}(-c)$, we may assume that $s=0$. Let $N_x$ be an \'etale model of $M_x$, and let $\mathrm{e}_x$ be a basis of $N_x$.  By Proposition \ref{prop:global-model}, for some Weierstrass subdomain $M(B)$ containing $x$, we can lift $\mathrm{e}_x'$ to a basis $\mathrm{e}_B$ of $M_B$ which generates a finite free good model of $M_B$.  Let $F$ be the matrix of $\varphi^d$ under $\mathrm{e}_B$. Note that $F(x)$, which is the matrix of $\varphi$ under $\mathrm{e}_x$, is invertible over $k(x)\otimes_{\Q}\r^{\int,r}_{K}$ for some $r>0$. We may suppose that $F$ is invertible over $\r_{B_K}^{\bd,r}$ by shrinking $r$. By Proposition \ref{prop:shrinking}, we may further shrink $M(B)$ so that
$\min\{w(F-F(x)),w(F^{-1}-(F(x))^{-1})\}>0$
over $\r^{\bd,r}_{B_K}$. This implies $F,F^{-1}\in\r^{\int,r}_{B_K}$. Hence the $\r^{\int}_{B_K}$-submodule $N_B$ of $M_B$ generated by $\mathrm{e}_B$ is a finite free \'etale model of $M_B$.

Note that if $N_B$ is a finite free pure $(c,d)$-model of $M_B$ as a $\varphi^d$-module, so is $\varphi^*N_B$. Thus $N_B$ and $\varphi^*N_B$ generate the same finite free good model of $M_B$ by Proposition \ref{prop:model-unique-family}, yielding that this model is stable under $\varphi$. This implies that $N_B$ is
a finite free pure $(c,d)$-model of $M_B$ as a $\varphi$-module.
\end{proof}

\subsection{Global slope filtration}
We set $L$ and $\widetilde{\r}$ (resp. $\widetilde{\r}^\int$, $\widetilde{\r}^\bd$, $\widetilde{\calE}$) separately in the following two cases.
\begin{enumerate}
\item[(AF)](absolute Frobenius case) If $K$ is an unramified extension of $\Q$ in $\overline{\mathbb{Q}}_p$ and $\varphi$ is a $q$-power absolute Frobenius lift, let $L=\widehat{\mathbb{Q}_p^{\mathrm{ur}}}$, and let $\widetilde{\r}$, $\widetilde{\r}^\int$, $\widetilde{\r}^\bd$, $\widetilde{\calE}$ denote $\bt_{\rig}^\dagger$, $\cup_{r>0}\at^{(0,r]}$, $\bt^\dagger$, $\bt$ respectively. (See \cite{Colmez07} for more details about the constructions of $\bt_{\rig}^\dagger$, $\at^{(0,r]}$, $\bt^\dagger$ and $\bt$). The latter can be  identified with $\Gamma_{\mathrm{an,con}}^\mathrm{alg}$, $\Gamma^{\mathrm{alg}}_{\mathrm{con}}$, $\Gamma_{\mathrm{con}}^\mathrm{alg}[\pi^{-1}]$ and $\Gamma^\mathrm{alg}[\pi^{-1}]$ respectively.
    (See \cite[\S1.1]{B08} for more explanations about these identifications.)
    Here the latter are different type of basic rings associated to the residue field  $\widehat{\mathbb{F}_p((u))^{\mathrm{alg}}}$ (where the completion is taken for the $u$-adic topology) introduced by Kedlaya in \cite[\S2]{Ke05}. On the other hand, $\widetilde{\r}_L$, $\widetilde{\r}^\int_L$, $\widetilde{\r}_L^\bd$, $\widetilde{\calE}_L$ are basic rings associated to $\overline{\mathbb{F}}_p((u^{\mathbb{Q}}))$. By \cite[Theorem 8]{Ke01}, $\widehat{\mathbb{F}_p((u))^{\mathrm{alg}}}$ is a proper closed subfield of $\overline{\mathbb{F}}_p((u^{\mathbb{Q}}))$. This leads to natural embeddings $\bt_{\rig}^\dagger\subset\widetilde{\r}_L$,  $\cup_{r>0}\at^{(0,r]}\subset\widetilde{\r}^{\int}_L$, $\bt^\dagger\subset\widetilde{\r}_L^\bd$, $\bt\subset\widetilde{\calE}_L$, which respect Frobenius actions, following \cite[\S2]{Ke05}.
\item[(RF)](relative Frobenius case) For general $K$, let $L$ be some admissible extension of $K$ with strongly difference-closed residue field $k_L$, and let $\widetilde{\r}$, $\widetilde{\r}^\int$, $\widetilde{\r}^\bd$, $\widetilde{\calE}$ denote $\widetilde{\r}_L$, $\widetilde{\r}^\int_L$, $\widetilde{\r}_L^\bd$, $\widetilde{\calE}_L$ respectively.
\end{enumerate}

\begin{remark}
In both cases, $L$ are admissible extensions of $K$ with strongly difference-closed residue fields. In the $\mathrm{AF}$ case,  $\r_K$, $\r_K^\bd$, $\calE_K$ are the basic rings associated to $k((T))$ (\cite[\S2.3]{Ke05}), and the $\varphi$-equivariant embeddings $\r_K\ra\widetilde{\r}_L$, $\r_K^\bd\ra\widetilde{\r}_L^\bd$,  $\calE_K\ra\widetilde{\calE}_L$  given in
Remark \ref{rem:embedding} are induced by the natural embedding $k((T))\ra \overline{\mathbb{F}}_p((u^{\mathbb{Q}}))$ defined as $\sum_{i>-\infty}a_iT^i\mapsto \sum_{i>-\infty}a_iu^i$; this embedding factors through $\widehat{\mathbb{F}_p((u))^{\mathrm{alg}}}$. Thus the embeddings  $\r_K\ra\widetilde{\r}_L$, $\r_K^\bd\ra\widetilde{\r}_L^\bd$,  $\calE_K\ra\widetilde{\calE}_L$ factor through $\widetilde{\r}$, $\widetilde{\r}^{\bd}$, $\widetilde{\calE}$ respectively.
\end{remark}

\begin{remark}
In the $\mathrm{AF}$ case, the slope theory for $\varphi$-modules over $\widetilde{\r}$, $\widetilde{\r}^{\bd}$, $\widetilde{\calE}$  bears the same properties as the slope theory for $\varphi$-modules over $\widetilde{\r}_L$, $\widetilde{\r}_L^\bd$, $\widetilde{\calE}_L$. In fact, all the results of $\S1.5,\S1.6, \S1.7$ for $\varphi$-modules over the extended base rings are motivated by their counterparts for $\varphi$-modules over $\widetilde{\r}$, $\widetilde{\r}^{\bd}$, $\widetilde{\calE}$ developed in \cite{Ke05}.
\end{remark}

\begin{definition}
Let $M_A$ (resp. $N_A$) be a family of $\varphi$-modules over $\r_{A_K}$ (resp. $\r^{\bd}_{A_K}$). For any $x\in M(A)$,  we set
\[
\widetilde{M}_x=M_A\otimes_{\r_{A_K}}(k(x)\otimes_{\Q}\widetilde{\r})\quad \text{(resp. $\widetilde{N}_x=N_A\otimes_{\r^{\bd}_{A_K}}(k(x)\otimes_{\Q}\widetilde{\r}^{\bd})$)}
\]
which is a $\varphi$-module over $k(x)\otimes_{\Q}\widetilde{\r}^\bd$ (resp. $k(x)\otimes_{\Q}\widetilde{\r}$).
\end{definition}

\begin{proposition}\label{prop:HN-polygon-multifield-extended}
Let $M_A$ (resp. $N_A$) be a family of $\varphi$-modules over $\r_{A_K}$ (resp. $\r_{A_K}^\bd$), and let $x\in M(A)$. Suppose that
\[
k(x)\otimes_{\Q}L=\oplus_{i=1}^nL_i
\]
where each $L_i$ is a finite field extension of $L$. Then the following are true.
\begin{enumerate}
\item[(1)]The residue field $k_{L_i}$ of $L_i$ is strongly difference-closed for each $1\leq i\leq n$.
\item[(2)]Let $\widetilde{M}_{x,i}=M_x\otimes_{k(x)\otimes_{\Q}\r_{K}}(L_i\otimes_{L}\widetilde{\r})$ (resp. $\widetilde{N}_{x,i}=N_x\otimes_{k(x)\otimes_{\Q}\r^{\bd}_{K}}(L_i\otimes_{L}\widetilde{\r}^{\bd})$) for $1\leq i\leq n$. Then the HN-polygon of each $\widetilde{M}_{x,i}$ (resp. generic HN-polygon of each $\widetilde{N}_{x,i}$) is the same as $M_x$'s (resp. $N_x$'s).
\end{enumerate}
\end{proposition}
\begin{proof}
By Proposition \ref{prop:HN-polygon-multifield}(1), the $\varphi$-action on each $L_i$ is an automorphism. Hence the residue field $k_{L_i}$ is inversive. For the rest of $(1)$, it reduces to show that any dualizable difference module $P$ over $k_{L_i}$ is trivial. Since $k_{L_i}$ is inversive, the $\varphi$-action on $P$ is therefore bijective. Hence $P$ is also dualizable as a difference module over $k_L$. This implies that $P$ a $\varphi$-invariant basis over $k_L$. Hence it admits a $\varphi$-invariant basis over $k_{L_i}$. For (2),
%suppose $k(x)\otimes_{\Q}K=\oplus_{j=1}^{m}K_j$ where each $K_j$ is a finite extension of $K$.
It is clear that each $\widetilde{M}_{x,i}$ (resp. $\widetilde{N}_{x,i}$) is a base change of some component of $M_x$ (resp. $N_x$). Hence the HN-polygon of $\widetilde{M}_{x,i}$ (resp. generic HN-polygon of $\widetilde{N}_{x,i}$) is the same as $M_x$'s (resp. $N_x$'s) by Proposition \ref{prop:robba-extension} (for the $\mathrm{RF}$ case) and \cite[Theorem 6.4.1]{Ke05} (for the $\mathrm{AF}$ case) (resp. Proposition \ref{prop:generic-slope-filtration} (for the $\mathrm{RF}$ case) and \cite[Proposition 5.3.1]{Ke05} (for the $\mathrm{AF}$ case)), yielding the proposition.
\end{proof}

In the situation of Proposition \ref{prop:HN-polygon-multifield-extended},  it is clear that $\widetilde{M}_x$ (resp. $\widetilde{N}_x$) is isomorphic to the direct sum of all $\widetilde{M}_{x,i}$ (resp. $\widetilde{N}_{x,i}$). We call each $\widetilde{M}_{x,i}$ (resp. $\widetilde{N}_{x,i}$) a \emph{component} of $\widetilde{M}_x$. We set the \emph{slopes} and \emph{HN-polygon} of $\widetilde{M}_x$ (resp. \emph{generic slopes} and \emph{generic HN-polygon} of $\widetilde{N}_x$) as the slopes and HN-polygon of $\widetilde{M}_{x,i}$ (resp. generic slopes and generic HN-polygon of $\widetilde{N}_{x,i}$). We set the \emph{HN filtration} of $\widetilde{M}_x$ (resp. \emph{generic HN filtration} of $\widetilde{N}_x$) as the direct sum of the HN filtrations of all $\widetilde{M}_{x,i}$ (generic HN filtrations of all $\widetilde{N}_{x,i}$).

\begin{lemma}\label{lem:Frobenius-equation-family}
Consider the Frobenius equation
\begin{equation}\label{eq:Frobenius-equation-family}
\varphi(\beta)-\pi^n\beta=\alpha.
\end{equation}\begin{enumerate}
\item[(1)]
Let $\alpha\in A\widehat{\otimes}_{\Q}\widetilde{\calE}$. If $n\neq0$, then (\ref{eq:Frobenius-equation-family})
admits a unique solution $\beta\in A\widehat{\otimes}_{\Q}\widetilde{\calE}$ which is
\begin{equation}\label{eq:generalize-solution-n-nega}
\beta=-\sum_{m=0}^\infty(\pi^{-n})^{\{m+1\}}\varphi^m(\alpha)
\end{equation}
if $n<0$, or
\begin{equation}\label{eq:generalize-solution-n-posi}
\beta=\sum_{m=0}^\infty(\pi^{-n})^{\{-m\}}\varphi^{-m-1}(\alpha)
\end{equation}
if $n>0$. Furthermore, if $n>0$, then $w(\beta)=w(\alpha)$, and if $n<0$, then $w(\beta)=w(\alpha)-n$. If $n=0$, then (\ref{eq:Frobenius-equation-family}) admits a solution $\beta\in A\widehat{\otimes}_{\Q}\widetilde{\calE}$
with $w(\beta)=w(\alpha)$.
\item[(2)]
Let $\alpha\in A\widehat{\otimes}_{\Q}\widetilde{\r}^{\bd,r}$.  If $n>0$, then (\ref{eq:generalize-solution-n-posi}) provides the unique solution $\beta\in A\widehat{\otimes}_{\Q}\widetilde{\r}^{\bd}$ of (\ref{eq:Frobenius-equation-family}). Furthermore, we have $\beta\in A\widehat{\otimes}_{\Q}\widetilde{\r}^{\bd,qr}$, $w(\beta)=w(\alpha)$ and $w_r(\beta)\geq\min\{w(\alpha),w_r(\alpha)\}$.
If $n=0$, then (\ref{eq:Frobenius-equation-family}) admits a solution $\beta\in A\widehat{\otimes}_{\Q}\widetilde{\r}^{\bd,qr}$ with $w(\beta)=w(\alpha)$ and $w_r(\beta)\geq w_r(\alpha)$.
\item[(3)]
Let $\alpha\in A\widehat{\otimes}_{\Q}\widetilde{\r}^{\bd,r}$, and write $\alpha=\sum_{i\in\mathbb{Q}}a_iu^i$ as an element of $\widetilde{\r}^{\bd,r}_{A_L}$. If $n<0$, then (\ref{eq:Frobenius-equation-family}) admits at most one solution $\beta\in A\widehat{\otimes}_{\Q}\widetilde{\r}$, and it has a solution if and only if
\begin{equation}\label{eq:obstruction}
\sum_{m\in\mathbb{Z}}(\pi^{-n})^{\{m+1\}}\varphi^m(a_{iq^{-m}})=0
\end{equation}
for every $i<0$. Furthermore, if $\beta$ is a solution of (\ref{eq:Frobenius-equation-family}), then it belongs to $ A\widehat{\otimes}_{\Q}\widetilde{\r}^{\bd,qr}$ and satisfies $w(\beta)=w(\alpha)-n, w_r(\beta)\geq w_r(\alpha)-C(q,r,n)$ where $C(q,r,n)$ is some constant which only depends on $q,r,n$. As a consequence of $(1)$, we see that $\beta$ is given by (\ref{eq:generalize-solution-n-nega}).
\end{enumerate}
\end{lemma}
\begin{proof}
The uniqueness part of $(1)$ and $(2)$ follow from the fact that $w$ is preserved by $\varphi$.
By Lemma \ref{lem:orthonomal-basis}, $A$ admits an orthogonal basis over $\Q$. Since $\varphi$ acts trivially on $A$, using an orthogonal basis, the rest of $(1)$ and $(2)$ reduce to the case $A=\Q$. For $(1)$, if $n\neq0$, it is clear that the series (\ref{eq:generalize-solution-n-nega}) and (\ref{eq:generalize-solution-n-posi}) converge in $\widetilde{\calE}$, and give a solution of (\ref{eq:Frobenius-equation-family}). For $n=0$, we apply Lemma \ref{lem:difference-closed}. For (2), it follows from Lemma \ref{lem:Frobenius-equation}(2) (for the $\mathrm{RF}$ case), \cite[Proposition 3.3.7(c)]{Ke05} (for the $\mathrm{AF}$ case  and $n>0$) and \cite[Lemma 5.1]{KL10}
(for the $\mathrm{AF}$ case and $n=0$).

For $(3)$, suppose that $\beta=\sum_{i\in\mathbb{Q}}b_iu^i\in A\widehat{\otimes}_{\Q}\widetilde{\r}$ is a solution of (\ref{eq:Frobenius-equation-family}). Comparing the coefficients of both sides of (\ref{eq:Frobenius-equation-family}), we get
$\varphi(b_{i/q})-\pi^nb_i=a_i$
for every $i\in\mathbb{Q}$; hence $b_i=\pi^{-n}\varphi(b_{i/q})-\pi^{-n}a_i$. Since $n<0$ and $\{a_i\}_{i\in\mathbb{Q}}$ is bounded, we get
\begin{equation}\label{eq:beta}
b_i=-\sum_{m=0}^\infty (\pi^{-n})^{\{m+1\}}\varphi^m(a_{iq^{-m}})
\end{equation}
by iteration. Thus $\beta$ is uniquely determined by $\alpha$ and belongs to $\widetilde{\r}_{A_L}^\bd$. Furthermore, it follows that there exist $C\in\mathbb{R}$ and $s>0$ such that
\[
v(\sum_{m=0}^\infty (\pi^{-n})^{\{m+1\}}\varphi^m(a_{iq^{-m}}))\geq C-si
\]
for any $i\in\mathbb{Q}$. Since $((\pi^{-n})^{\{m+1\}}\varphi^m(a_{iq^{-m}}))^{\{k\}}=(\pi^{-n})^{\{m+k+1\}}\varphi^{m+k}(a_{q^{-m}}))$ for any $k\in \mathbb{Z}$, it follows that
\begin{eqnarray*}
(\sum_{m=-k}^\infty(\pi^{-n})^{\{m+1\}}\varphi^m(a_{(iq^{-m}}))^{\{k\}}
&=&\sum_{m=0}^\infty(\pi^{-n})^{\{m+1\}}\varphi^m(a_{iq^{k-m}})\\
&=&\sum_{m=0}^\infty(\pi^{-n})^{\{m+1\}}\varphi^m(a_{(iq^k)q^{-m}}).
\end{eqnarray*}
Hence
\[
v(\sum_{m=-k}^\infty(\pi^{-n})^{\{m+1\}}\varphi^m(a_{iq^{-m}}))
=v(\sum_{m=0}^\infty(\pi^{-n})^{\{m+1\}}\varphi^m(a_{(iq^k)q^{-m}}))+nk\geq C-siq^k+nk.
\]
Therefore, if $i<0$, then $v(\sum_{m=-k}^\infty(\pi^{-n})^{\{m+1\}}\varphi^m(a_{iq^{-m}}))\rightarrow +\infty$ as $k\rightarrow +\infty$, yielding $\sum_{m\in\mathbb{Z}}(\pi^{-n})^{\{m+1\}}\varphi^m(a_{iq^{-m}})=0$. This proves the ``only if'' part of (3).

To prove the ``if'' part, for any $f=\sum_{i\in\mathbb{Q}}a_iu^i\in\widetilde{\r}^r_{A_L}$ and $c\in\mathbb{R}$, we set
\[
w_r^{c,-}(f)=\min_{i\leq c}\{v(a_i)+ri\}.
\]
It is clear that $w_r^{c,-}(f)\ra\infty$ as $c\ra-\infty$. Now suppose $\sum_{m\in\mathbb{Z}}(\pi^{-n})^{\{m+1\}}\varphi^m(a_{iq^{-m}})=0$ for every $i<0$. If $i\leq-1$, then for each $m\leq-1$,
\begin{equation*}
\begin{split}
v((\pi^{-n})^{\{m+1\}}\varphi^m(a_{iq^{-m}}))&=(v(a_{iq^{-m}})+riq^{-m})-riq^{-m}-n(m+1)\\
&\geq w^{i,-}_r(\alpha)-riq^{-m}-n(m+1)\geq (w^{i,-}_r(\alpha)-ri)-C_1(q,r,n)
\end{split}
\end{equation*}
for some constant $C_1(q,r,n)$. Hence
\begin{equation}\label{eq:ileq-1}
\begin{split}
w_r((\sum_{m=0}^\infty (\pi^{-n})^{\{m+1\}}\varphi^m(a_{iq^{-m}}))u^i)&=v(-\sum_{m=-1}^{-\infty}(\pi^{-n})^{\{m+1\}}\varphi^m(a_{iq^{-m}}))+ri\\
&\geq  w^{i,-}_r(\alpha)-C_1(q,r,n)
\end{split}
\end{equation}
for $i\leq-1$. If $i>-1$, then for any $m\geq0$,
\[
v((\pi^{-n})^{\{m+1\}}\varphi^m(a_{iq^{-m}}))\geq w_r(\alpha)-riq^{-m}-n(m+1)\geq (w_r(\alpha)-ri)-C_2(q,r,n)
\]
for some constant $C_2(q,r,n)$, yielding
\begin{equation}\label{eq:i>-1}
w_r((\sum_{m=0}^\infty (\pi^{-n})^{\{m+1\}}\varphi^m(a_{iq^{-m}}))u^i)\geq w_r(\alpha)-C_2(q,r,n).
\end{equation}
Now suppose that $\beta$ is given by (\ref{eq:solution-n-nega}). Since the series is convergent in $A\widehat{\otimes}_{\Q}\widetilde{\calE}$ (hence in $\widetilde{\calE}_{A_L}$), a short computation shows that the $i$-th coefficients of $\beta$ is just $b_i$ given by (\ref{eq:beta}). Hence
\begin{equation}
\beta=-\sum_{i\in\mathbb{Q}}(\sum_{m=0}^\infty (\pi^{-n})^{\{m+1\}}\varphi^m(a_{iq^{-m}}))u^i.
\end{equation}
We claim that $\beta\in\widetilde{\r}_{A_L}^{\bd,r}$. Since $\beta\in A\widehat{\otimes}_{\Q}\widetilde{\calE}$, it satisfies
(1) of Definition \ref{def:extended-Robba}. By (\ref{eq:ileq-1}) and (\ref{eq:i>-1}), we see that $\beta$ satisfies (2) of Definition \ref{def:extended-Robba}; hence $\beta\in\widetilde{\r}_{A_L}^{\bd,r}$, and it satisfies $w_r(\beta)\geq w_r(\alpha)-C(q,r,n)$ for
\[
C(q,r,n)=\max\{C_1(q,r,n),C_2(q,r,n)\}.
\]
Furthermore, $\varphi(\beta)=\alpha+\pi^n\beta\in\widetilde{\r}_{A_L}^{\bd,r}$ implies $\beta\in\widetilde{\r}_{A_L}^{\bd,qr}$. Since $\widetilde{\calE}$, which is complete with respect to $w$,
is a closed subspace of $\widetilde{\calE}_L$, and $\widetilde{\calE}\cap\widetilde{\r}_{L}^{\bd,qr}=\widetilde{\r}^{\bd,qr}$, we deduce that
\[
\beta\in\widetilde{\r}_{A_L}^{\bd,qr}\cap  (A\widehat{\otimes}_{\Q}\widetilde{\calE})=A\widehat{\otimes}_{\Q}\widetilde{\r}^{\bd,qr}
\]
by Lemma \ref{lem:intersection}.
\end{proof}
In the situation of Lemma \ref{lem:Frobenius-equation-family}(3), we call the ideal of $A_L$ generated by the left hand sides of (\ref{eq:obstruction}) for all $i<0$ the \emph{obstruction} of the equation (\ref{eq:Frobenius-equation-family}).

\begin{lemma}\label{lem:Zariski-closed}
Let $L'$ be a $p$-adic field, and let $a\in A_{L'}$. Then the set
\[
\{x\in M(A)|a(x)=0\}
\]
is a Zariski closed subset of $M(A)$.
\end{lemma}
\begin{proof}
Let $\{e_i\}_{i\in I}$ be an orthogonal basis of $L'$ over $\Q$, and write $a=\sum_{i\in I}a_ie_i$ with $a_i\in A$. It is then clear that
 $a(x)=0$ if and only if $a_i(x)=0$ for all $i\in I$. This yields the Lemma.
\end{proof}

\begin{lemma}\label{lem:Zariski-closed-obstruction}
Keep notations as in Lemma \ref{lem:Frobenius-equation-family}(3). Then the set $S$ of $x\in M(A)$ at which the specialization of (\ref{eq:Frobenius-equation-family}) admits a solution in $k(x)\otimes_{\Q}\widetilde{\r}^{\bd}$ forms a Zariski closed subset $M(B)$ of $M(A)$. Furthermore, (\ref{eq:Frobenius-equation-family}) admits a unique solution in $B\widehat{\otimes}_{\Q}\widetilde{\r}^\bd$.
\end{lemma}
\begin{proof}
By Lemma \ref{lem:Frobenius-equation-family}(3), we see that $S$ is just the set of points at which the image of the obstruction in $k(x)\otimes_{\Q}L$ is the zero ideal. Hence $S$ is a Zariski closed subset $M(B)$ of $M(A)$ by Lemma \ref{lem:Zariski-closed}. Furthermore, since the obstruction vanishes in $B_L$, we get the rest of the Lemma by Lemma \ref{lem:Frobenius-equation-family}(3) again.
\end{proof}

The following lemma is based on \cite[Proposition 5.4.5]{Ke05}.

\begin{lemma}\label{lemma:good-basis-family}
Let $D$ be an $n\times n$ diagonal matrix with entries $D_{ii}=\pi^{a_i}$ satisfying $a_1\geq \cdots\geq a_n$, and let $F\in\mathrm{GL}_n(A\widehat{\otimes}_{\Q}\widetilde{\r}^{\bd,r})$ for some $r>0$. If $w(FD^{-1}-I_n)>0$ and $w_r(FD^{-1}-I_n)>0$, then there exists $U\in A\widehat{\otimes}_{\Q}\widetilde{\r}^{\bd,qr}$ with $w(U-I_n)>0$ and $w_r(U-I_n)>0$, such that $U^{-1}F\varphi(U)D^{-1}-I_n$ is upper triangular nilpotent.
\end{lemma}
\begin{proof}
 Put $c_0=\min\{w(FD^{-1}-I_n), w_r(FD^{-1}-I_n)\}$ and $U_0=I_n$. We will inductively construct a sequence $ U_1, U_2\cdots\in\mathrm{GL}_n( A\widehat{\otimes}_{\Q}\widetilde{\r}^{\bd,qr})$ satisfying
\[
\min\{w(U_l^{-1}F\varphi(U_l)D^{-1}-I_n),w_r(U_l^{-1}F\varphi(U_l)D^{-1}-I_n)\}\geq c_0,
\]
\[
w(U_l-I_n)\geq c_0, \min\{ w(U_{l+1}-U_l), w_{qr}(U_{l+1}-U_l)\}\geq (l+1)c_0,
\]
and the lower triangular part of $U_l^{-1}F\varphi(U_l)D^{-1}-I_n$ has both $w$ and $w_r$ valuations $\geq(l+1)c_0$ for $l\geq1$. Given $U_l$, put $F_l=U_l^{-1}F\varphi(U_l)$, and write $F_lD^{-1}-I_n=B_l+C_l$ where $B_l$ is upper triangular nilpotent and $C_l$ is lower triangular. Then $\min\{w(C_l),w_r(C_l)\}\geq(l+1)c_0$. We claim that there exists an $n\times n$ lower triangular matrix $X_l$ over $ A\widehat{\otimes}_{\Q}\widetilde{\r}^{\bd,qr}$ satisfying $C_l+X_l=D\varphi(X_l)D^{-1}$ and
\[
\min\{w(X_l),w(D\varphi(X_l)D^{-1}), w_{qr}(X_l), w_r(X_l), w_r(D\varphi(X_l)D^{-1})\}\geq(l+1)c_0.
\]
In fact, since $a_i\geq a_j$ as $i\leq j$, this amounts to solving a system of equations of the forms
\begin{equation}\label{eq:Frobenius-family}
c+x=\pi^{m}\varphi(x),\qquad c\in A\widehat{\otimes}_{\Q}\widetilde{\r}^{\bd,r}, m\leq0.
\end{equation}
By Lemma \ref{lem:Frobenius-equation-family}(2), (\ref{eq:Frobenius-family}) has a solution $x\in A\widehat{\otimes}_{\Q}\widetilde{\r}^{\bd,qr}$ with $w(x)\geq w(c)$ and $w_{r}(x)\geq\min\{w(c),w_r(c)\}$. Hence $w(\pi^m\varphi(x))\geq \min\{w(c),w(x)\}\geq w(c)$ and
\[
w_{qr}(x)=w_r(\varphi(x))\geq w_r(\pi^m\varphi(x))\geq \min\{w_r(c),w_r(x)\}\geq\min\{w(c),w_r(c)\}.
\]
This yields the claim. Put $U_{l+1}=U_l(I_n-X_l)$; then
\begin{equation*}
\begin{split}
w(U_{l+1}-I_n)\geq &c_0, \quad w_{qr}(U_{l+1}-U_l)=w_{qr}(X_l)\geq (l+1)c_0,\\
w(U_{l+1}-U_l)&=w(X_l)\geq (l+1)c_0.
\end{split}
\end{equation*}
We have
\[
U_{l+1}^{-1}F\varphi(U_{l+1})D^{-1}-I_n=(I_n+X_l+\cdots)(I_n+B_l+C_l)(I_n-D\varphi(X_l)D^{-1})-I_n.
\]
It follows that
\[
\min\{w(U_{l+1}^{-1}F\varphi(U_{l+1})D^{-1}-I_n),w_r(U_{l+1}^{-1}F\varphi(U_{l+1})D^{-1}-I_n)\}\geq c_0
\]
and
\[
\min\{w(U_{l+1}^{-1}F\varphi(U_{l+1})D^{-1}-I_n-B_l),w_r(U_{l+1}^{-1}F\varphi(U_{l+1})D^{-1}-I_n-B_l)\}\geq (l+1)c_0.
\]
This yields the inductive step. Then $U=\lim_{l\rightarrow\infty} U_{l}$ satisfies the desired property.
\end{proof}

%Let $N$ be a $\varphi$-module over $\bt^\dagger$. As mentioned earlier, de Jong proves that the reverse filtration of $M=N\otimes_{\bt^\dagger}\bt$ descends to $N$. For our purpose, we need a refinement of his result in the manner of Proposition \ref{prop:reverse-filtration}.

%For any successive quotient of the reverse filtration of $M$, its Dieudonn\'e-Manin decompositions descend to the corresponding successive quotient of the reverse filtration of $N$.
%\end{proposition}
%\begin{proof}
%It suffices to show that any eigenvector of the first step of the reverse filtration of $M$ descends to $N$; this follows from \cite[Lemma 5.4.1]{Ke05}.
%\end{proof}

\begin{proposition}\label{prop:family-coordinate}
Let $M_A$ be a family of $\varphi$-modules over $\r_{A_K}$ of rank $n$, and let $x\in M(A)$. Let $N_x$ be a model of $M_x$. Suppose that the generic slopes (counted with multiplicity) of $N_x$ are
\[
a_{1}/a\geq\cdots\geq a_{n}/a
\]
where $a$ is a positive integer and $a_i\in\mathbb{Z}$ for $1\leq i\leq n$. Then there exists a Weierstrass subdomain $M(B)$ containing $x$ such that $\widetilde{M}_B=M_A\otimes_{\r_{A_K}}(B\widehat{\otimes}_{\Q}\widetilde{\r})$ is finite free over $(B\widehat{\otimes}_{\Q}\widetilde{\r})$, and admits a basis under which the matrix of $\varphi^a$ is an upper triangular matrix $F$ over $ B\widehat{\otimes}_{\Q}\widetilde{\r}^{\bd}$ with diagonal entries $F_{ii}=\pi^{-a_{n+1-i}}$.
\end{proposition}
\begin{proof}
It suffices to treat the case that $a=1$
by replacing $\varphi$ with $\varphi^a$.  We claim that
$\widetilde{N}_x$
admits a filtration such that the $i$-th successive quotient is isomorphic to $V_{\pi^{-a_{n+1-i}},1}$. For the $\mathrm{RF}$ case, the claim follows from Propositions \ref{prop:slope-decomposition}, \ref{prop:reverse-filtration} and Lemma \ref{lem:V-n-1}. For the $\mathrm{AF}$ case, the claim follows from \cite[Theorem 14.6.3]{Ke07}, \cite[Theorem 6.3.3(b)]{Ke05} and Lemma \ref{lem:V-n-1}. It thus follows that $\widetilde{N}_x$ admits a basis $\mathrm{e}_x$ under which the matrix $F_x$ of $\varphi$ is upper triangular with diagonal entries $(F_x)_{ii}=\pi^{-a_{n+1-i}}$.

By Proposition \ref{prop:global-model}, after shrinking $M(A)$, we may suppose that $M_A$ admits a finite free model $N_A$ with a basis $\mathrm{e}_A$. Let $U_x$ be the square matrix satisfying $p_x(\e_A)U_x=\mathrm{e}_x$. Then $U_x$ is invertible over $\widetilde{\r}^{\bd,r}$ for some $r>0$. By Corollary \ref{cor:lifting}, after shrinking $M(A)$, we lift $U_x$ to an invertible matrix $U$ over $A\widehat{\otimes}_{\Q}\widetilde{\r}^{\bd,r}$. Let $F\in A\widehat{\otimes}_{\Q}\widetilde{\r}^{\bd,r'}$ be the matrix of $\varphi$ under the base $\mathrm{e}=\mathrm{e}_AU$, and let $D$ be a lift of $F_x$ over $\widetilde{\r}^{\bd,r'}$ such that $D$ is upper triangular and $D_{ii}=\pi^{a_i}$. Note that $(FD^{-1})(x)=I_n$. By Proposition \ref{prop:shrinking}, we choose a Weierstrass subdomain $M(B)$ containing $x$ such that $\min\{w(FD^{-1}-I_n),w_{r'}(FD^{-1}-I_n)\}>0$ over $B\widehat{\otimes}_{\Q}\widetilde{\r}^{\bd,r'}$. Then by Lemma \ref{lemma:good-basis-family}, there exists a matrix $V$ over $B\widehat{\otimes}_{\Q}\widetilde{\r}^{\bd,qr}$ such that $V^{-1}F\varphi(V)D^{-1}-I_n$ is upper triangular nilpotent. It follows that the basis $\e V$ satisfies the desired property.
\end{proof}

\begin{theorem}\label{thm:semicontinuity}
Let $M_A$ be a family of $\varphi$-modules over $\r_{A_K}$. Then for any $x\in
M(A)$, there is a Weierstrass subdomain $M(B)$ containing $x$ such that the HN-polygon of $M_y$ lies above the HN-polygon of
$M_x$ with the same endpoint for
any $y\in M(B)$.
\end{theorem}
\begin{proof}
By Proposition \ref{prop:good-model-point}, we choose a good model $N_x$ of $M_x$. We apply Proposition \ref{prop:family-coordinate} to $N_x$. It then follows from Corollary \ref{cor:additive} that for any $y\in M(B)$, the generic HN-polygon of $N_y$ is the same as the generic HN-polygon of
$N_x$; hence it is the same as the HN-polygon of $M_x$. We thus deduce that the HN-polygon of $M_y$ lies above the HN-polygon of $M_x$ with the same endpoint by Proposition \ref{prop:polygon-comparison}.
\end{proof}

\begin{definition}
Let $\widetilde{M}_A$ be a $\varphi$-module over $A\widehat{\otimes}_{\Q}\widetilde{\r}$. For $c,d\in\mathbb{Z}$ with $d>0$, a \emph{$(c,d)$-pure model} of $\widetilde{M}_A$ is a finite free $A\widehat{\otimes}_{\Q}\widetilde{\r}^\int$-submodule $\widetilde{N}_A$  with
$$\widetilde{N}_A\otimes_{A\widehat{\otimes}_{\Q}\widetilde{\r}^\int}(A\widehat{\otimes}_{\Q}\widetilde{\r})=\widetilde{M}_A$$
so that $\widetilde{N}_A\otimes_{A\widehat{\otimes}_{\Q}\widetilde{\r}^\int}(A\widehat{\otimes}_{\Q}\widetilde{\r}^\bd)$ is stable under $\varphi$ and the $\varphi$-action induces an isomorphism $\pi^{c}(\varphi^d)^*\widetilde{N}_A\cong \widetilde{N}_A$. For $s\in\mathbb{Q}$, we say that $\widetilde{M}_A$ is \emph{pure of slope $s$} if $\widetilde{M}_A$ admits a $(c,d)$-pure model for some (hence any) $c,d\in\mathbb{Z}$ with $d>0$ and $s=c/d$. If $s=0$, we also say that $\widetilde{M}_A$ is \emph{\'etale}, and a $(0,1)$-pure model is also called an \emph{\'etale model}. By a \emph{slope filtration} of  $\widetilde{M}_A$ we mean a finite filtration
of $\varphi$-submodules of $\widetilde{M}_A$ such that the successive quotients are pure $\varphi$-modules over $A\widehat{\otimes}_{\Q}\widetilde{\r}$ with decreasing slopes.
\end{definition}

\begin{proposition}\label{prop:slope-filtration-unique}
Any $\varphi$-module over $A\widehat{\otimes}_{\Q}\widetilde{\r}$ admits at most one slope filtration.
\end{proposition}

\begin{proof}
It suffices to show that for any two pure $\varphi$-modules $\widetilde{M}_A$ and $\widetilde{M}_A'$ over $A\widehat{\otimes}_{\Q}\widetilde{\r}$, if the slope of $\widetilde{M}_A$ is bigger than the slope of $\widetilde{M}_A'$, then there is no nontrivial morphism from $\widetilde{M}_A$ to $\widetilde{M}_A'$. For this, by replacing $\varphi$ with a suitable powers of it, and by identifying $\mathrm{Hom}(\widetilde{M}_A, \widetilde{M}_A')$ with $(\widetilde{M}_A^{\vee}\otimes\widetilde{M}_A')^{\varphi=1}$, it suffices to show that any \'etale $\varphi$-module over $A\widehat{\otimes}_{\Q}\widetilde{\r}$ does not have rank 1 pure $\varphi$-submodule with positive integral slope. If the contrary is true, then there exists a nonzero column vector $\textbf{v}$ over $A\widehat{\otimes}_{\Q}\widetilde{\r}^r$, an invertible matrix $W$ over $A\widehat{\otimes}_{\Q}\widetilde{\r}^{\int,r}$ (in the $\mathrm{AF}$ case, set $\widetilde{\r}^{\int,r}$ to be $\at^{(0,r]}$) and some negative integer $n$ such that $W\varphi(\textbf{v})=\pi^n\textbf{v}$. By Proposition \ref{prop:limit-generalize} and Lemma \ref{lem:comparison}, we may suppose $w_s(W)>n$ for any $0<s\leq r$ by shrinking $r$.  It therefore follows that
\[
w_{s/q}(\textbf{v})=w_{s/q}(W\varphi(\textbf{v}))-n\geq w_{s/q}(W)+w_{s}(\textbf{v})-n>w_s(\textbf{v})
\]
for any $0<s\leq r$. This implies $w_{s/q^n}(\textbf{v})>w_{s}(\textbf{v})$ for any $n\geq 1$. We claim that $\textbf{v}$ is over $A\widehat{\otimes}_{\Q}\widetilde{\r}^\bd$. In fact, if $a_i$ the $i$-th coefficient of some entry of $\textbf{v}$, we then have
\[
v(a_i)+\frac{is}{q^n}\geq w_{s/q}(\textbf{v})>w_s(\textbf{v})
\]
for any $n\geq 1$. Hence $v(a_i)>w_s(\textbf{v})$; thus $v$ is over $A\widehat{\otimes}_{\Q}\widetilde{\r}^\bd$ and satisfies $w(\textbf{v})>w_s(\textbf{v})$. It then follows that $w(\textbf{v})=w(\varphi(\textbf{v})W)-n> w(\varphi(\textbf{v}))+w(W)\geq w(\textbf{v})$, yielding a contradiction.
\end{proof}

\begin{lemma}\label{lem:maximal-slope}
Let $L'$ be an extension of $K$ with strongly difference-closed residue field. Let $M$ be a $\varphi$-module over $\widetilde{\r}_{L'}$. Suppose that the maximal slope $m$ of $M$ is integral. Then for any $s\in\mathbb{Q}$, $m\geq s$ if and only if $M$ admits a nonzero eigenvector of $\varphi$ with eigenvalue $\pi^{[-s]}$.
\end{lemma}
\begin{proof}
If $\varphi(v)=\pi^{[-s]}v$ for some nonzero $v\in M$, it follows that $m\geq -[-s]\geq s$. Conversely, suppose $m\geq s$. Using Theorem \ref{thm:Robba-decomposition}, let $V_{\lambda,d}$ with $\lambda\in L'$ be in a Dieudonn\'e-Manin decomposition of the first step of the HN filtration of $M$. Since $v(\lambda)/d=-m$ is an integer, by Lemma \ref{lem:V-n-1}, we deduce that $V_{\lambda,d}$ admits a nonzero $\varphi$-eigenvector $v$ with eigenvalue $\pi^{-m}$. If $-m=[-s]$, then we are done. Otherwise, let $n=[-s]+m$, and put
\[
f=\sum_{i\in\mathbb{Z}}(\pi^{-n})^{\{i\}}u^{q^i}.
\]
Since $n>0$, it is clear that $f$ is a well-defined element of $\widetilde{\r}_{L'}$, and satisfies $\varphi(f)=\pi^nf$. It follows that $fv$ is a nonzero $\varphi$-eigenvector of $M$ with eigenvalue $\pi^{-m+n}=\pi^{[-s]}$.
\end{proof}

\begin{remark}
Let $M_A$ be a family of $\varphi$-modules over $\r_{A_K}$. It is clear that if  $\widetilde{M}_A=M_A\otimes_{\r_{A_K}}(A\widehat{\otimes}_{\Q}\widetilde{\r})$ is a $\varphi$-module over $A\widehat{\otimes}_{\Q}\widetilde{\r}$ admitting a slope filtration, the HN-polygons of $M_x$ are constant over $M(A)$.  In $\S2.4$, we will construct a family of $\varphi$-modules so that its HN-polygons are not even locally constant over the base. Thus this family does not admit a slope filtration over the extended Robba ring.
\end{remark}

\begin{theorem}\label{thm:global-filtration}
Let $M_A$ be a family of $\varphi$-modules over $\r_{A_K}$, and let $x\in M(A)$. Then there exists a Weierstrass subdomain $M(B)$ containing $x$ such that
the set of $y\in M(B)$ where the HN-polygon of $M_y$ coincides with the HN-polygon of $M_x$
forms a Zariski closed subset $M(C)$ of $M(B)$, and
\[
\widetilde{M}_{C}=M_A\otimes_{\r_{A_K}}(C\widehat{\otimes}_{\Q}\widetilde{\r})
\]
admits a unique slope filtration which lifts the HN-filtration of $\widetilde{M}_x$.
\end{theorem}
\begin{proof}
Let $a'$ be the least common multiple of the denominators of the slopes of $M_x$. Let $a=a'(2^n)!$ where $n$ is the rank of $M$. We view $M$ as a $\varphi^a$-module. Suppose that the slopes of $M_x$ are $s_1>\cdots>s_l$ where each $s_j$ has multiplicity $d_j$. Then all $s_j$ are integers. By Proposition \ref{prop:family-coordinate}, there exists a Weierstrass subdomain $M(B)$ containing $x$ and a basis $\widetilde{\mathrm{e}}_B$ of $\widetilde{M}_B=M_A\otimes_{\r_{A_K}}(B\widehat{\otimes}_{\Q}\widetilde{\r})$ such that the matrix $F$ of $\varphi^a$ under $\widetilde{\mathrm{e}}_B$ is $n\times n$ upper triangular over $B\widehat{\otimes}_{\Q}\widetilde{\r}^{\bd,r}$ for some $r>0$ with diagonal entries $F_{ii}=\pi^{m_i}$ where $m_i=-s_{j}$ if $d_l+\cdots+d_{j+1}<i\leq d_l+\cdots+d_{j}$.

As explained in the proof of Theorem \ref{thm:semicontinuity}, the HN-polygon of $M_y$ lies above the HN-polygon of $M_x$ for any $y\in M(B)$. Therefore the HN-polygon of $M_y$ coincides with the HN-polygon of $M_x$ if and only if the former lies below the latter. Since HN-polygons are convex, by Proposition \ref{prop:wedge}, we deduce that the HN-polygon of $M_y$ lies above the HN-polygon of $M_x$ if and only if the maximal slope of $\wedge^{d_1+\cdots+d_j}M_y$ is no less than $\sum_{i=1}^js_id_i$ for each $1\leq j\leq l$. By Proposition \ref{prop:HN-polygon-multifield-extended}(2), $M_y$ and $\widetilde{M}_y$ have the same HN-polygons.  Let $n_j$ be the rank of $\wedge^{d_1+\cdots+d_j}\widetilde{M}_B$; then $n_j\leq 2^n$. By the construction of $a$, we see that all the slopes of $\wedge^{d_1+\cdots+d_j}\widetilde{M}_y$ are integral. Hence by Proposition \ref{prop:HN-polygon-multifield-extended}(1), Lemma \ref{lem:maximal-slope} (for the $\mathrm{RF}$ case) and \cite[Proposition 3.3.2]{Ke05} (for the $\mathrm{AF}$ case), we conclude that the HN-polygon of $M_y$ lies below the HN-polygon of $M_x$ if and only if each component of $\wedge^{d_1+\cdots+d_j}\widetilde{M}_y$ admits a nonzero $\varphi^a$-eigenvector with eigenvalue $\pi^{-s_1d_1-\cdots-s_jd_j}$ for $1\leq j\leq l$. Let $S_j$ be the set of $y\in M(B)$ which satisfies this condition for $j$.
Therefore, to prove the first part of the theorem, it suffices to show that each $S_j$ is a Zariski closed subset of  $M(B)$.

Note that under the basis $\wedge^{d_1+\cdots+d_j}\widetilde{\e}_B$ of $\wedge^{d_1+\cdots+d_j}\widetilde{M}_B$,
the matrix $G$ for $\varphi^a$ is upper triangular, and its diagonal entries are of the forms $\pi^{-m}$ where $m$ goes through all the sums of $d_1+\cdots+d_j$ elements of the slope multiset of $M_x$. In particular, $G_{n_j,n_j}=\pi^{-s_1d_1-\cdots-s_jd_j}$ is the smallest power of $\pi$ among diagonal entries. Using the basis $\wedge^{d_1+\cdots+d_j}\widetilde{\e}_B$ and matrix $G$, a short computation shows that finding $\varphi^a$-eigenvectors with eigenvalues $\pi^{-s_1d_1-\cdots-s_jd_j}$ amounts to solving a series of equations $\varphi(\beta_i)-h_i\beta_i=\alpha_i$ for $1\leq i\leq n_j$, where
\[
h_i=\pi^{-s_1d_1-\cdots-s_jd_j}/G_{n_j+1-i,n_j+1-i},
\]
and
\[
\alpha_i=-G_{n_j+1-i,n_j+1-i}^{-1}(G_{n_j+1-i,n_j}\varphi(\beta_1)+\cdots G_{n_j+1-i,n_j+2-i}\varphi(\beta_{i-1})).
\]
Note that $h_1=1, \alpha_1=0$, and that $h_i$ is a negative power of $\pi$ for each $i\geq 2$. Let $L'=L^{\varphi=1}$; it follows that $\beta_1\in(B\widehat{\otimes}_{\Q}\widetilde{\r}^\bd)^{\varphi}=B_{L'}$. Furthermore, by Lemma \ref{lem:Frobenius-equation-family}(3), for any initial value $\beta_1\in B_{L'}$, this series of equation admits at most one solution.
By its definition, we see that $S_j$ is just the set $y\in M(B)$ at which the specialization of this series of equations admits a solution for
the initial value $1$ (hence for any initial value) in $k(y)\otimes_{\Q}L'$.
Applying Lemmas \ref{lem:Frobenius-equation-family} and \ref{lem:Zariski-closed-obstruction} inductively, we thus deduce that $S_j$ is a Zariski closed subset of $M(B)$, yielding the first part of the Theorem.

To show the rest of the theorem, we may suppose that $B=C$. Using the basis $\widetilde{\e}_B$ and matrix $F$, a short computation shows that finding $\varphi^a$-eigenvectors with eigenvalues $\pi^{-s_1}$ in $\widetilde{M}_B$ amounts to solving a series of equations $\varphi(\beta_i)-h_i\beta_i=\alpha_i$ for $1\leq i\leq n$, where
\[
h_i=\pi^{-s_1}/F_{n+1-i,n+1-i},
\]
and
\[
\alpha_i=-F_{n+1-i,n+1-i}^{-1}(F_{n+1-i,n}\varphi(\beta_1)+\cdots F_{n+1-i,n+2-i}\varphi(\beta_{i-1})).
\]
Note that $h_1=\cdots=h_{d_1}=1,\alpha_1=0$, and $h_{i}$ is a negative power of $\pi$ for each $d_1+1\leq i\leq n$. Note that for $\alpha\in B\widehat{\otimes}_{\Q}\widetilde{\r}^{\bd}$, any two solutions of the equation $\varphi(\beta)-\beta=\alpha$ differs by an element of $B_{L'}$. We therefore deduce from Lemma \ref{lem:Frobenius-equation-family}(2) that $\beta_1,\dots,\beta_{d_1}$ are of the forms
\begin{eqnarray*}
&\beta_1&=c_1\\
&\beta_2&=c_1\beta_{21}+c_2\\
&\vdots&\\
&\beta_{d_1}&=c_1\beta_{d_11}+c_2\beta_{d_12}+\cdots+c_{d_1}
\end{eqnarray*}
where $\beta_{ij}$ are fixed elements of $B\widehat{\otimes}_{\Q}\widetilde{\r}^{\bd}$ determined by the coefficients of $F_{n-c,n-d}$ for $0\leq c\leq d_1-1$ and $0\leq d\leq d_1-2$ and $c_1,\dots,c_{d_1}\in B_{L'}$.
Furthermore, by Lemma \ref{lem:Frobenius-equation-family}, we see that for any initial values $c_1,\dots,c_d$, both this series of equations and its specialization at any $y\in M(B)$ admit at most one solution. On the other hand, by Theorem \ref{thm:Robba-decomposition} (for the $\mathrm{RF}$ case), \cite[Proposition 4.2.5]{Ke05} (for the $\mathrm{AF}$ case) and Lemma \ref{lem:V-n-1}, we see that the set of $\varphi^a$-eigenvectors with eigenvalue $\pi^{-s_1}$ in $\widetilde{M}_y$ is a free $k(y)\otimes_{\Q}L'$-module of rank $d_1$.
We therefore deduce that the specialization of this series of equations at $y$ admits a unique solution in $k(y)\otimes_{\Q}\widetilde{\r}^\bd$ with the initial values $c_1(y),\dots,c_{d_1}(y)$. Hence by Lemmas \ref{lem:Frobenius-equation-family} and \ref{lem:Zariski-closed-obstruction}, we conclude that this series of equations admits a unique solution in $B\widehat{\otimes}_{\Q}\widetilde{\r}^\bd$ for any initial values $c_1,\dots, c_d$.
As a consequence, the set of $\varphi^a$-eigenvectors with eigenvalue $\pi^{-s_1}$ is a free $B_{L'}$-module generated by  $\textbf{v}_i=(0,\cdots,0,1,\cdots)^{t}$ where the first $1$ lies in the $i$-th entry for $1\leq i\leq d_1$. Set $\widetilde{M}^1_{B}$ as the $\varphi^a$-submodule of $\widetilde{M}_B$ generated by $\textbf{v}_i$ for all $1\leq i\leq d_1$.

It is clear that $\widetilde{M}^1_{B}$ is free of rank $d_1$ and $\widetilde{M}_B/\widetilde{M}^1_B$ is free of rank $n-d_1$. Furthermore, $\widetilde{M}^1_{B}$ is pure of slope $s_1$ as a $\varphi^a$-module. Thus by induction on $\widetilde{M}_B/\widetilde{M}^1_B$, we deduce that as a $\varphi^a$-module, $\widetilde{M}_B$ admits a slope filtration
\begin{equation}\label{eq:HN-filtration}
0\subset\widetilde{M}_{B}^1\subset\cdots\subset\widetilde{M}_B^l=\widetilde{M}_B.
\end{equation}
 Note that the $\varphi$-pullback of (\ref{eq:HN-filtration}) is also the slope filtration of $\widetilde{M}_B$. Hence (\ref{eq:HN-filtration}) is stable under $\varphi^*$ by Proposition \ref{prop:slope-filtration-unique},  yielding that each $\widetilde{M}_B^j$ is also a $\varphi$-submodule of $\widetilde{M}_B$. It remains to show that the successive quotients are pure as $\varphi$-modules. By induction, we only need to show this for $\widetilde{M}_B^1$. Since $k_L$ is strongly difference-closed, by \cite[Lemma 14.3.3]{Ke07}, we choose some $\lambda\in L^\times$ such that $\varphi^a(\lambda)/\lambda=\varphi(\pi^{s_i})/\pi^{s_i}$. A short computation shows that if $\textbf{v}$ is a $\varphi^a$-eigenvector with eigenvalues $\pi^{-s_i}$, so is $\lambda\varphi(\textbf{v})$. Hence the free $B_{L'}$-module generated by $\{\textbf{v}_i\}_{1\leq i\leq d_1}$ is stable under $\varphi$. This implies that the finite free good model of $\widetilde{M}_B^1$ generated by $\{\textbf{v}_i\}_{1\leq i\leq d_1}$ is stable under $\varphi$, yielding that $\widetilde{M}_B^1$ is pure of slope $s_1$ as a $\varphi$-module. We therefore conclude that (\ref{eq:HN-filtration}) is the slope filtration of $\widetilde{M}_B$ as a $\varphi$-module over $B\widehat{\otimes}_{\Q}\widetilde{\r}$.

\end{proof}
\subsection{An example}
In this subsection we construct a family of $\varphi$-modules with nonconstant HN-polygons. Our construction is inspired by the computation of $\mathrm{H}^1$ of rank 1 $(\varphi,\Gamma)$-modules in \cite{C07}.

Let $\Gamma=\mathrm{Gal}(\Q(\mu_{p^\infty})/\Q)$. We set the $\varphi, \Gamma$-actions on $\r_K$ as
\[
\gamma(f(T))=f((1+T)^{\chi(\gamma)}-1), \quad \varphi(f(T))=f((1+T)^p-1),\qquad  \gamma\in\Gamma, f(T)\in\r_K,
\]
where $\chi$ is the $p$-adic cyclotomic character. Set the ``$p$-adic $2\pi i$'' $t=\log(1+T)$ which satisfies $\varphi(t)=pt$ and $\gamma(t)=\chi(\gamma)t$.

\begin{definition}
By a $(\varphi,\Gamma)$-module over $\r_K$ we mean a $\varphi$-module over $\r_K$ equipped with a continuous semilinear $\Gamma$-action which commutes with the $\varphi$-action.
\end{definition}

\begin{definition}
For any continuous character $\delta:\Q^\times\rightarrow K^\times$, define the rank $1$ $(\varphi,\Gamma)$-module \emph{$\r_K(\delta)$} by setting the $\varphi,\Gamma$-action as
\begin{equation}\label{eq:twist-def}
\varphi(av)=\delta(p)\varphi(a)v, \quad \gamma(av)=\delta(\chi(\gamma))\gamma(a)v, \qquad a\in R_K
\end{equation}
for some $\r_K$-basis $v$. For $\delta(x)=x^{-n}$ for some $n\in\mathbb{Z}$, we denote $\r_K(\delta)$ by $\r_K(n)$. For any $(\varphi,\Gamma)$-module $M$ over $\r_K$, set the $(\varphi,\Gamma)$-module $M(\delta)=M\otimes_{\r_K}\r_K(\delta)$.
\end{definition}
\begin{remark}
We set $\r_K(n)$ as $\r_K(x^{-n})$ in order to match our sign convention of slopes. Beware that our sign convention is opposite to the one used in \cite{C07}.
\end{remark}

If $p>2$, then $\Gamma$ is topologically procyclic. We fix a topological generator $\gamma$ of $\Gamma$.
\begin{definition}
Suppose $p>2$. For a $(\varphi,\Gamma)$-module $M$ over $\r_K$, set the complex $C^{\bullet}_{\varphi,\gamma}(M)$ as
 \[
  0\stackrel{}{\longrightarrow}M\stackrel{d_{1}}{\longrightarrow}M\oplus M
 \stackrel{d_{2}}{\longrightarrow}M\stackrel{}{\longrightarrow}0
\]
with $d_1(x) = ((\gamma - 1)x, (\varphi - 1)x)$ and $d_2(x,y) =
(\varphi - 1)x - (\gamma - 1)y$. Let $\mathrm{H}^{\bullet}(M)$ denote
cohomology groups of this complex.
\end{definition}
It is straightforward to see that $\mathrm{H}^1(M)$ classifies the extension of the trivial $(\varphi,\Gamma)$-module $\r_K$ by $M$.
\begin{remark}
If $p=2$, $\Gamma$ is no longer topologically procyclic, we need to modify the definition of $C^{\bullet}_{\varphi,\gamma}(M)$. See \cite[\S 2.1]{L08} for more details.
\end{remark}

\begin{lemma}\label{lem:Gamma-stable}
For a $(\varphi,\Gamma)$-module $M$ over $\r_K$, the slope filtration of $M$ is a filtration of $(\varphi,\Gamma)$-submodules.
\end{lemma}
\begin{proof}
Since the slope filtration of $M$ is unique and $\varphi,\Gamma$-actions commute, it is stable under the $\Gamma$-action. This yields the lemma.
\end{proof}

\begin{lemma}\label{lem:etale-split}
If $M$ is a $(\varphi,\Gamma)$-module over $\r_K$ satisfying the exact sequence
\[
0\stackrel{}{\longrightarrow}\r_K(-1)\stackrel{}{\longrightarrow}
M\stackrel{}{\longrightarrow}\r_K(1)\stackrel{}{\longrightarrow}0.
\]
then $M$ is \'etale if and only if the
exact sequence is non-split.
\end{lemma}
\begin{proof}
The ``only if'' part is trivial. Suppose that the exact sequence is non-split. We
first have $\deg(M)=0$. If $M$ is not \'etale, by Lemma \ref{lem:Gamma-stable}, it has a
rank 1 $(\varphi,\Gamma)$-submodule $N$ with positive slope. Since $\r_K(-1)$ is a saturated
$(\varphi,\Gamma)$-submodule of $M$ with negative slope, we have $N\cap\r_K(-1)=0$
 by Corollary \ref{cor:stable}. This implies that $N$ maps isomorphically to
a $(\varphi,\Gamma)$-submodule of $\r_K(1)$. The $(\varphi,\Gamma)$-submodules of $\r_K(1)$ are of the forms $t^k\r_K(1)$
for $k\in\mathbb{N}$ (see \cite[Lemme 3.2, Remarque 3.3]{C07} for a proof (the proof works for our general $K$)). Hence $N\cong t^k\r_K(1)$ for some $k\geq 1$ as the exact sequence is non-split. It therefore follows that $\deg(N)\leq0$, yielding a contradiction.
\end{proof}

\begin{lemma}\label{lem:injective-cohomology}
Suppose $p>2$. Then the natural map $\mathrm{H}^1(\r_{\mathbb{Q}_p}(\delta))\ra \mathrm{H}^1(\r_L(\delta))$ is injective for any continuous character $\delta:\Q^\times\rightarrow \Q^\times$ such that $\delta(p)\neq p^{-n}$ for any $n\in\mathbb{N}$.
\end{lemma}
\begin{proof}
Suppose $(a,b)$ represents an element in the kernel of this map; then
\[
a=(\delta(\chi(\gamma))\gamma-1)f, \quad b=(\delta(p)\varphi-1)f
\]
for some $f=\sum_{i\in\mathbb{Z}}a_iT^i\in\r_L$. We claim that $f$ belongs to $\r_{\Q}$. We proceed by induction on $j$ to show that $a_{-j},a_{j}\in\Q$ for each $j\in \mathbb{N}$.
The constant term of $(\delta(p)\varphi-1)f$ is $(\delta(p)-1)a_0$. Since $\delta(p)\neq1$, we get $a_0\in\mathbb{Q}_p$. Now suppose $a_{1-j},\dots, a_0,\dots,a_{j-1}\in \mathbb{Q}_p$ for some $j\geq1$. Note that the coefficient of $T^j$ in $(\delta(p)\varphi-1)f$ is the sum of $(\delta(p)p^{j}-1)a_j$ and a $\Q$-linear combination of $a_1,\dots,a_{j-1}$. Since $\delta(p)p^{j}-1\neq0$, we get $a_j\in\Q$. Note that for any positive integer $k$, we have
\[
\varphi(\frac{1}{T^k})=\frac{1}{((1+T)^p-1)^k}=\frac{1}{T^{pk}}\cdot\frac{1}{((1+\frac{1}{T})^p-\frac{1}{T^p})^k}
=\frac{1}{T^{pk}}(1-\frac{pk}{T}-\cdots).
\]
Hence the coefficient of $T^{-pj}$ in $(\delta(p)\varphi-1)f$ is the sum of $(\delta(p)-1)a_{-j}$ and a linear combination of  $a_{-1},\dots,a_{-j+1}$, yielding $a_{-j}\in\Q$. Hence $f\in \r_{\Q}$. This yields $(a,b)=0$ in $\mathrm{H}^1(\r_{\Q}(\delta))$.
\end{proof}

\begin{example}
Let $A=\Q\langle x\rangle$. By \cite[Theorem 2.9]{C07}, $\mathrm{H}^1(\r_{\Q}(-2))$ is a one dimensional $\Q$-vector space. Choose a representative $(a,b)$ of a nonzero element of $\mathrm{H}^1(\r_{\Q}(-2))$. Let $M_A$ be the rank 2
$\varphi$-module over $\r_{A_K}=\r_{K\langle x\rangle}$ such that the $\varphi$-action
is defined by the matrix
\[
\left(
\begin{matrix}
 p^2&xb \\
 0&1
\end{matrix}
\right)
\]
for some basis. Using the same basis, we equip a $\Gamma$-action on $M_A$ by the matrix
\[
\left(\begin{matrix}
 \chi^2(\gamma)&xa \\
 0&1
\end{matrix}
\right)
\]
for any $\gamma\in\Gamma$. It is clear that at each
$y\in M(A)$, $M_y$ is an extension of $k(y)\otimes_{\Q}\r_{K}$ by
$k(y)\otimes_{\Q}\r_{K}(-2)$ defined by $(x(y)a,x(y)b)$.
It follows from Lemmas \ref{lem:injective-cohomology} and \ref{lem:etale-split} that $M_y(1)$ is \'etale if and only if $y$ is not the origin. Hence the HN-polygons of $M_y$ are not locally constant around the origin.
\end{example}

\end{document}